\documentclass[a4paper]{article}  
\usepackage[margin=1in]{geometry}

\usepackage{amsmath}
\usepackage{amsfonts}
\usepackage{amsthm}
\usepackage{amssymb}
\usepackage{mathrsfs}

\usepackage{graphicx}
\usepackage{subfigure}
\usepackage{comment}

\usepackage{tikz}
\usetikzlibrary{arrows.meta,calc,decorations.pathreplacing}

\usepackage{listings}

\usepackage{hyperref}
\hypersetup{
    colorlinks = true,
    bookmarksopen = true,
    bookmarksnumbered = true,
    linkcolor = blue,
    anchorcolor = blue,
    citecolor = blue,
}
\usepackage{aliascnt}
\usepackage[capitalise, nameinlink]{cleveref}

\newtheorem{theorem}{Theorem}[section]

\newaliascnt{definition}{theorem}
\newtheorem{definition}[definition]{Definition}
\aliascntresetthe{definition}

\newaliascnt{assumption}{theorem}

\aliascntresetthe{assumption}

\newaliascnt{proposition}{theorem}

\aliascntresetthe{proposition}

\newaliascnt{lemma}{theorem}
\newtheorem{lemma}[lemma]{Lemma}
\aliascntresetthe{lemma}

\newaliascnt{corollary}{theorem}
\newtheorem{corollary}[corollary]{Corollary}
\aliascntresetthe{corollary}

\newaliascnt{remark}{theorem}
\newtheorem{remark}[remark]{Remark}
\aliascntresetthe{remark}

\crefname{theorem}{Theorem}{Theorems}
\Crefname{theorem}{Theorem}{Theorems}
\crefname{definition}{Definition}{Definitions}
\Crefname{definition}{Definition}{Definitions}
\crefname{assumption}{Assumption}{Assumptions}
\Crefname{assumption}{Assumption}{Assumptions}
\crefname{proposition}{Proposition}{Propositions}
\Crefname{proposition}{Proposition}{Propositions}
\crefname{lemma}{Lemma}{Lemmas}
\Crefname{lemma}{Lemma}{Lemmas}
\crefname{corollary}{Corollary}{Corollaries}
\Crefname{corollary}{Corollary}{Corollaries}
\crefname{remark}{Remark}{Remarks}
\Crefname{remark}{Remark}{Remarks}

\newcommand{\mi}{\imath}

\newcommand{\vphi}{\varphi}

\newcommand{\tphi}{\tilde{\phi}}

\newcommand{\bi}{\boldsymbol{i}}

\newcommand{\bk}{\boldsymbol{k}}
\newcommand{\bn}{\boldsymbol{n}}
\newcommand{\bx}{\boldsymbol{x}}
\newcommand{\by}{\boldsymbol{y}}
\newcommand{\bv}{\boldsymbol{v}}
\newcommand{\be}{\boldsymbol{e}}
\newcommand{\bt}{\boldsymbol{t}}
\newcommand{\balpha}{\boldsymbol{\alpha}}
\newcommand{\bbeta}{\boldsymbol{\beta}}
\newcommand{\bgamma}{\boldsymbol{\gamma}}
\newcommand{\bxi}{\boldsymbol{\xi}}
\newcommand{\bdelta}{\boldsymbol{\delta}}
\newcommand{\bzero}{\boldsymbol{0}}

\newcommand{\re}{\mathrm{e}}

\newcommand{\rd}{\mathrm{d}}

\newcommand{\bbR}{\mathbb{R}}
\newcommand{\bbZ}{\mathbb{Z}}
\newcommand{\bbN}{\mathbb{N}}
\newcommand{\bbC}{\mathbb{C}}
\newcommand{\bbS}{\mathbb{S}}

\newcommand{\calX}{\mathcal{X}}

\newcommand{\calG}{\mathcal{G}}
\newcommand{\calC}{\mathcal{C}}

\newcommand{\hf}{\widehat{f}}

\newcommand{\rep}{\mathrm{Re}}  

\newcommand{\ccinf}{C_c^{\infty}}  
\newcommand{\supp}{\operatorname{supp}}

\newcommand{\onenorm}[1]{| #1 |_1} 
\newcommand{\infnorm}[1]{| #1 |_\infty} 

\newcommand{\aco}[2]{\mathcal{A}_{#1}^{#2}}  
\newcommand{\fto}[1]{\mathcal{F}\left[#1\right]}  

\newcommand{\Vpara}{V_{\text{para}}}  

\newcommand{\lmin}{\lambda_{\min}}
\newcommand{\lmax}{\lambda_{\max}}

\begin{document}

\title{SinCoTrap: A High-Order Locally Corrected Trapezoidal Rule for Periodic Singular Integrals in Arbitrary Dimensions}
\date{}

\author{Hengzhun Chen\thanks{School of Mathematical Sciences, Fudan
    University (\href{mailto:hengzhunchen21@m.fudan.edu.cn}{\texttt{hengzhunchen21@m.fudan.edu.cn}}).} 
\and Yingzhou Li\thanks{School of Mathematical Sciences, Shanghai Key
    Laboratory for Contemporary Applied Mathematics, Fudan University and
    Key Laboratory of Computational Physical Sciences, Ministry of Education
    (\href{mailto:yingzhouli@fudan.edu.cn}{\texttt{yingzhouli@fudan.edu.cn}}).} }

\maketitle

\begin{abstract}

We present \emph{SinCoTrap} (Singularity-Corrected Trapezoidal Rule), a
high-order locally corrected trapezoidal method for periodic singular
integrals in arbitrary dimension $d$ with kernel $|\bx|^{-s}$, $0<s<d$.
The scheme preserves the uniform tensor grid and modifies only a fixed,
small stencil of weights near the singularity. For a correction order
$p$, the resulting quadrature attains the error rate $O(h^{2p+2+d-s})$.
We derive explicit, mesh-independent limiting correction weights via
analytic continuation of a special generalization of the Riemann zeta
function, yielding rapidly computable formulas that can be pretabulated
for each $(d,s,p)$. This makes SinCoTrap both efficient in application
and robust for high-order accuracy across a broad class of periodic
singular integrals.

\end{abstract}

\section{Introduction}

Singular integrals frequently arise in numerical analysis and scientific
computing in many applications. Accurate evaluation of these integrals
is essential for reliable computation of physical quantities, but the
presence of singularities poses significant challenges for standard
quadrature methods. Specifically, we focus on singular integrals in
arbitrary dimension $d$ that takes the form
\begin{equation*}
    I[\vphi \cdot \sigma] = \int_V \vphi(\bx) \sigma(\bx) \rd \bx, \quad 
    \sigma(\bx) = \frac{1}{|\bx - \bx_0|^{s}}, \quad \bx \in \bbR^d,
\end{equation*}
where $\sigma(\bx)$ is a singular function with an isolated singularity
at $\bx_0 \in \mathrm{Int}(V)$ for $0 < s < d$, $\vphi(\bx)$ is a smooth
function, and the entire integrand $\vphi(\bx)\sigma(\bx)$ is smooth
except at the singular points and is periodic with unit cell $V$.
Equivalently, $\vphi\cdot\sigma$ defines a $C^{\infty}$ function on the
punctured torus.

Various approaches have been developed for singular quadrature to
address the challenges posed by singularities, which can be broadly
classified into three categories: (1) \textit{Singularity cancellation}
\cite{duffy1982quadrature}, where a change of variables (such as to
spherical coordinates) cancels the singularity via the Jacobian,
allowing standard quadrature rules to be applied to the resulting
regular integral; (2) \textit{Singularity subtraction}
\cite{helsing2013higher,merkel1986integral}, where the singular
component is subtracted from the integrand and integrated separately by
introducing an auxiliary function that closely approximates the singular
behavior, while the regular part is handled numerically; and (3)
\textit{Singularity correction}
\cite{kapur1997high,marin2014corrected,rokhlin1990end}, where the
quadrature weights are modified to account for the effect of the
singularity and achieve high-order accuracy.

Our focus in this work is on the third category, specifically on
singularity correction methods for trapezoidal rule. The trapezoidal
rule is a simple, stable, and widely used quadrature method, celebrated
for its super-algebraic convergence for smooth periodic functions
\cite{trefethen2014exponentially}. Its accuracy deteriorates when either
periodicity or smoothness is lost: a smooth non-periodic integrand
creates boundary error, which can be corrected by endpoint or
near-boundary modifications, whereas a periodic integrand with an
isolated singularity, as in the present work, creates a local
singularity error. In both cases, one modifies only a small number of
quadrature weights while retaining the underlying uniform grid. The
resulting correction weights can be precomputed for a given boundary
behavior or singularity type and reused for a broad class of problems.

The pioneering corrected trapezoidal rule was introduced by
Rokhlin~\cite{rokhlin1990end} for one-dimensional endpoint
singularities. The central concept of the method involves adding local
correction terms to analytically cancel the leading-order terms of the
quadrature error of the trapezoidal rule. The correction weights are
computed by solving a linear system, constructed by requiring that the
corrected rule integrate exactly over the monomials, up to a desired
degree, multiplying the singular function. However, the practical
utility of this approach is limited because achieving high-order
accuracy requires correction weights whose magnitude grows rapidly,
leading to numerical instability. 

Subsequent work by Kapur and Rokhlin~\cite{kapur1997high} extended this
approach by introducing additional quadrature nodes outside the
integration interval. This strategy yielded smaller correction weights
and enabled higher-order accuracy. Furthermore, they derived a
mesh-size-independent rule by computing the asymptotic limit of the
weights as $h\rightarrow 0$, enhancing its practical utility, though
explicit closed-form expressions for these limiting weights were not
provided. The principal drawbacks of this scheme are its requirement for
function evaluations outside the domain, which can be problematic or
undefined, and the difficulties involved in computing the asymptotic
weights. 

To address these issues, Marin et al. \cite{marin2014corrected}
introduced an alternative approach using a compactly supported smooth
window function $g(\bx)$. By multiplying the singular function by $g(\bx)$,
the singularity is localized and the effect of boundary errors are
annihilated, thereby simplifying the computation of the correction
weights. For one-dimensional singular kernels of the form
$\sigma(\bx)=|\bx|^s$ with $s > -1$, they derived an explicit formula for
the converged weights, revealing a connection to the Riemann zeta
function. The method was also extended to two dimensions for the
singular kernel $\sigma(\bx)=1/|\bx|$, although a closed-form expression
for the converged weights remained elusive. Later work by Jiang and Li
\cite{jiang2022arbitrarily} improved the two-dimensional scheme by
proving the non-singularity of the underlying linear system for solving
the correction weights and employing Richardson extrapolation to compute
the weights numerically. Besides, Wu and Martinsson
\cite{wu2021corrected} found an implicit formula for the two-dimensional
weights, connecting them to the higher-order parametric derivatives of
the Epstein zeta function.

In this work, motivated by the need to accelerate the convergence of
thermodynamic limit of exchange energy in electronic structure
calculations \cite{chen2026high,xing2024unified}
without increasing computational cost, we extend the one dimensional
high-order, singularity-corrected trapezoidal rules for singular
functions in \cite{marin2014corrected} to arbitrary dimensions. Our
contributions are twofold. First, we provide a rigorous analysis of the
corrected trapezoidal rule's construction in a multi-dimensional
setting. Second, we derive simple and explicit formulas for the
correction weights via analytic continuation, which enable their
efficient and accurate computation, thereby significantly enhancing the
method's practical utility.

Beyond exchange energy calculations in electronic structure theory,
singularity corrected trapezoidal rules find broad applications across
scientific and engineering disciplines. For example, the evaluation of
surface singular integrals in boundary integral equations in 2D
\cite{bao2024singularity,wu2021zeta} and 3D \cite{izzo2023high,
wu2021corrected,wu2023unified}; the numerical discretization of the
fractional Laplacian in non-local Fokker-Planck equations in 2D
\cite{jiang2022arbitrarily}; the computation of Coulomb potentials in 3D
\cite{aguilar2005high}; and the evaluation of integral operators in
Lippmann-Schwinger equations for scattering problems
\cite{duan2009high}.

The rest of this paper is organized as follows.
\Cref{sec:singularity_correction} details the algorithm for the
singularity-corrected trapezoidal rule in arbitrary dimensions,
including the explicit construction of the correction weights.
\Cref{sec:convergence_analysis} provides a theoretical analysis of the
method's convergence, establishing guarantees for its high-order
accuracy. \Cref{sec:extension} discusses extensions of the
method to parallelotope domains and anisotropic types of singularities.
Numerical results in \cref{sec:numerical_results} validate the
theoretical analysis and demonstrate the algorithm's practical
performance. Finally, \cref{sec:conclusion} offers concluding remarks,
while the Appendix collects supporting proofs and technical derivations.

\section{Singularity corrected trapezoidal rule} \label{sec:singularity_correction}

This section presents a systematic framework for constructing a locally
corrected trapezoidal rule that handles isolated singularities in
arbitrary dimensions. We begin by establishing the necessary notation
and motivating the design of the local correction procedure.
Subsequently, we provide a complete description of the algorithm,
including explicit constructions for the correction weights. For
clarity, extended proofs and technical derivations are deferred to the
appendix.

\subsection*{Setup and notations}

In this paper, we denote integer lattice points with bold letters (e.g.,
$\bi\in\bbZ^d$). Furthermore, we use Greek letters, such as $\balpha =
(\alpha_1, \ldots, \alpha_d) \in \bbN^d$, to represent nonnegative
integer lattice points, which frequently serve as multi-indices in
Taylor expansions. The distinction between general integer points $\bi$
and nonnegative multi-indices $\balpha$ will be clear from context.
For vector $\bi=(i_1,\ldots,i_d) \in \bbR^d$, different norms for
vectors are defined as follows:
\begin{equation*}
    \infnorm{\bi} := \max(|i_1|, \cdots, |i_d|), \quad 
    \onenorm{\bi} := |i_1| + \cdots + |i_d|, \quad 
    |\bi| := \left(|i_1|^2 + \cdots + |i_d|^2\right)^{1/2}.
\end{equation*}

By translation, we assume the isolated singularity is located at the
origin and consider
\begin{equation*}
    I[\vphi \cdot \sigma] = \int_V \vphi(\bx)\sigma(\bx)\,\rd \bx,
    \qquad
    \sigma(\bx) = \frac{1}{|\bx|^{s}}, \qquad 0 < s < d,
\end{equation*}
where $\vphi$ is smooth, $\vphi\cdot\sigma$ is smooth except at the
singular points and periodic with unit cell $V=[-a,a]^d$ for some $a>0$.

We discretize $V$ by a uniform tensor-product grid. Let $N\in\bbN_{+}$
and set mesh size $h=a/N$. For lattice indices
$\bi=(i_1,\ldots,i_d)\in\bbZ^d$, define grid points
\begin{equation*}
    \bx_{\bi} := \bi h,
    \qquad
    \calX := \{\bx_{\bi} : \infnorm{\bi}\le N\}.
\end{equation*}
The set $\calX$ contains the origin and is symmetric in the sense that
$\bx_{\bi}\in\calX$ implies $-\bx_{\bi}\in\calX$.

Finally, we denote by $T_h^0[f]$ the \emph{punctured} trapezoidal rule
(excluding the origin). It is convenient to write it in the
tensor-product form 
\footnote{Given any set $A$, we use $\# A$ to denote the number of elements in $A$.}
\begin{equation*}
    T_h^0[f] := h^d \sum_{0 < \infnorm{\bi} \leq N} c_{\bi}\, f(\bx_{\bi}),
    \qquad
    c_{\bi} := 2^{-m(\bi)}, \quad m(\bi):=\#\{j: |i_j|=N\}.
\end{equation*}
Thus $c_{\bi}=1$ for interior points, and each boundary coordinate
contributes a factor $1/2$ (e.g., in three dimensions: face/edge/corner
weights $1/2$, $1/4$, $1/8$).

\subsection{Motivation}

For smooth periodic integrands, the trapezoidal rule converges
super-algebraically \cite{trefethen2014exponentially}. However, in our
setting, the integrand satisfies smooth periodic boundary condition but
has an isolated singularity at the origin, which limits the convergence
of the standard rule. The goal of singularity correction is to restore
high-order accuracy by modifying only a small number of quadrature
weights near the singular point.

\paragraph{Step 1: extract the singularity effect.}
Note that the boundary coordinates in the trapezoidal rule involve more
complicated weights, in order to simplify the analysis and focus on the
singularity, we introduce a smooth radial cutoff function
\begin{equation} \label{eq:boundary_cutoff}
    \eta(\bx) = \eta(|\bx|) = 
    \begin{cases}
        1, & |\bx|\leq \frac{a}{4},\\
        0, & |\bx|\geq \frac{a}{2}.
    \end{cases}
\end{equation}
This cutoff separates the local singularity effect and the boundary
contribution, which makes the singular correction independent of the
particular quadrature treatment \cite{fornberg2021improving,
kapur1997high} used near the boundary for more general boundary
conditions.

Using the cutoff function $\eta(\bx)$, we split:
\begin{equation*}
    I[\vphi\cdot\sigma]
    = \int_V \vphi(\bx)\sigma(\bx)\left(1-\eta(\bx)\right)\rd\bx
    + \int_V \vphi(\bx)\sigma(\bx)\,\eta(\bx)\rd\bx.
\end{equation*}
The first integrand agrees with the original integrand near $\partial V$
and is smooth periodic; hence its trapezoidal error is
super-algebraically small. The second term is supported strictly inside
$V$ and contains the singularity at the origin. Therefore, by
introducing $\phi=\vphi \eta$, it suffices to focus on integrals of the
form
\begin{equation*}
    I[\phi\cdot\sigma] = \int_V \phi(\bx)\sigma(\bx)\,\rd\bx,
\end{equation*}
where $\phi$ is smooth and compactly supported in $V$.

\paragraph{Step 2: local Taylor surrogate and the correction term.}
The difficulty is that $\phi(\bx)\sigma(\bx)$ is singular at $\bx=0$. To
isolate the leading singular behavior and improve the smoothness of the
integrand, fix an order $p\in\bbN$ and introduce a localized Taylor
surrogate. Let $P_{\phi}$ be the Taylor polynomial of $\phi$ at $\bx=0$
of degree $2p+1$,
\begin{equation*}
    P_{\phi}(\bx) := \sum_{\onenorm{\balpha}=0}^{2p+1}
    \frac{\partial^{\balpha}\phi(0)}{\balpha!}\,\bx^{\balpha},
\end{equation*}
and choose a radial localized window function $g\in C_c^{\infty}(V)$
such that
\begin{equation} \label{eq:localized_function_g}
    g(\bx)=g(|\bx|),\qquad g(0)=1,\qquad \partial^{\balpha}g(0)=0,
    \quad 1\leq \onenorm{\balpha}\leq 2p+1.
\end{equation}
Define $\tphi(\bx):=P_{\phi}(\bx)g(\bx)$. Then $\phi-\tphi$ is compactly
supported in $V$ and vanishes to order $2p+1$ at the origin:
\begin{equation*}
    \partial^{\balpha}\bigl(\phi-\tphi\bigr)(0)=0,
    \quad 0\leq \onenorm{\balpha}\leq 2p+1,
\end{equation*}
so that $(\phi-\tphi)(\bx)=O(|\bx|^{2p+2})$ near the origin. Hence
$(\phi-\tphi)(\bx)\sigma(\bx)=O(|\bx|^{2p+2-s})$, which has a weakened
singularity compared with $\phi(\bx)\sigma(\bx)$. This suggests the
decomposition
\begin{align*}
    I[\phi\cdot\sigma]
    &= I[(\phi-\tphi)\cdot\sigma] + I[\tphi\cdot\sigma] \\
    &= T_h^0[\phi\cdot\sigma] + (I-T_h^0)[(\phi-\tphi)\cdot\sigma]
       + \bigl(I[\tphi\cdot\sigma]-T_h^0[\tphi\cdot\sigma]\bigr).
\end{align*}
The remaining task is to approximate the last difference
$I[\tphi\cdot\sigma]-T_h^0[\tphi\cdot\sigma]$ by a \emph{local}
correction supported on a fixed stencil near $\bx=0$.

\paragraph{Step 3: moment matching and a linear system.}
Motivated by Step 2, we approximate the local defect
$I[\tphi\cdot\sigma]-T_h^0[\tphi\cdot\sigma]$ by a correction supported
on a small stencil $\calX_p\subset\bbZ^d$:
\begin{equation*}
    I[\tphi \cdot \sigma] - T_h^0[\tphi \cdot \sigma] 
    = a(h) \sum_{\bi \in \calX_p} w_{\bi}^h \, \tphi(\bi h),
\end{equation*}
where $a(h)$ is a scaling factor and $\{w_{\bi}^h\}$ are correction
weights. Since $\tphi=P_\phi g$, this requirement is enforced by
matching the trapezoidal defect on localized monomials. That is, for all
multi-indices $\balpha$ with $0\le \onenorm{\balpha}\le 2p+1$, we impose
\begin{equation} \label{eq:weight_equation}
    \int_V g(\bx)\sigma(\bx)\bx^{\balpha}\,\rd\bx
    = T_h^0[g\cdot\sigma\cdot \bx^{\balpha}]
    + a(h) \sum_{\bi\in\calX_p} w_{\bi}^h\,(\bi h)^{\balpha}\,g(\bi h).
\end{equation}
This gives a finite linear system for the unknowns $\{w_{\bi}^h\}$.
Because $g(\bx)$ and $\sigma(\bx)$ are radial, and because the domain,
grid, and stencil are symmetric, the moment conditions associated with
odd monomials are automatically satisfied. Thus the system reduces to
even moments and can be made square by a suitable choice of $\calX_p$,
together with symmetry conditions on the weights.

In practice we will avoid constructing $P_\phi g$ or evaluating
derivatives of $\phi$. Thus, the resulting corrected trapezoidal rule is
defined by applying the same local correction to the original function
$\phi$:
\begin{equation*}
    T_h^0[\phi\cdot\sigma]
    + a(h)\sum_{\bi\in\calX_p} w_{\bi}^h\,\phi(\bi h).
\end{equation*}

Finally, letting $h\to 0$ yields \emph{converged} weights
$w_{\bi}=\lim_{h\to 0^+} w_{\bi}^h$ that are independent of $h$ (and, in
fact, independent of the particular choice of $g$ within the above
class). We highlight that these converged weights give the same
approximation order as the $h$-dependent weights, while being more
practical to compute and apply.

\subsection{Linear system for the correction weights}

We now derive the concrete linear system for the correction weights. In
this paper, we adopt the correction stencil $\calX_p$ with $p \in \bbN$
defined as
\begin{equation*}
    \calX_p = \left\{ \bi \in \bbZ^d \, : \, 0 \leq \onenorm{\bi} \leq p \right\}.
\end{equation*}
This is a natural choice for the radial singular kernel $1/|\bx|^s$,
though other stencils are possible. For illustration, we visualize these
stencils for $d=2$ and $p=0, 1, 2$ in
\cref{fig:stencil_Xp_radial_symmetric_2d}.

\begin{figure}[htbp]
    \centering
    \includegraphics[width=0.9\textwidth]{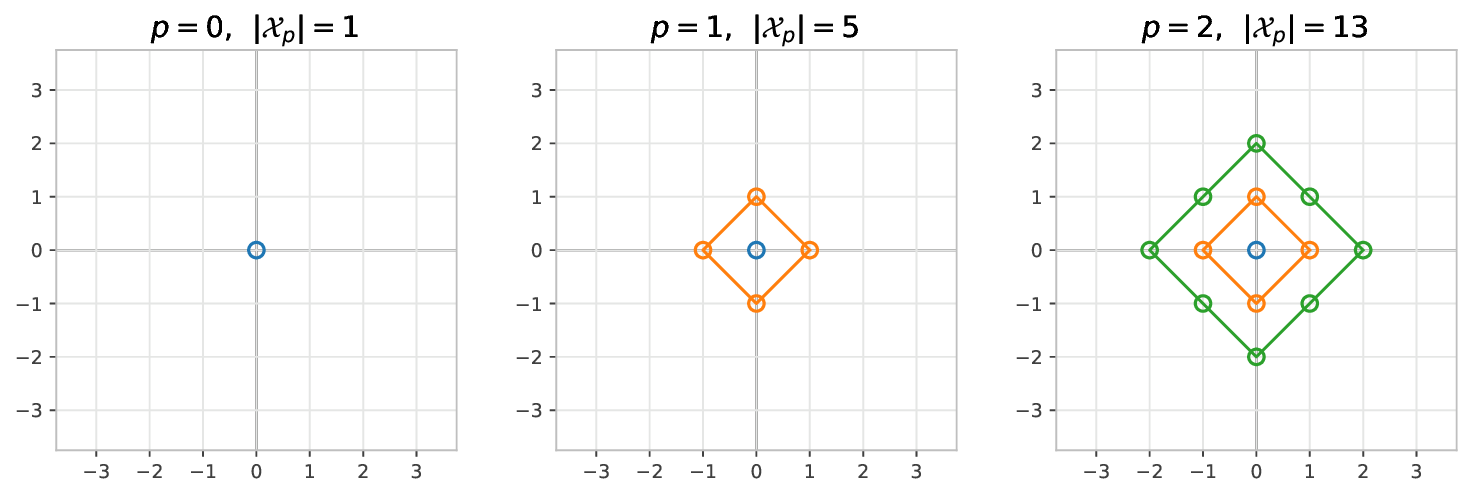}
    \caption{Symmetric correction stencils $\calX_p$ for $d=2$ and $p=0, 1, 2$.}
    \label{fig:stencil_Xp_radial_symmetric_2d}
\end{figure}

Given the radial symmetry of both the singular kernel and the localized
function $g(\bx)$, and the symmetry of the integral domain $V$ and grid
points $\calX$, we enforce symmetry of the weights under independent
sign flips of each coordinate. For $\bbeta=(\beta_1, \beta_2, \ldots,
\beta_d) \in\bbN^d$, define the sign-orbit
\begin{equation*}
    \calG_{\bbeta} := \left\{ (\pm \beta_1, \pm \beta_2, \ldots, \pm \beta_d) \in \bbZ^d \right\}.
\end{equation*}
We then impose the symmetry constraint
\begin{equation} \label{eq:symmetric_condition}
    w_{\bi}^h = w_{\bbeta}^h, \quad \text{for } \bi \in \calG_{\bbeta}, \, \, 0 \leq \onenorm{\bbeta} \leq p.
\end{equation}
Under the symmetry of weights, all moment constraints in
\eqref{eq:weight_equation} associated with odd monomials vanish
automatically, so it suffices to consider even moments only:
\begin{equation} \label{eq:weight_equation_even}
    \int_V g(\bx) \sigma(\bx) \bx^{2\balpha} \rd \bx = T^0_h[g \cdot \sigma \cdot \bx^{2\balpha}] 
    + a(h) \sum_{0 \leq \onenorm{\bbeta} \leq p} w_{\bbeta}^h \sum_{\bi \in \calG_{\bbeta}} (\bi h)^{2\balpha} \, g(\bi h), 
    \quad 0 \leq \onenorm{\balpha} \leq p.
\end{equation}
At this time, the number of equations in \eqref{eq:weight_equation_even}
is equal to the number of distinct unknowns $w_{\bbeta}^h$, which we
denote by $N_w$. 

\begin{remark}
    The symmetry \eqref{eq:symmetric_condition} can be strengthened
    further (e.g., grouping points with the same Euclidean distance to
    the origin) to reduce the number of unknowns. Such additional
    reductions must be paired with a compatible set of moment conditions
    to keep the linear system nonsingular.
\end{remark}

\textbf{Scaling factor.} 
For the homogeneous kernel $\sigma(\bx)=|\bx|^{-s}$, the scaling factor
$a(h)$ is determined by matching orders of $h$ at both sides of
\eqref{eq:weight_equation_even}. Applying the change variables
$\bx=h\by$ we have
\begin{align*}
    &\quad \int_V g(\bx) \sigma(\bx) \bx^{2\balpha} \rd \bx - T^0_h[g \cdot \sigma \cdot \bx^{2\balpha}] \\
    & = h^{d+2\onenorm{\balpha}-s} \Big( \int_{|\bx|_{\infty}\leq N} g(\bx h)\sigma(\bx) \bx^{2\balpha} \rd \bx - \sum_{0 < \infnorm{\bi} \leq N} g(\bi h) \sigma(\bi) \bi ^{2\balpha} \Big),
\end{align*}
where $N=a/h$. Therefore, we take the scaling factor as
\begin{equation*}
    a(h) = h^{d-s}.
\end{equation*}

\textbf{Matrix form.}
To recast \eqref{eq:weight_equation_even} in matrix form, order the
moment indices as
\begin{equation*}
    \left\{ \balpha \in \bbN^d \, : \, 0 \leq \onenorm{\balpha} \leq p \right\} 
    = \left\{ \balpha_1, \balpha_2, \ldots, \balpha_{N_w} \right\}.
\end{equation*}
Besides, order the sign-orbit representatives as
$\bbeta_1,\ldots,\bbeta_{N_w}$ in the same sequential order, where
$\bbeta_k\in\bbN^d$ and $0\leq\onenorm{\bbeta_k}\leq p$.
Then the linear system \eqref{eq:weight_equation_even} can be rewritten
as
\begin{equation*}
    \sum_{k=1}^{N_w} w_{\bbeta_k}^h \sum_{\bi \in \calG_{\bbeta_k}} \bi^{2\balpha_j} \, g(\bi h) 
    = \int_{\bbR^d} g(\bx h) \sigma(\bx) \bx^{2\balpha_j} \rd \bx
    - \sum_{\infnorm{\bi}=1}^{\infty} g(\bi h) \sigma(\bi) \bi^{2\balpha_j} 
    := b_j^h,
\end{equation*} 
for $j=1, \ldots, N_w$, where we have used that $g(\bx)$ is
compactly supported. Since $g(\bi h)$ takes the same value on each orbit
$\calG_{\bbeta_k}$, we introduce the matrix notations 
\begin{equation} \label{eq:matrix_of_zeta}
    (A)_{jk} = \sum_{\bi \in \calG_{\bbeta_k}} \bi^{2\balpha_j} 
    \quad \text{and} \quad (G^h)_{jk}=g(\bbeta_k h) \delta_{jk}, 
\end{equation}
for $j, k=1, \ldots, N_w$, where $\delta_{jk}$ is the Kronecker delta
notation and $G^h$ is diagonal. Thus, we can express the linear system
in matrix form as
\begin{equation} \label{eq:weight_linear_system}
    A G^h w^h = b^h,
\end{equation}
where $w^h$ and $b^h$ are the weight vector and the right-hand-side
vector, respectively.

The well-posedness of the moment-matching system rests on the
invertibility of the coefficient matrix $A$, which depends only on the
chosen stencil and the monomials used in the moment conditions.

\begin{lemma}[Invertibility of the coefficient matrix] \label{lem:non_singular_coeff}
    The coefficient matrix $A$ defined in \eqref{eq:matrix_of_zeta} is
    nonsingular.
\end{lemma}

The proof of \cref{lem:non_singular_coeff} is given in
\cref{app:nonsingularity_proof}.

Since $g(0)=1$, the diagonal matrix $G^h$ satisfies $G^h\to I$ as
$h\to 0^+$. Thus the limiting weights are determined by the
$h$-independent system $Aw=b$, once the limit of the right-hand side
$b^h$ is understood. This is the focus of the next section.

\subsection{Converged correction weights}

We now pass to the limit $h\to 0^+$ in the weight system
\eqref{eq:weight_linear_system}. Recall that $A$ is independent of $h$,
while $G^h\to I$ as $h\to 0^+$. Thus the key step is to understand the
asymptotic behavior (and then the limit) of the right-hand side vector
$b^h$. This section has two goals: (i) justify the existence and
convergence rate of the limiting weights $w$, and (ii) evaluate the
limiting right-hand side $\lim_{h \to 0^+} b^h$ via analytic
continuation. Proofs of technical claims are deferred to the appendix.

For the first part, we have the following results on the convergence
order of the right-hand side $b^h$ and the correction weights $w^h$ with
respect to $h$.

\begin{lemma}[Convergence order of RHS] \label{lem:error_order_rhs} 
    Let $p \in \bbN$. Assume $g(\bx)\in \ccinf(\bbR^d)$ satisfies
    \begin{equation*}
        \partial^{\bbeta} g(0) = 0, \quad 
        \text{for } 1 \leq \onenorm{\bbeta} \leq 2p+1.
    \end{equation*}
    Given multi-index $\balpha \in \bbN^d$ such that $0 \leq
    \onenorm{\balpha} \leq p$, let $s\in \bbC$ lie in the strip $0 <
    \rep(s) < d + 2\onenorm{\balpha}$. Then for $h > 0$ sufficiently
    small we have
    \begin{equation} \label{eq:error_order_rhs}
        \int_{\bbR^d} g(\bx h) \frac{\bx^{2\balpha}}{|\bx|^s} \rd \bx 
        - \sum_{\infnorm{\bi}=1}^{\infty} g(\bi h) \frac{\bi^{2\balpha}}{|\bi|^s} 
        = g(0)b_{\balpha}(s) + O(h^{2p+2}),
    \end{equation}
    where $b_{\balpha}(s)$ are constants independent of $g(\bx)$ and
    $h$.
\end{lemma}

We remark that in order to give an explicit formula of the right-hand
side $b_{\balpha}(s)$ by analytic continuation, we extend the range of
$s$ to the strip $0 < \rep(s) < d + 2\onenorm{\balpha}$ in
\cref{lem:error_order_rhs}. 

\begin{lemma}[Convergence order of weights] \label{lem:error_order_weights}
    For corrected weights $w_{\bi}^h$ defined in
    \eqref{eq:weight_equation} satisfying symmetric condition
    \eqref{eq:symmetric_condition} and its limit $w_{\bi} =
    \lim_{h\rightarrow 0^+} w_{\bi}^h$, we have
    \begin{equation*}
        \left| w_{\bi}^h - w_{\bi} \right| = O(h^{2p+2}), \quad 0 \leq \onenorm{\bi} \leq p.
    \end{equation*}
\end{lemma}

\begin{proof}
    According to linear system \eqref{eq:weight_linear_system} and the
    non-singularity in \cref{lem:non_singular_coeff}, we have
    \begin{equation*}
        w^h = (G^h)^{-1} A^{-1} b^h, \quad w = A^{-1} b.
    \end{equation*}
    By the assumptions on $g(\bx)$, Taylor's theorem gives
    \begin{equation*}
        |g(\bbeta h) - 1| = |g(\bbeta h)-g(0)| \leq C|\bbeta h|^{2p+2} \leq C h^{2p+2},
        \qquad 0\leq \onenorm{\bbeta}\leq p.
    \end{equation*}
    Hence $\|I-G^h\|_2\leq Ch^{2p+2}$ and $(G^h)^{-1}$ is uniformly
    bounded for all sufficiently small $h$. Moreover,
    \cref{lem:error_order_rhs} implies $\|b^h-b\|_2\leq Ch^{2p+2}$.
    Since $b$ is fixed, this also gives the uniform bound 
    $\|b^h\|_2\leq \|b\|_2+\|b^h-b\|_2\leq C$ for all sufficiently small
    $h$. Therefore
    \begin{align*}
        \|w^h - w\|_2 
        &= \| (G^h)^{-1} A^{-1} b^h - A^{-1} b \|_2 
        = \| (G^h)^{-1} (I - G^h) A^{-1} b^h + A^{-1} (b^h - b) \|_2 \\
        &\leq \|(G^h)^{-1}\|_2 \|I - G^h\|_2 \|A^{-1}\|_2 \|b^h\|_2 + \|A^{-1}\|_2 \|b^h - b\|_2 
        leq C h^{2p+2},
    \end{align*}
    which yields the stated componentwise bound.
\end{proof}

For the second part, we compute the limit of right-hand side $b^h$ as
$h\rightarrow 0^+$ explicitly. Inspired by the computation of lattice
sums and Wigner limit for the electron sums in chemistry
\cite{borwein2014lattice}, we derive an explicit expression of the limit
of $b^h$ by means of analytic continuation.

To make the analytic continuation between different regions clear, we
introduce the following notation.

\begin{definition}[Analytic continuation operator]
    Let $D_1$ be a non-empty domain in $\bbC$, let $D_2\subset\bbC$ be a
    non-empty set, and let $\Omega\subset\bbC$ be a connected open set
    such that $D_1\cup D_2\subset\Omega$. For an analytic function
    $f^{D_1}$ on $D_1$, suppose there exists an analytic function $F$ on
    $\Omega$ such that $F=f^{D_1}$ on $D_1$. The analytic continuation
    operator from $D_1$ to $D_2$ is then defined by restricting $F$ to
    $D_2$:
    \begin{equation*}
        \aco{D_1}{D_2}[f^{D_1}](z) := F(z), \quad z \in D_2.
    \end{equation*}
    We also denote the restriction over $D_2$ by $f^{D_2}$.
\end{definition}

\begin{remark}
    The operator $\aco{D_1}{D_2}$ is defined only when such an analytic
    continuation exists. The sets $D_1$ and $D_2$ need not overlap, and
    uniqueness follows from the identity theorem on the connected open
    set $\Omega$. For an analytic pair $(f, D)$, we write $f^{D}(z)$
    when we wish to emphasize the domain of analyticity; otherwise we
    simply write $f(z)$.
\end{remark}

To compute the limit of $b^h_{\balpha}$ as $h\rightarrow 0^+$ for a
given $\balpha \in \bbN^d$, we introduce two sets in the complex plane
$\bbC$. For brevity, set
\begin{equation*}
    s_{\balpha}:=d+2\onenorm{\balpha}.
\end{equation*}
We define
\begin{align*}
    D_1 &= \left\{ s \in \bbC \, : \, 0 < \rep(s) < s_{\balpha} \right\}, \\
    D_2 &= \left\{ s \in \bbC \, : \, s_{\balpha} \leq \rep(s) < s_{\balpha} + 1 \right\} \setminus \{s_{\balpha}\}.
\end{align*}
Their union
\begin{equation*}
    D_1\cup D_2
    = \left\{s\in\bbC:0<\rep(s)<s_{\balpha}+1\right\}
      \setminus \{s_{\balpha}\}
\end{equation*}
is the connected open set through which the analytic continuation is
performed.

By \cref{lem:error_order_rhs}, the limit of $b^h_{\balpha}$ as $h\to0^+$
exists. Since this limit is unique, it can be identified along any
sequence $h\to0^+$. We choose the admissible mesh sizes $h=a/N$ and let
$N\to\infty$, which gives
\begin{equation} \label{eq:def_b_alpha}
    b_{\balpha}^{D_1}(s) = \lim_{N\rightarrow \infty} 
    \left( \int_{\bbR^d} 
        g(\bx \frac{a}{N}) \frac{\bx^{2\balpha}}{|\bx|^s} \rd \bx 
        - \sum_{\infnorm{\bi}=1}^{\infty} g(\bi \frac{a}{N}) \frac{\bi^{2\balpha}}{|\bi|^s} 
    \right), 
    \quad s \in D_1.
\end{equation}
To examine this limit, since $\supp(g)\subset[-a,a]^d$, we introduce the
partial expressions 
\begin{equation*}
    u_N^{D_1}(s) := \int_{\bbR^d} g(\bx \frac{a}{N}) \frac{\bx^{2\balpha}}{|\bx|^s} \rd \bx, \quad 
    v_N^{D_1}(s) := \sum_{\infnorm{\bi}=1}^{N} g(\bi \frac{a}{N}) \frac{\bi^{2\balpha}}{|\bi|^s}.
\end{equation*}
A key observation is that, for any fixed $N$, the integral
$u_N^{D_1}(s)$ converges only for $\rep(s) < d + 2\onenorm{\balpha}$,
whereas the summation $v_N^{D_1}(s)$ is well-defined for all $s\in D_1
\cup D_2$.

\begin{lemma} \label{lem:holom_bs_D1}
    The function
    \begin{equation*}
        b_{\balpha}^{D_1}(s) = \lim_{N\rightarrow \infty} \left( \left[u_N^{D_1} - v_N^{D_1}\right](s) \right)    
    \end{equation*}
    is analytic on $D_1$.
\end{lemma}

\begin{lemma} \label{lem:exist_bs_D2}
    The function $b_{\balpha}^{D_1}(s)$ possesses an analytic
    continuation over $D_1 \cup D_2$ and we have
    \begin{equation*}
        b_{\balpha}^{D_2}(s) := \aco{D_1}{D_2}\left[b_{\balpha}^{D_1}\right](s) 
        = \aco{D_1}{D_2} \left[ \lim_{N\rightarrow \infty} \left( u_N^{D_1} - v_N^{D_1} \right) \right] (s) 
        = \lim_{N\rightarrow \infty} \left( \aco{D_1}{D_2} \left[u_N^{D_1} - v_N^{D_1}\right](s) \right).
    \end{equation*}
\end{lemma}

Moreover, we introduce another domain $D_3 \subseteq D_2$ that removes
the line $\{s \in \bbC: \rep(s) = s_{\balpha}\}$, defined as 
\begin{equation*}
    D_3 = \left\{ s \in \bbC \, : \, s_{\balpha} < \rep(s) < s_{\balpha} + 1 \right\}.
\end{equation*} 
The purpose of introducing $D_3$ is to avoid the line containing the
pole, so that we have a simple expression of $b^{D_3}_{\balpha}(s)$ over
$D_3$ (such an expression is only well-defined on $D_3$).

\begin{lemma} \label{lem:expression_bs_D3}
    For $s \in D_3$, there is
    \begin{equation*}
        b_{\balpha}^{D_3}(s) 
        = \lim_{N\rightarrow\infty} \left( \aco{D_1}{D_3}[u_N^{D_1}](s) - \aco{D_1}{D_3}[v_N^{D_1}](s) \right) 
        = - \lim_{N\rightarrow\infty} \aco{D_1}{D_3}[v_N^{D_1}](s)
        = - \sum_{\infnorm{\bi}=1}^{\infty} \frac{\bi^{2\balpha}}{|\bi|^s}.
    \end{equation*}
\end{lemma}

\begin{corollary}
    For $s \in D_1$, there is
    \begin{equation*}
        b_{\balpha}^{D_1}(s) = \aco{D_3}{D_1}[b_{\balpha}^{D_3}](s) 
        = \aco{D_3}{D_1} \left[ -\sum_{\infnorm{\bi}=1}^{\infty} \frac{\bi^{2\balpha}}{|\bi|^s} \right](s).
    \end{equation*}
\end{corollary}

For ease of understanding, we illustrate the above analytic-continuation
procedure at a point $s_0\in D_1\cap\bbR$ in
\cref{fig:analytic_continuation_balpha}.

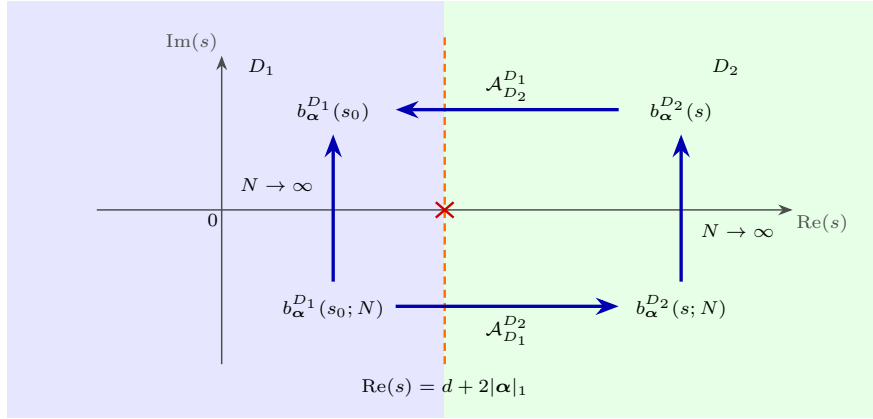
\begin{figure}[htbp]
    \centering
    \begin{tikzpicture}[x=1.18cm,y=1.02cm,>=Stealth, font=\small]
        \def\xmin{-1.4}
        \def\xmax{ 6.4}
        \def\ymin{-2}
        \def\ymax{ 2}

        \def\xPole{2.5}

        \def\bgPadX{1.0}
        \def\bgPadY{0.7}

        \def\xL{1.25}     
        \def\xR{5.15}     
        \def\yB{-1.05}    
        \def\yT{ 1.10}    

        \colorlet{axiscol}{black!70}
        \colorlet{flowcol}{blue!70!black}
        \colorlet{bgDOne}{blue!10}
        \colorlet{bgDTwo}{green!10}
        \colorlet{sepcol}{orange!85!red}

        \colorlet{limitcol}{black}  
        \colorlet{acocol}{black}       
        \colorlet{polecol}{red!80!black}       

        \tikzset{
            axis/.style={axiscol, line width=0.55pt},
            sepline/.style={sepcol, densely dashed, line width=0.9pt},
            flow/.style={flowcol, line width=1.25pt},
            lab/.style={font=\scriptsize, inner sep=1.2pt},
            limlab/.style={lab, text=limitcol, font=\scriptsize\bfseries},
            acolab/.style={lab, text=acocol, font=\scriptsize\bfseries},
        }

        \begin{scope}
            \clip ({\xmin-\bgPadX},{\ymin-\bgPadY}) rectangle ({\xmax+\bgPadX},{\ymax+\bgPadY});
            \fill[bgDOne, opacity=0.35] ({\xmin-\bgPadX},{\ymin-\bgPadY}) rectangle (\xPole,{\ymax+\bgPadY});
            \fill[bgDTwo, opacity=0.28] (\xPole,{\ymin-\bgPadY}) rectangle ({\xmax+\bgPadX},{\ymax+\bgPadY});
        \end{scope}

        \draw[axis, ->] (\xmin,0) -- (\xmax,0) node[below right, lab] {$\rep(s)$};
        \draw[axis, ->] (0,\ymin) -- (0,\ymax) node[above left, lab] {$\text{Im}(s)$};

        \node[lab, anchor=north east] at (0,0) {$0$};

        \draw[sepline] (\xPole,\ymin) -- (\xPole,\ymax+0.3);
        \node[lab, below=4pt] at (\xPole,\ymin) {$\rep(s)=d+2\onenorm{\balpha}$};

        \draw[polecol, line width=0.9pt] (\xPole,0) ++(-0.10,-0.10) -- ++(0.20,0.20);
        \draw[polecol, line width=0.9pt] (\xPole,0) ++(-0.10, 0.10) -- ++(0.20,-0.20);

        \node[lab, anchor=north] at ({(\xmin+\xPole)/2-0.1},\ymax) {$D_1$};
        \node[lab, anchor=north] at ({(\xPole+\xmax)/2+1.2},\ymax) {$D_2$};

        \node[lab, anchor=south] at (\xL,\yT) {$b_{\balpha}^{D_1}(s_0)$};
        \node[lab, anchor=north] at (\xL,\yB) {$b_{\balpha}^{D_1}(s_0;N)$};

        \node[lab, anchor=south] at (\xR,\yT) {$b_{\balpha}^{D_2}(s)$};
        \node[lab, anchor=north] at (\xR,\yB) {$b_{\balpha}^{D_2}(s;N)$};
        \draw[flow, ->] (\xL,\yB+0.12) -- (\xL,\yT-0.12);
        \node[limlab, anchor=east] at (\xL-0.18,{(\yB+\yT)/2+0.3}) {$N\to\infty$};

        \draw[flow, ->] (\xR,\yB+0.12) -- (\xR,\yT-0.12);
        \node[limlab, anchor=west] at (\xR+0.18,{(\yB+\yT)/2-0.3}) {$N\to\infty$};

        \draw[flow, ->] (\xL+0.7,\yB-0.2) -- (\xR-0.7,\yB-0.2);
        \node[acolab, below=2pt] at ({(\xL+\xR)/2},\yB-0.2) {$\aco{D_1}{D_2}$};

        \draw[flow, ->] (\xR-0.7,\yT+0.2) -- (\xL+0.7,\yT+0.2);
        \node[acolab, above=2pt] at ({(\xL+\xR)/2},\yT+0.2) {$\aco{D_2}{D_1}$};

    \end{tikzpicture}
    \caption{A schematic of computing the limit using
    analytic-continuation. For $b_{\balpha}^{D_1}(s_0)$ defined in
    \eqref{eq:def_b_alpha}, start from $b_{\balpha}^{D_1}(s_0;N)$,
    continue to $D_2$ to obtain $b_{\balpha}^{D_2}(s;N)$, take
    $N\to\infty$ in $D_2$, then continue back to recover
    $b_{\balpha}^{D_1}(s_0)$.}
    \label{fig:analytic_continuation_balpha}
\end{figure}

In order to compute the corrected weights, we will derive an explicit
expression of $b_{\balpha}^{D_1}(s)$, which is a generalized zeta
function (up to sign). Define
\begin{equation*}
    Z_{2\balpha}(s) := \sum_{\bn \in \bbZ^{d} \setminus \{\bzero\}} \frac{\bn^{2\balpha}}{|\bn|^s}, 
    \quad \rep(s) > d+2\onenorm{\balpha}. 
\end{equation*}
In one dimensional scenario, it happens to be
$2\zeta(s-2\onenorm{\balpha})$, where $\zeta(s)$ is the well-known
Riemann zeta function. In higher dimension, the case $\balpha=\bzero$
reduces to the Epstein zeta function
\cite{buchheit2024computation,epstein1903theorie,epstein1906theorie},
which has a meromorphic continuation to the whole complex plane $\bbC$
with the only simple pole $s=d$. Following the idea of Riemann, we
derive a functional equation for $Z_{2\balpha}(s)$ that yields its
meromorphic continuation and a rapidly convergent series representation
for numerical evaluation.

\begin{theorem} \label{thm:zeta_function} 
    For generalized zeta function defined as 
    \begin{equation*}
        Z_{2\balpha}(s) := \sum_{\bn \in \bbZ^{d} \setminus \{\bzero\}} \frac{\bn^{2\balpha}}{|\bn|^s}, 
        \quad \rep(s) > d+2\onenorm{\balpha},
    \end{equation*}    
    it has a meromorphic continuation to the whole complex plane $\bbC$
    with the only simple pole at $s=d+2\onenorm{\balpha}$, and satisfies
    the functional equation
    \begin{equation} \label{eq:zeta_func_equation}
        \begin{split}
            \frac{\Gamma(s/2)}{\pi^{s/2}} Z_{2\balpha}(s) 
            &= \left(-\frac{1}{4\pi}\right)^{\onenorm{\balpha}} \frac{2H_{2\balpha}(0)}{s - d - 2\onenorm{\balpha}} 
            - \frac{2\delta_{0, \balpha}}{s} 
            + \sum_{\bn\in \bbZ^d \setminus \{\bzero\}} \bn^{2\balpha} \int_{1}^{\infty} \re^{-|\bn|^2\pi t} \, t^{\frac{s}{2} -1} \rd t \\
            &\quad + \left(-\frac{1}{4\pi}\right)^{\onenorm{\balpha}} \sum_{\bk \in \bbZ^d \setminus \{\bzero\}}
            \int_{1}^{\infty} \re^{-|\bk|^2 \pi t} H_{2\balpha}(\sqrt{\pi t} \, \bk) \, t^{\frac{d + 2 \onenorm{\balpha} - s}{2} - 1} \rd t,        
        \end{split}
    \end{equation}
    where $H_{2\balpha}(\bx)$ is the multivariate Hermite polynomial in
    order $2\balpha$ defined in \cref{def:hermite_poly}.
\end{theorem}

For numerical evaluation, it is useful to rewrite the integral terms in
\eqref{eq:zeta_func_equation} using the upper incomplete Gamma function
\begin{equation*}
    \Gamma(s, x) := \int_x^{\infty} t^{s-1} \re^{-t} \,\rd t,
    \quad \text{for } s\in\bbC, \, x>0.
\end{equation*}
For $s\in\bbC$ and $x>0$, the change of variables $u=tx$ gives
\begin{equation*}
    \int_{1}^{\infty} t^{s - 1} \re^{-tx} \rd t = \frac{1}{x^s} \, \Gamma(s, x).
\end{equation*}
Applying this identity to the first integral series in
\eqref{eq:zeta_func_equation} yields
\begin{equation} \label{eq:zeta_series_n}
    \begin{split}
        \sum_{\infnorm{\bn}=1}^{\infty} \bn^{2\balpha} \int_{1}^{\infty} \re^{-|\bn|^2\pi t} \, t^{\frac{s}{2} -1} \,\rd t 
        &= \sum_{\infnorm{\bn}=1}^{\infty} \bn^{2\balpha} \frac{1}{|\bn|^s \pi^{s/2}} \Gamma(\frac{s}{2}, |\bn|^2 \pi) \\
        &= \frac{1}{\pi^{s/2}} \sum_{\infnorm{\bn}=1}^{\infty} \left(\frac{\bn}{|\bn|}\right)^{2\balpha} 
        \frac{1}{|\bn|^{s-2\onenorm{\balpha}}} \Gamma(\frac{s}{2}, |\bn|^2 \pi).         
    \end{split}
\end{equation}
For the second integral series, we first expand
$H_{2\balpha}(\bx) = \sum_{0 \leq \onenorm{\bbeta} \leq
2\onenorm{\balpha}} c_{\bbeta} \bx^{\bbeta}$ and then apply the same
identity of incomplete Gamma function term by term. This gives
\begin{multline} \label{eq:zeta_series_k}
    \sum_{\infnorm{\bk}=1}^{\infty} \int_{1}^{\infty} \re^{-|\bk|^2 \pi t} H_{2\balpha}(\sqrt{\pi t} \, \bk) \, t^{\frac{d+2\onenorm{\balpha}-s}{2} - 1} \,\rd t \\
    = \pi^{\frac{s-d-2\onenorm{\balpha}}{2}} \sum_{0 \leq \onenorm{\bbeta} \leq 2\onenorm{\balpha}} c_{\bbeta} \sum_{\infnorm{\bk}=1}^{\infty} 
    \left( \frac{\bk}{|\bk|} \right)^{\bbeta} \frac{1}{|\bk|^{d+2\onenorm{\balpha}-s}} \Gamma\!\left(\frac{\onenorm{\bbeta}+d +2\onenorm{\balpha} -s}{2}, |\bk|^2 \pi\right).
\end{multline}

The series in representations \eqref{eq:zeta_series_n} and
\eqref{eq:zeta_series_k} are rapidly convergent because the incomplete
Gamma factors decay exponentially in the lattice index. The following
tail estimate quantifies this decay and can be used to choose a finite
summation cutoff.

\begin{corollary}
    Given $a, b \in \bbR$, for $N \in \bbN$ such that $\pi N^2 \geq
    \max\left\{a, 1, \frac{2a+b+d}{2}\right\}$,    
    we have 
    \begin{equation*}
        \sum_{\infnorm{\bn}=N+1}^{\infty} |\bn|^{b} \, \Gamma(a, |\bn|^2 \pi) 
        \leq 3^{d-1} \, d^{1 + \max\{0, b/2\}} \, \pi^a \cdot N^{2a+b+d} \re^{-\pi N^2}.   
    \end{equation*}
\end{corollary}

\begin{proof}
    This is a special case of \cref{lem:tail_estimate}.
\end{proof}

In practice, this bound is used to choose the truncation range needed
for a prescribed accuracy. To reduce roundoff error when summing the
series, one may use Kahan summation \cite{kahan1965pracniques}. Finally,
because the linear system for the correction weights can become
ill-conditioned for large $p$, it is preferable to solve it either
symbolically or using high-precision arithmetic.

\subsection{The SinCoTrap quadrature rule}

We now summarize the quadrature rule obtained from the preceding
construction. Fix a correction order $p\in\bbN$ and the correction
stencil
\begin{equation*}
    \calX_p=\{\bi\in\bbZ^d:0\leq\onenorm{\bi}\leq p\}.    
\end{equation*}
The converged correction weights $w$ are determined by the linear system
\begin{equation*}
    Aw=b,
\end{equation*}
where for $\balpha, \bbeta\in\bbN^d$ such that $0\leq\onenorm{\balpha},
\onenorm{\bbeta}\leq p$, 
\begin{equation*}
    A_{\balpha,\bbeta}
    = \sum_{\bi\in\calG_{\bbeta}}\bi^{2\balpha},
    \qquad
    b_{\balpha}(s)
    = -Z_{2\balpha}(s).
\end{equation*}
Here $\calG_{\bbeta} = \left\{ (\pm \beta_1, \pm \beta_2, \ldots, \pm
\beta_d) \in \bbZ^d \right\}$. The weights are then extended to all
$\bi\in\calX_p$ by $w_{\bi}=w_{\bbeta}$ for $\bi\in\calG_{\bbeta}$.

For $h=a/N$, the SinCoTrap rule for $\sigma(\bx)=|\bx|^{-s}$ is defined
as
\begin{equation} \label{eq:sincotrap_compact}
    Q^p_h[\phi \cdot \sigma]
    := T^0_h[\phi \cdot \sigma] + \calC_{h}^p[\phi],
    \qquad
    \calC_{h}^p[\phi]
    := h^{d-s} \sum_{\bi \in \calX_p} w_{\bi}\,\phi(\bi h),
\end{equation}
where $\calC_{h}^p[\cdot]$ is the correction operator for the
singularity.

\section{Convergence order analysis} \label{sec:convergence_analysis}

This section establishes the convergence rate of the singularity
corrected trapezoidal rule for singular kernel $\sigma(\bx)=|\bx|^{-s}$
with $0<s<d$.

\begin{lemma}[Odd-moment cancellation] \label{lem:odd_component_cancellation}
    Let $g(\bx)=g(|\bx|)$ be smooth and compactly supported in
    $V=[-a,a]^d$, and let $\sigma(\bx)=|\bx|^{-s}$ with $0<s<d$. If a
    multi-index $\balpha$ has at least one odd component, then, for any
    correction order $p\in\bbN$, the exact integral and the numerical
    quadrature vanish,
    \begin{equation*}
        I[g\cdot\sigma\cdot\bx^{\balpha}]=0,
        \qquad \text{and} \qquad
        Q_h^p[g\cdot\sigma\cdot\bx^{\balpha}]=0.
    \end{equation*}
\end{lemma}

\begin{proof}
    Without loss of generality, suppose that $\alpha_1$ is odd. Since
    $g$ and $\sigma$ are radial, then we have
    \begin{align*}
        I[g\cdot\sigma\cdot\bx^{\balpha}]
        &=\int_{[-a,a]^{d-1}}\int_0^a
        g(|\bx|)\sigma(|\bx|)
        \left[x_1^{\alpha_1}+(-x_1)^{\alpha_1}\right]
        \prod_{j=2}^d x_j^{\alpha_j}\,\rd x_1\cdots\rd x_d
        =0.
    \end{align*}
    Similarly, using
    $c_{(i_1,i_2,\ldots,i_d)}=c_{(-i_1,i_2,\ldots,i_d)}$, we obtain
    \begin{align*}
        T_h^0[g\cdot\sigma\cdot\bx^{\balpha}]
        &=h^d\sum_{i_2=-N}^N\cdots\sum_{i_d=-N}^N
        \sum_{i_1=1}^N c_{\bi}\,g(|\bi h|)\sigma(|\bi h|)
        \left[(i_1h)^{\alpha_1}+(-i_1h)^{\alpha_1}\right]
        \prod_{j=2}^d(i_jh)^{\alpha_j}
        =0.
    \end{align*}
    For every sign-orbit $\calG_{\bbeta}$, the symmetry condition
    \eqref{eq:symmetric_condition} gives
    \begin{equation*}
        \sum_{\bi\in\calG_{\bbeta}}(\bi h)^{\balpha}
        =\frac12\sum_{\bi\in\calG_{\bbeta}}
        \left[(\bi h)^{\balpha}
        +((-i_1,i_2,\ldots,i_d)h)^{\balpha}\right]=0.
    \end{equation*}
    Consequently, the correction $\calC_h^p[g\cdot\bx^{\balpha}]$ is
    zero and we conclude that $Q_h^p[g\cdot\sigma\cdot\bx^{\balpha}]=0$.
\end{proof}

Now we state the main theorem regarding the convergence order of
SinCoTrap for compactly supported singular integrals.

\begin{theorem} \label{thm:error_order_compact_fun}
    Let $V=[-a,a]^d\subseteq\bbR^d$ and consider
    \begin{equation*}
        I[\phi \cdot \sigma] = \int_{V} \phi(\bx)\sigma(\bx)\,\rd \bx,
        \qquad
        \sigma(\bx)=\frac{1}{|\bx|^s},
        \qquad
        0<s<d,
    \end{equation*}
    where $\phi$ is smooth and compactly supported in $V$. For a
    correction order $p\in\bbN$, let $Q_h^p[\cdot]$ be the SinCoTrap
    rule defined in \eqref{eq:sincotrap_compact}. Then, as $h\to0^+$,
    \begin{equation*}
        I[\phi \cdot \sigma] = Q^p_h[\phi \cdot \sigma] + O(h^{2p+2+d-s}).
    \end{equation*}
\end{theorem}

\begin{proof}
    Let $P_{\phi}(\bx)$ be the Taylor polynomial of $\phi(\bx)$ at
    $\bx=0$ of degree $2p+1$ and define the localized surrogate as
    \begin{equation*}
        \tphi(\bx) := P_{\phi}(\bx) g(\bx), \quad 
        P_{\phi}(\bx) := \sum_{\onenorm{\balpha} = 0}^{2p+1} \frac{\partial^{\balpha} \phi(0)}{\balpha!} \bx^{\balpha},
    \end{equation*}
    where $g(\bx) \in \ccinf(\bbR^d)$ is compactly supported in domain
    $V$ such that
    \begin{equation*}
        g(\bx)=g(|\bx|),\qquad g(0)=1,\qquad \partial^{\balpha}g(0)=0,
        \quad 1\leq \onenorm{\balpha}\leq 2p+1.
    \end{equation*}
    Decompose the quadrature error of the singular integral as:
    \begin{align*}
        I[\phi \cdot \sigma] - Q_h^p[\phi \cdot \sigma] 
        &= (I - T^0_h) [(\phi - \tphi) \cdot \sigma] + \calC_{h}^p[\tphi - \phi] + (I - Q_h^p)[\tphi \cdot \sigma] \\
        &:= E_1 + E_2 + E_3.
    \end{align*}

    For $E_1$, denote $f:= \phi - \tphi$, then by direct calculation we
    have
    \begin{equation} \label{eq:diff_phi_tphi}
        f(0) = 0, \quad \partial^{\balpha} f(0) = 0, \quad 
        \text{for } 1 \leq \onenorm{\balpha} \leq 2p+1.
    \end{equation}
    Thus, using the fact that $f \in C^{\infty}_{\text{c}}(V)$, we can
    write
    \begin{align*}
        E_1 = I[f\cdot \sigma] - T^0_h[f \cdot \sigma]
        &= \int_{[-a, a]^d} f(\bx) \frac{1}{|\bx|^s} \rd \bx 
            - \sum_{\infnorm{\bi}=1}^N h^d f(\bi h) \frac{1}{|\bi h|^s} \\
        &= h^{d-s} \Big( \int_{\bbR^d} f(\bx h) \frac{1}{|\bx|^s} \rd \bx 
            - \sum_{\infnorm{\bi}=1}^{\infty} f(\bi h) \frac{1}{|\bi|^s} \Big) 
        = O(h^{2p+2+d-s}),
    \end{align*}
    where we used \cref{lem:error_order_rhs} with $\balpha=\bzero$ and
    $f(0)=0$ in the last step.

    For $E_2$, note that Taylor's theorem with \eqref{eq:diff_phi_tphi}
    gives $|\tphi(\bx)-\phi(\bx)|\leq C|\bx|^{2p+2}$ on $V$,
    hence on the fixed stencil $\calX_p$ with bounded weights, we have
    \begin{equation*}
        |E_2|
        \leq h^{d-s}\sum_{\bi\in\calX_p} |w_{\bi}|\,|\tphi(\bi h)-\phi(\bi h)|
        \leq C h^{2p+2+d-s}.
    \end{equation*}

    For $E_3$, expand $\tphi=P_\phi g$ and use
    \cref{lem:odd_component_cancellation} to drop the odd monomials we
    obtain
    \begin{equation*}
        (I-Q_h^p)[\tphi\cdot\sigma]
        = \sum_{\onenorm{\balpha}=0}^{p}\frac{\partial^{2\balpha}\phi(0)}{(2\balpha)!}\,
        (I-Q_h^p)[\bx^{2\balpha}g\cdot\sigma].
    \end{equation*}
    For each even monomial above, according to the moment conditions
    \eqref{eq:weight_equation} we have
    \begin{equation*}
        I[\bx^{2\balpha}g\cdot\sigma] - T_h^0[\bx^{2\balpha}g\cdot\sigma]
        = h^{d-s}\sum_{\bi\in\calX_p} w_{\bi}^h\,(\bi h)^{2\balpha}g(\bi h).
    \end{equation*}
    On the other hand, the SinCoTrap rule $Q_h^p$ uses the converged
    weights $w_{\bi}$, so
    \begin{equation*}
        (I-Q_h^p)[\bx^{2\balpha}g\cdot\sigma]
        = \left(I-T_h^0\right)[\bx^{2\balpha}g\cdot\sigma] 
        - h^{d-s}\sum_{\bi\in\calX_p} w_{\bi}\,(\bi h)^{2\balpha}g(\bi h).
    \end{equation*}
    Substituting the moment identity gives
    \begin{equation*}
        (I-Q_h^p)[\bx^{2\balpha}g\cdot\sigma]
        = h^{d-s}\sum_{\bi\in\calX_p}
        (w_{\bi}^h-w_{\bi})\,(\bi h)^{2\balpha}g(\bi h).
    \end{equation*}
    Thus, by \cref{lem:error_order_weights} we have
    \begin{align*}
        |E_3| 
        \leq C h^{d-s} \sum_{\onenorm{\balpha} = 0}^{p} \sum_{\bi \in \calX_p} 
            \left|(\bi h)^{2\balpha} g(\bi h)\right| \left| w_{\bi}^h - w_{\bi} \right| 
        \leq C h^{2p+2+d-s}.
    \end{align*}

    Combining the three bounds yields the claim.
\end{proof}

\begin{remark}
    The assumption that $\phi$ is smooth is imposed only for
    convenience; the argument requires only finitely many derivatives.
    In particular, it suffices to assume that $\phi\in C_c^{d+2p+3}(V)$
    and to use the corresponding finite-regularity version of
    \cref{lem:error_order_rhs}, obtained by applying Poisson summation
    to finitely differentiable, compactly supported functions. This
    regularity threshold is sufficient and may not be optimal.
\end{remark}

\begin{remark}
    Replacing $w_{\bi}$ by the non-converged weights $w_{\bi}^h$ removes
    the term $E_3$ above, but $E_1$ and $E_2$ still scale like
    $h^{2p+2+d-s}$, so the overall order remains unchanged.
\end{remark}

Next we give the convergence rate for singular integrals satisfying the
smooth periodic boundary conditions.

\begin{corollary}
    Let $V=[-a,a]^d\subseteq\bbR^d$ and consider
    \begin{equation*}
        I[\vphi \cdot \sigma] = \int_V \vphi(\bx)\sigma(\bx)\,\rd \bx,
        \qquad
        \sigma(\bx)=\frac{1}{|\bx|^s},
        \qquad
        0<s<d,
    \end{equation*}
    where $\vphi(\bx)$ is smooth, $\vphi(\bx)\sigma(\bx)$ is periodic on
    $V$ and is smooth except at the singular points. For a correction
    order $p\in\bbN$, let $Q_h^p[\cdot]$ be the SinCoTrap rule defined
    in \eqref{eq:sincotrap_compact}. Then, as $h\to0^+$,
    \begin{equation*}
        I[\vphi \cdot \sigma] = Q^p_h[\vphi \cdot \sigma] + O(h^{2p+2+d-s}).
    \end{equation*}
\end{corollary}

\begin{proof}
    Using the smooth cutoff $\eta$ defined in
    \eqref{eq:boundary_cutoff}, we split the integral as
    \begin{equation*}
        I[\vphi \cdot \sigma] 
        = I[\vphi \cdot \sigma \cdot \eta] + I[\vphi \cdot \sigma \cdot (1 - \eta)].
    \end{equation*}
    For the first term, since $\vphi(\bx)\eta(\bx)$ is smooth and
    compactly supported in $V$, \cref{thm:error_order_compact_fun}
    implies
    \begin{equation*}
        I[\vphi \cdot \sigma \cdot \eta] 
        = T^0_h[\vphi \cdot \sigma \cdot \eta] + \calC_{h}^p[\vphi \cdot \eta] + O(h^{2p+2+d-s}).
    \end{equation*}
    For the second term, since $\vphi(\bx) \sigma(\bx) (1 - \eta(\bx))$
    is $0$ at $x=0$, the punctured trapezoidal rule is equivalent to the
    standard trapezoidal rule. Moreover, since $\vphi(\bx) \sigma(\bx)
    (1 - \eta(\bx))$ is smooth and periodic with respect to the domain
    $V$, the Euler-Maclaurin formula yields
    \begin{equation*}
        I[\vphi \cdot \sigma \cdot (1 - \eta)] = T^0_h[\vphi \cdot \sigma \cdot (1 - \eta)] + O(h^{q}), 
    \end{equation*}
    for any $q>0$. Finally, note that $\calC_{h}^p[\vphi \cdot \eta] =
    \calC_{h}^p[\vphi]$ for $h>0$ sufficient small, we have
    \begin{align*}
        I[\vphi \cdot \sigma] 
        &= T^0_h[\vphi \cdot \sigma \cdot \eta] + \calC_{h}^p[\vphi \cdot \eta] + O(h^{2p+2+d-s}) 
            + T^0_h[\vphi \cdot \sigma \cdot (1 - \eta)] + O(h^{q}) \\
        &= T^0_h[\vphi \cdot \sigma] + \calC_{h}^p[\vphi] + O(h^{2p+2+d-s}),
    \end{align*}
    which completes the proof.
\end{proof}

\section{Extensions} \label{sec:extension}

In this section, we present two extensions of the SinCoTrap. First, we
extend the method to integrals with singular kernel $1/|\bx^\top B
\bx|^{s/2}$, where $B$ is a symmetric positive definite matrix. Second, we
extend the method to integrals over general bounded domains,
specifically parallelotopes. We note that these two extensions are
equivalent under a linear transformation and they find applications in
electronic structure calculations for solids with parallelepiped-shaped
unit cells. 

\subsection{Extension to anisotropic singular kernel}

We begin with an anisotropic singular kernel induced by a symmetric
positive definite matrix $B \in \bbR^{d\times d}$. Consider
\begin{equation*}
    I[\phi \cdot \sigma_B] = \int_{V} \phi(\bx) \sigma_B(\bx) \rd \bx, \quad 
    \sigma_B(\bx) = \frac{1}{|\bx^\top B \bx|^{s/2}} \quad 0 < s < d,
\end{equation*}
where $V=[-a, a]^d$, $\phi(\bx)$ is smooth and compactly supported in
$V$. Since $B$ is invertible, the only singular point of $I[\phi \cdot
\sigma_B]$ remains at $\bx=0$. For discretization, we employ the same
uniform grid used for the isotropic kernel $\sigma(\bx)=1/|\bx|^s$.

\paragraph{Central symmetry and moment constraints.}
Since the singular kernel is changed, the local correction weights must
be recomputed accordingly. In contrast to the isotropic kernel
$|\bx|^{-s}$, the anisotropic kernel $\sigma_B(\bx)$ is generally not
radially symmetric (unless $B$ is a scalar multiple of the identity). It
retains only \emph{central symmetry}
\begin{equation*}
    \sigma_B(-\bx)=\sigma_B(\bx),
\end{equation*}
since $\bx^\top B \bx = (-\bx)^\top B (-\bx)$. This reduced symmetry
affects both the design of the correction stencil $\calX_p$ and which
moment constraints are automatically satisfied.

Following the moment-matching construction in
\cref{sec:singularity_correction}, we again determine the correction
weights by requiring exactness on localized monomials
$g(\bx)\sigma_B(\bx)\bx^{\balpha}$. Accordingly, we assume the
correction stencil $\calX_p\subset\bbZ^d$ is centrally symmetric and
contains the origin, i.e., $\bi\in\calX_p$ implies $-\bi\in\calX_p$ and
$0\in\calX_p$, and we impose the same symmetry on the weights:
\begin{equation*}
    w_{-\bi}^h = w_{\bi}^h.
\end{equation*}
These weights are determined by enforcing, for $0\le
\onenorm{\balpha}\le 2p+1$,
\begin{equation*}
    \int_V g(\bx)\sigma_B(\bx)\bx^{\balpha}\,\rd \bx
    = T^0_h[g \cdot \sigma_B \cdot \bx^{\balpha}]
    + a(h) \sum_{\bi \in \calX_p} w_{\bi}^h \, (\bi h)^{\balpha} \, g(\bi h).
\end{equation*}
As in the isotropic case, we choose a radial localized function  $g\in\ccinf(V)$ such
that $g(0)=1$ and $\partial^{\bbeta}g(0)=0$ for 
$1\leq \onenorm{\bbeta}\leq 2p+1$.

With only central symmetry, the automatic cancellation applies to
monomials of \emph{odd total degree}: if $\onenorm{\balpha}$ is odd,
then both the integral and the corrected rule vanish (see
\cref{lem:odd_l1_cancellation}). In contrast, monomials with
$\onenorm{\balpha}$ even but with at least one odd component are not
forced to vanish in general. Therefore, the moment-matching system must
include all multi-indices with even total degree,
\begin{equation*}
    \onenorm{\balpha}=2q,\quad \text{for } q=0,1,\ldots,p,
\end{equation*}
which determines the number of independent constraints (and hence the
number of independent weights).

\paragraph{Stencil choice and linear system.}
As discussed above, in the anisotropic case the symmetry only forces
cancellation of moments with odd total degree. Hence, to obtain a square
moment-matching system we choose the stencil so that the number of
independent weights matches the number of \emph{even-total-degree}
moment conditions, i.e., $N_w$.

A convenient way is to take one representative in the nonnegative
orthant for each multi-index of even total degree up to $2p$ and then
include its negative. Specifically, define
\begin{equation*}
    \calX_p := \calX_p^{+} \cup \calX_p^{-}, \qquad
    \calX_p^+ := \left\{ \bbeta \in \bbN^d : \onenorm{\bbeta} \leq 2p,\ \onenorm{\bbeta}\ \text{even} \right\}, \qquad
    \calX_p^- := -\calX_p^+.
\end{equation*}
Since every nonzero point in $\calX_p^+$ is paired with its negative in
$\calX_p^-$, and conversely, thus the full stencil $\calX_p$ is
centrally symmetric. We impose the central symmetry
$w_{-\bbeta}^h=w_{\bbeta}^h$, so the independent unknowns can be taken
as $\{w_{\bbeta}^h\}_{\bbeta\in\calX_p^+}$. (Here the origin is included
in $\calX_p^+$ and is not duplicated by the sign pairing.) For
illustration, we plot the correction stencils $\calX_p$ for $p=0,1,2$ in
two dimensions in \cref{fig:stencil_even_central_symmetric_2d}.

\begin{figure}[htbp]
    \centering
    \includegraphics[width=0.9\textwidth]{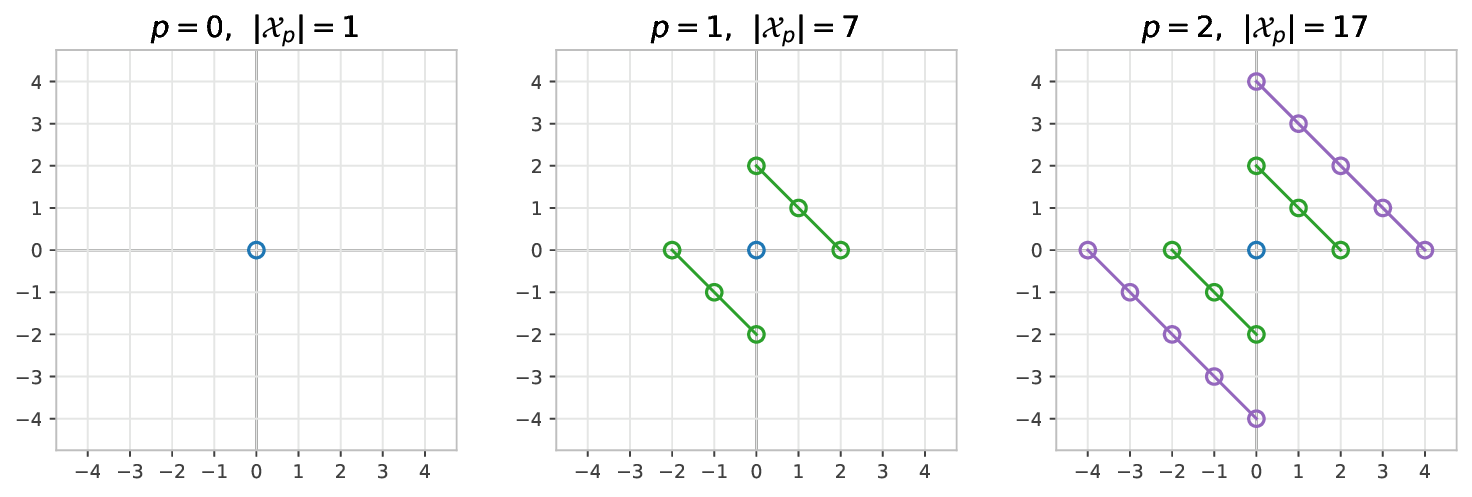}
    \caption{Correction stencils $\calX_p$ for $d=2$ and $p=0, 1, 2$ for
    anisotropic singular kernel $1/|\bx^\top B \bx|^{s/2}$.}
    \label{fig:stencil_even_central_symmetric_2d}
\end{figure}

\begin{remark}
    The stencil above is a simple and convenient choice but not unique.
    Any alternative stencil can be used provided that:
    \begin{enumerate}
        \item it is centrally symmetric and contains the origin;
        \item the number of independent weights (after enforcing
        $w_{-\bi}^h=w_{\bi}^h$) matches the number of even-total-degree
        moment constraints, i.e., $\#\calX_p^+ = N_w$;
        \item the resulting matrix for moment matching is nonsingular.
    \end{enumerate}
\end{remark}

Now we derive the matrix form of the moment-matching system. The scaling
factor remains $a(h)=h^{d-s}$ by the same homogeneity argument as in
\cref{sec:singularity_correction}. Order the even-total-degree moment
indices as $\balpha_1,\ldots,\balpha_{N_w}$ and the independent stencil
nodes in $\calX_p^+$ as $\bbeta_1,\ldots,\bbeta_{N_w}$ with the same
ordering. Then the moment conditions become
\begin{equation*}
    \sum_{k=1}^{N_w} w_{\bbeta_k}^h
    \sum_{\bi \in \{\bbeta_k, -\bbeta_k\}} \bi^{\balpha_j}\, g(\bi h)
    =
    \int_{\bbR^d} g(\bx h)\sigma_B(\bx)\bx^{\balpha_j}\,\rd\bx
    - \sum_{\infnorm{\bi}=1}^{\infty} g(\bi h)\sigma_B(\bi)\bi^{\balpha_j} := b_j^h,
\end{equation*}
for $j=1,2,\ldots,N_w$. As in \cref{sec:singularity_correction}, it is
convenient to introduce
\begin{equation} \label{eq:matrix_of_scaled_zeta}
    A_{jk} := \sum_{\bi \in \{\bbeta_k, -\bbeta_k\}} \bi^{\balpha_j},
    \qquad
    (G^h)_{jk} := g(\bbeta_k h)\delta_{jk},
\end{equation}
Then the system becomes
\begin{equation*}
    A G^h w^h = b^h.
\end{equation*}

\begin{remark}
    It is tempting to reuse the sign-orbit (radial-symmetry) grouping
    from the isotropic case. However, once the constraint set includes a
    multi-index $\balpha$ with $\onenorm{\balpha}$ even but with at
    least one odd component, the corresponding row in the coefficient
    matrix defined as $A_{jk} := \sum_{\bi \in \calG_{\bbeta_k}}
    \bi^{\balpha_j}$ becomes zero under sign-orbit summation, making the
    system singular and thus the right-hand side vector no longer lies
    in the spanned space of $A$. This is precisely why we switch to a
    merely centrally symmetric stencil in the anisotropic setting.
\end{remark}

In practice, for a given choice of stencil $\calX_p$, the nonsingularity
of the resulting moment matrix $A$ can be checked directly. Since every
entry of $A$ is an integer, $\det(A)\in\bbZ$ as well, and it can
therefore be computed exactly using integer arithmetic (thus avoiding
roundoff issues). Although exact determinant computation has cost that
grows rapidly with the matrix size, this check remains very fast for the
moderate values of $p$ used in practice.

\paragraph{Weight computation.}
The limiting weights $w_{\bi}=\lim_{h\to 0^+} w_{\bi}^h$ can be defined
by passing to the limit $h\to 0^+$, and the convergence analysis follows
similarly to \cref{sec:singularity_correction}. Furthermore, the
right-hand side can be evaluated using analytic continuation through a
generalized zeta function associated with matrix $B$.

\begin{theorem} \label{thm:scaled_zeta_function}
    For any symmetric positive definite matrix $B \in \bbR^{d\times d}$,
    define
    \begin{equation*}
        Z_{\balpha}^B(s) := \sum_{\bn \in \bbZ^{d} \setminus \{\bzero\}} \frac{\bn^{\balpha}}{(\bn^\top B \bn)^{s/2}}, 
        \quad \rep(s) > d + \onenorm{\balpha},
    \end{equation*}
    where $\onenorm{\balpha}$ is an even number. Then $Z_{\balpha}^B(s)$
    admits a meromorphic extension to the whole complex plane with at
    most a simple pole at $s = d + \onenorm{\balpha}$. Further, it
    satisfies the functional equation as
    \begin{align*}
        \frac{\Gamma(s/2)}{\pi^{s/2}} Z_{\balpha}^B(s)
        &= |B|^{-\frac{1}{2}} \left(-\frac{1}{4\pi}\right)^{\frac{\onenorm{\balpha}}{2}} 
            \frac{2H_{\balpha}(0; B^{-1})}{s - d - \onenorm{\balpha}} 
        - \frac{2\delta_{0, \balpha}}{s} 
        + \sum_{\bn\in\bbZ^d \setminus \{\bzero\}} \bn^{\balpha} 
            \int_{1}^{\infty} \re^{-\pi t \, \bn^\top B \bn} \, t^{\frac{s}{2} -1} \rd t \\
        &\quad 
        + |B|^{-\frac{1}{2}} \left(-\frac{1}{4\pi}\right)^{\frac{\onenorm{\balpha}}{2}} 
            \sum_{\bk\in \bbZ^d \setminus \{\bzero\}} \int_{1}^{\infty} 
            \re^{-\pi t \, \bk^\top B^{-1} \bk} H_{\balpha}(\sqrt{\pi t} \, \bk; B^{-1}) 
            \, t^{\frac{d+\onenorm{\balpha}-s}{2} - 1} \rd t,
    \end{align*}
    where $|B|$ is the determinant of $B$ and $H_{\balpha}(x; B^{-1})$
    is the multivariate Hermite polynomial with covariance matrix
    $B^{-1}$ defined in \cref{def:hermite_poly}.
\end{theorem}

For numerical evaluation, the two lattice series in the functional
equation are evaluated by finite summation. As in the isotropic case,
rewriting the integral terms with upper incomplete Gamma functions
reveals exponential decay in the lattice index, and the corresponding
tail bound determines a summation cutoff for a prescribed accuracy. The
anisotropic estimate follows the same argument, and the detailed
derivation is given in \cref{app:generalized_zeta_function}.

\subsection{Extension to parallelotope domain}

We now extend the SinCoTrap method to singular integrals over a general
parallelotope with the isotropic kernel $\sigma(\bx)=1/|\bx|^s$. Let
$\Vpara$ be a parallelotope  defined as
\begin{equation} \label{eq:def_parallelotope}
    \Vpara = \Big\{ \sum_{j=1}^d c_j \bv_j \mid -1 \leq c_j \leq 1, j=1, \ldots, d \Big\},
\end{equation}
where $\{\bv_1, \ldots, \bv_d\}$ are linearly independent vectors in
$\bbR^d$. Introduce a linear transformation $\bx = T \by$ such that 
\begin{equation*}
    [\bv_1, \ldots, \bv_d] = T \cdot a [\be_1, \ldots, \be_d] = T \cdot a I_d,
\end{equation*}
where $a>0$ and $I_d$ is the $d$-by-$d$ identity matrix, such a
transformation maps the hypercube $[-a, a]^d$ to the parallelotope
$\Vpara$. Then we have
\begin{equation*}
    \int_{\Vpara} \phi(\bx) \frac{1}{|\bx|^{s}} \rd \bx 
    = |\det(T)| \int_{[-a, a]^d} \phi(T \by) \frac{1}{|\by^\top T^\top T \by|^{s/2}} \rd \by.
\end{equation*}
Thus, integration over a parallelotope with an isotropic kernel reduces
to the cube integral with anisotropic kernel $1/|\by^\top B\by|^{s/2}$
where $B := T^\top T$. Conversely, given any symmetric positive definite
matrix $B$, one may choose $T$ such that $B=T^\top T$ (e.g., Cholesky
factorization), to convert the cube integral with anisotropic kernel to
an integral over parallelotope with isotropic kernel. Therefore, the two
extensions are equivalent up to a linear change of variables.

For numerical computation, assume that each coordinate direction of the
parallelotope is discretized with the same number of points $2N+1$ with
$N \in \bbN$. We construct the grid on $\Vpara$ by mapping the uniform
grid on $[-a,a]^d$ through the linear transformation $\bx=T\by$. In
particular, for $\by_{\bi}=\bi h$ where $h=a/N$, we define the
corresponding points in $\Vpara$ by $\bx_{\bi}:=T\by_{\bi}=T(\bi h)$.
The resulting quadrature rule is
\begin{equation} \label{eq:sincotrap_parallelotope}
    Q_{\text{para}}^p\!\left[\phi \cdot \frac{1}{|\bx|^s}\right]
    = |\det(T)| \Bigl(
    h^d \sum_{0<\infnorm{\bi}\le N} c_{\bi}\, \phi(\bx_{\bi}) \frac{1}{|\bx_{\bi}|^s}
    + h^{d-s} \sum_{\bi\in\calX_p} w_{\bi}\, \phi(\bx_{\bi})
    \Bigr),
\end{equation}
where $c_{\bi}$ are the standard coefficients of trapezoidal
rule and $\{w_{\bi}\}$ are the converged SinCoTrap weights (computed
once for the associated matrix $B=T^\top T$ and parameters $d,s,p$).

\subsection{Convergence analysis}

Since the two extensions are equivalent via a linear transformation, we
only present the convergence analysis for integral with singular kernel
$1/|\bx^\top B \bx|^{s/2}$ over the hypercube domain. The results and
their proofs largely parallel to those for the isotropic case, though the
anisotropy introduced by the matrix $B$ entails more intricate notation
and computations. Below, we state the main theorems, highlighting the
essential modifications required to accommodate the anisotropic setting.

We begin with the basic cancellation property implied by central
symmetry.

\begin{lemma}[Odd total-degree cancellation] \label{lem:odd_l1_cancellation}
    Given multi-index $\balpha \in \bbN^d$ with $\onenorm{\balpha}$
    being odd, and smooth radial function $g(\bx)=g(|\bx|)$ being
    compactly supported within $V=[-a, a]^d$, if the correction stencil
    $\calX_p$ and weights $\{w_{\bi}\}$ satisfy central symmetry, i.e.,
    $\bi\in\calX_p\Rightarrow -\bi\in\calX_p$ and $w_{-\bi}=w_{\bi}$,
    then for $\sigma_B(\bx) = |\bx^\top B \bx|^{-s/2}$,
    \begin{equation*}
        I[g \cdot \sigma_B \cdot \bx^{\balpha}] = 0, 
        \qquad \text{and} \qquad
        Q_h^p[g \cdot \sigma_B \cdot \bx^{\balpha}] = 0.
    \end{equation*}
\end{lemma}

\begin{proof}
    Define $F(\bx):=g(\bx)\sigma_B(\bx)\bx^{\balpha}$. The radiality of
    $g$ and the central symmetry of $\sigma_B$ imply that
    $g(-\bx)=g(\bx)$ and $\sigma_B(-\bx)=\sigma_B(\bx)$. Moreover,
    since $\onenorm{\balpha}$ is odd,
    \begin{equation*}
        (-\bx)^{\balpha}
        =(-1)^{\onenorm{\balpha}}\bx^{\balpha}
        =-\bx^{\balpha}.
    \end{equation*}
    Hence $F(-\bx)=-F(\bx)$. Because $V$ is centrally symmetric,
    \begin{equation*}
        I[g\cdot\sigma_B\cdot\bx^{\balpha}]
        = \int_V F(\bx)\,\rd\bx
        = \frac12\int_V\bigl(F(\bx)+F(-\bx)\bigr)\,\rd\bx
        = 0.
    \end{equation*}
    For the punctured trapezoidal sum, the index set
    $\{\bi\in\bbZ^d:0<\infnorm{\bi}\le N\}$ is centrally symmetric and
    $c_{-\bi}=c_{\bi}$, hence terms cancel in pairs:
    \begin{equation*}
        T_h^0[g\cdot\sigma_B\cdot\bx^{\balpha}]
        = \frac{h^d}{2}\sum_{0<\infnorm{\bi}\le N} c_{\bi}\Bigl(F(\bi h)+F(-\bi h)\Bigr)=0.
    \end{equation*}
    Likewise, central symmetry of $\calX_p$ and $w_{-\bi}=w_{\bi}$ gives
    \begin{equation*}
        h^{d-s}\sum_{\bi\in\calX_p} w_{\bi}\,g(\bi h)\,(\bi h)^{\balpha}
        = \frac{h^{d-s}}{2}\sum_{\bi\in\calX_p} w_{\bi}\Bigl(g(\bi h)(\bi h)^{\balpha}+g(-\bi h)(-\bi h)^{\balpha}\Bigr)=0.
    \end{equation*}
    Therefore $Q_h^p[g\cdot\sigma_B\cdot\bx^{\balpha}]=0$.
\end{proof}

\begin{theorem} \label{thm:error_order_anisotropic}
    Let $V=[-a, a]^d \subseteq \bbR^d$ and consider the singular
    integral
    \begin{equation*}
        I[\phi \cdot \sigma_B]
        = \int_{V} \phi(\bx)\sigma_B(\bx)\,\rd \bx,
        \qquad
        \sigma_B(\bx) = \frac{1}{|\bx^\top B \bx|^{s/2}},
        \qquad 0 < s < d,
    \end{equation*}
    where $B\in\bbR^{d\times d}$ is symmetric positive definite and
    $\phi$ is smooth. Assume that either $\phi$ is compactly supported
    in $V$ or $\phi\cdot\sigma_B$ is periodic on $V$ and is smooth
    expect at the singular points. Let $Q_h^p$ be the singularity corrected
    trapezoidal rule constructed using the centrally symmetric stencil
    and the associated converged weights:
    \begin{equation*}
        Q^p_h[\phi \cdot \sigma_B] = T^0_h[\phi \cdot \sigma_B] + \calC_{h}^p[\phi], \quad 
        \calC_{h}^p[\phi] = h^{d-s} \sum_{\bi \in \calX_p} w_{\bi} \, \phi(\bi h).
    \end{equation*}
    Then as $h\to0^+$ we have
    \begin{equation*}
        I[\phi \cdot \sigma_B] = Q^p_h[\phi \cdot \sigma_B] + O(h^{2p+2+d-s}).
    \end{equation*}
\end{theorem}

\begin{proof}
    The argument is similar to the proof of
    \cref{thm:error_order_compact_fun} after replacing
    $\sigma(\bx)=|\bx|^{-s}$ by $\sigma_B(\bx)=|\bx^\top B \bx|^{-s/2}$.
    The only change is that the componentwise odd-moment cancellation
    (\cref{lem:odd_component_cancellation}) is replaced by the odd
    total-degree cancellation (\cref{lem:odd_l1_cancellation}).
\end{proof}

\section{Numerical results} \label{sec:numerical_results}

This section validates the predicted convergence rate of SinCoTrap on a
three-dimensional test from electronic structure calculations involving
the singular kernel $\sigma(\bx)=1/|\bx|^{2}$ ($d=3$, $s=2$), which
arises in the computation of exchange energy. We report the
relative error against a high-accuracy reference value and estimate
rates from log--log slopes.

\subsection{Example 1: isotropic cube}

Consider the following singular integral over a cube domain $V=[-2,
2]^3$:
\begin{equation*}
    \int_{[-2, 2]^3} \phi(\bx) \frac{1}{|\bx|^2} \rd \bx, \quad 
    \phi(\bx) = \phi(|\bx|) = 
    \begin{cases}
        \exp(\frac{4}{|\bx|^2 - 1}), & |\bx| < 1, \\
        0, & |\bx| \geq 1.
    \end{cases}
\end{equation*}
$\phi(\bx)$ is smooth and compactly supported in domain $V$. Note that
the singularity can be canceled by a spherical transformation, i.e.,
\begin{equation} \label{eq:reference_value_example_1}
    \int_{[-2, 2]^3} \phi(\bx) \frac{1}{|\bx|^2} \rd \bx 
    = 4 \pi \int_{0}^1 \phi(r) \rd r 
    = 2 \pi \int_{-1}^1 \phi(r) \rd r.
\end{equation}
Thus, we compute the reference value by a simple 1D trapezoidal rule
with a very fine mesh.

As established in \cref{thm:error_order_compact_fun}, SinCoTrap achieves
a convergence rate of $O(h^{2p+3})$ for $d=3$, $s=2$. Our numerical
experiments employ correction orders $p=0, 1, 2$ for a sequence of mesh
sizes $h$, with expected convergence rates of $O(h^3)$, $O(h^5)$, and
$O(h^7)$, respectively. After imposing the sign symmetry
\eqref{eq:symmetric_condition}, the corresponding numbers of
independent correction weights are $N_w=1,4,10$ for $p=0,1,2$,
respectively.

The left panel of \cref{fig:convergence_order} confirms the predicted
convergence rates $O(h^3)$, $O(h^5)$, and $O(h^7)$ for $p=0,1,2$,
respectively, whereas the uncorrected rule is only first-order accurate.
These results demonstrate that SinCoTrap effectively compensates for the
singularity and recovers the predicted high-order convergence using only
a few correction weights localized near the singular point.

\subsection{Example 2: parallelotope (anisotropic extension)}

We consider the parallelotope $\Vpara$ in $\bbR^3$ defined by
\eqref{eq:def_parallelotope} with generating vectors
\begin{equation*}
    \bv_1 = \begin{pmatrix} 2 & 0 & 0 \end{pmatrix}^\top, \quad 
    \bv_2 = \begin{pmatrix} 0 & 2 & 0 \end{pmatrix}^\top, \quad 
    \bv_3 = \begin{pmatrix} 0 & 2\cos \frac{\pi}{3} & 2\sin \frac{\pi}{3} \end{pmatrix}^\top,
\end{equation*}
and evaluate the singular integral
\begin{equation*}
    \int_{\Vpara} \phi(\bx) \frac{1}{|\bx|^2} \rd \bx, \quad 
    \phi(\bx) = \phi(|\bx|) = 
    \begin{cases}
        \exp(\frac{4}{|\bx|^2 - 1}), & |\bx| < 1, \\
        0, & |\bx| \geq 1,
    \end{cases}
\end{equation*}
The support of $\phi$ is contained in the unit ball, which lies inside
$\Vpara$. Thus, the reference value is obtained via a spherical
transformation as in \eqref{eq:reference_value_example_1}. To discretize
the integral and apply the quadrature rule
\eqref{eq:sincotrap_parallelotope}, we introduce the linear
transformation $\bx=T\by$, where $(\bv_1,\bv_2,\bv_3)=T(2I)$, which maps
the reference cube $[-2,2]^3$ onto $\Vpara$. We discretize each
reference coordinate with $2N+1$ points and parametric step size
$h=\frac{2}{N}$. Thus, the horizontal axis in the right panel of
\cref{fig:convergence_order} represents the spacing $h$ in the reference
coordinates $\by$; the corresponding physical grid increments are
$T(h\be_j)$.

For $d=3$ and $s=2$, \cref{thm:error_order_anisotropic} predicts an
error of order $O(h^{2p+3})$. We again use $p=0,1,2$, for which the
expected rates are $O(h^3)$, $O(h^5)$, and $O(h^7)$, respectively.
Because the transformed kernel is anisotropic, the moment system
includes every multi-index of even total degree up to $2p$. The
corresponding numbers of independent correction weights are $N_w=1,7,22$
for $p=0,1,2$, respectively.

\begin{figure}[htbp]
    \centering
    \subfigure[Isotropic kernel in 3D.]{
        \includegraphics[width=0.45\textwidth]{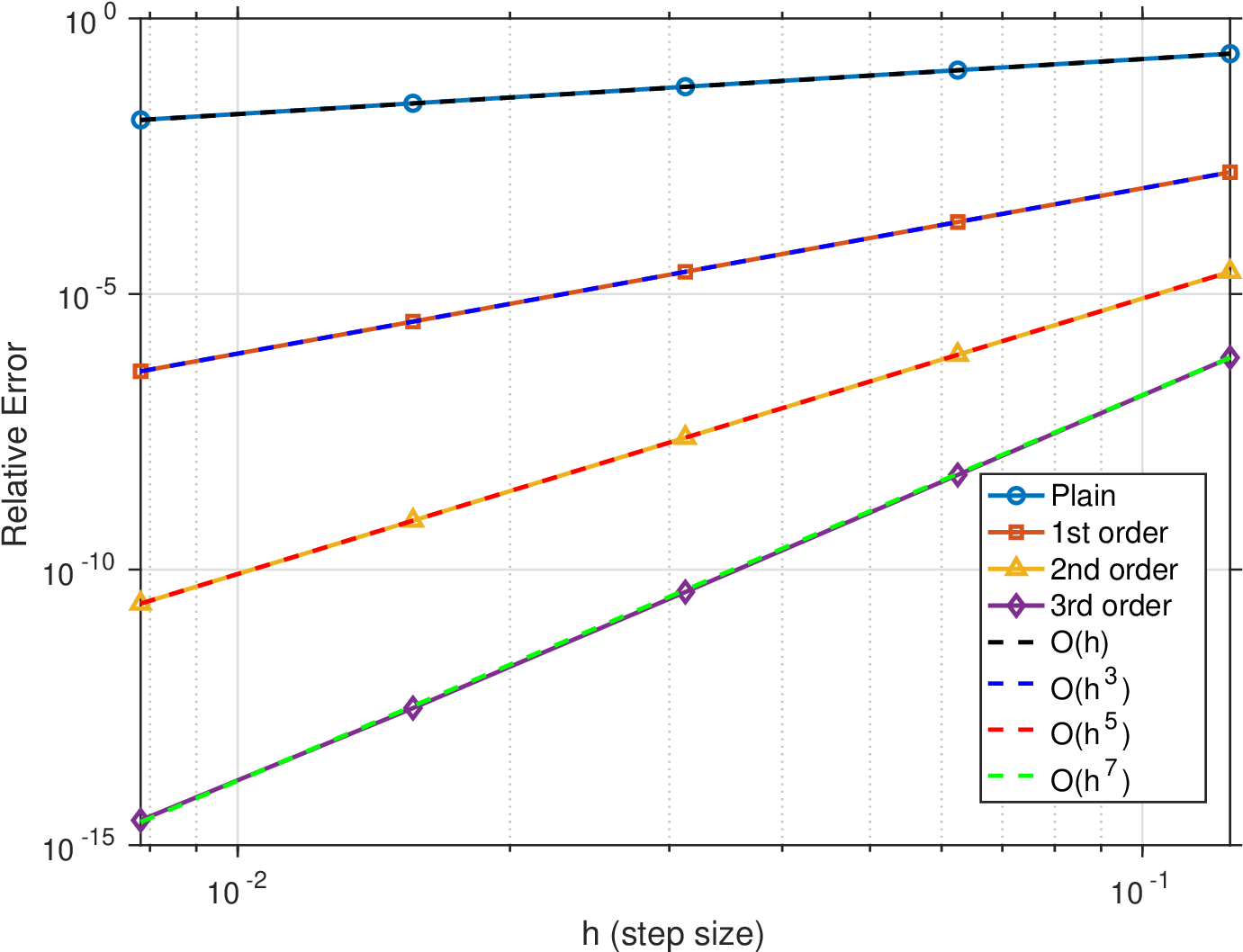}
        \label{fig:convergence_order_isotropic}
    }
    \hfill
    \subfigure[Anisotropic kernel (parallelotope case) in 3D.]{
        \includegraphics[width=0.45\textwidth]{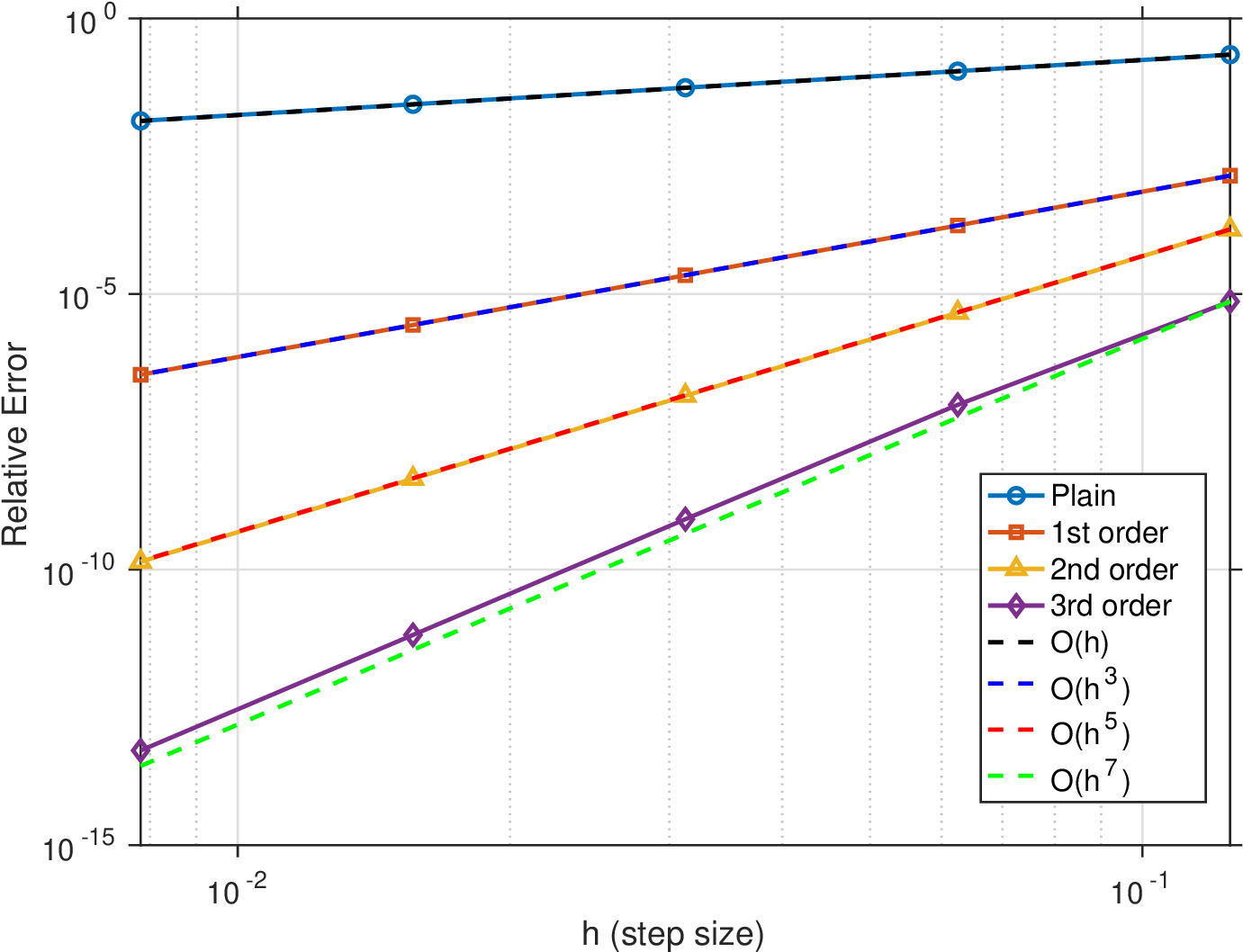}
        \label{fig:convergence_order_anisotropic}
    }
    \caption{Convergence rates for SinCoTrap with (left) isotropic and
    (right) anisotropic singular kernels, demonstrating the expected
    $O(h^{2p+3})$ scaling for $p=0,1,2$ (here $d=3$, $s=2$). In the
    legends, ``1st order,'' ``2nd order,'' and ``3rd order'' refer to
    the correction levels $p=0,1,2$, respectively.}
    \label{fig:convergence_order}
\end{figure}

The right panel of \cref{fig:convergence_order} confirms the predicted
convergence rates $O(h^3)$, $O(h^5)$, and $O(h^7)$ for $p=0,1,2$,
respectively, whereas the uncorrected rule is only first-order accurate.
Thus, in the anisotropic setting, SinCoTrap retains its designed
high-order convergence using only a few correction weights localized
near the singular point.

\section{Conclusion and discussion} \label{sec:conclusion}

In this paper, we have developed a locally singularity corrected
trapezoidal rule (SinCoTrap) for evaluating singular integrals in
arbitrary dimensions, with rigorous theoretical analysis and an
efficient computational scheme. Our method preserves the simplicity of
the standard trapezoidal rule while achieving high-order accuracy
through local corrections applied in a small neighborhood of the
singularity. Both theoretical proofs and numerical experiments confirm
the scheme's effectiveness, demonstrating significantly improved
convergence rates. A key insight of this work is the connection between
the correction weights and generalized zeta functions, which elucidates
the underlying mathematical structure of the method.

While our analysis assumes the integrand is either compactly supported
or periodic, the quadrature rule can be applied to more general
functions. In such cases, to maintain high-order accuracy, it must be
combined with a boundary correction scheme \cite{fornberg2021improving,
kapur1997high} for the trapezoidal rule.

Several directions for future research emerge from this work. First,
extension to more general singular kernels beyond the power-law form
$1/|\bx|^s$ would broaden the applicability of the method. Second,
investigation of optimal stencil shapes and sizes for specific
applications could further improve computational efficiency. Third,
generalization to nearly singular integrals, where the singularity is
close to but not within the integration domain, presents an interesting
challenge. Finally, exploring applications in various scientific and
engineering fields, such as computational chemistry, fluid dynamics and
electromagnetics, could demonstrate the practical utility of the
proposed quadrature scheme.

\clearpage

\appendix

\section*{Appendices}
\addcontentsline{toc}{section}{Appendices}

This appendix contains the proofs and technical derivations deferred
from the main text, including the convergence order analysis, proofs of
the nonsingularity of the moment-matching matrices, properties of the
right-hand sides of the moment-matching system, and a discussion of the
generalized zeta functions.

\section{Preliminaries for the Appendix}
\label{app:preliminaries}

For reference, we fix the Fourier transform convention 
\begin{equation*}
    \mathcal{F}[f](\bxi) = \hf(\bxi) 
    := \int_{\bbR^d} f(\bx)\, \re^{-2\pi\mi \,\bx \cdot \bxi}\, \rd \bx.
\end{equation*}
We will repeatedly use the Poisson summation formula, which connects
lattice sums with Fourier series and is central to both the error
analysis and the analytic continuation of the zeta functions.


\begin{lemma}[Poisson summation \cite{grafakos2008classical}]
    \label{lem:poisson_summation}
    Let $f$ be a function in Schwartz space $\mathcal{S}(\bbR^d)$. Then
    we have
    \begin{equation*}
        \sum_{\bi \in \bbZ^d} f(\bi) = \sum_{\bk \in \bbZ^d} \hf(\bk).
    \end{equation*}
\end{lemma}


\section{Proofs for convergence order analysis}

\begin{lemma} \label{lem:linfty_shell_count}
    Given $n\in\bbN_+$, define the $\ell_\infty$-shell $S_n :=
    \{\bk\in\bbZ^d : \infnorm{\bk}=n\}$. Then the number of lattice
    points in $S_n$ is
    \footnote{Given any set $A$, we use $\# A$ to denote the number of elements in $A$.}
    \begin{equation*}
        \#S_n = (2n+1)^d-(2n-1)^d \le 2d\,3^{d-1}\,n^{d-1}.
    \end{equation*}
\end{lemma}

\begin{proof}
    The number of lattice points $\#S_n$ is the difference of number of
    lattice points in nested cubes, i.e.,
    \begin{equation*}
        \#\{\bk\in\bbZ^d:\infnorm{\bk}\le n\}=(2n+1)^d,
        \quad
        \#\{\bk\in\bbZ^d:\infnorm{\bk}\le n-1\}=(2n-1)^d.
    \end{equation*}
    For the upper bound, apply the mean value theorem to $f(x)=x^d$ on
    the interval $(2n-1,2n+1)$, then there exists $\xi\in(2n-1,2n+1)$
    such that
    \begin{equation*}
        (2n+1)^d-(2n-1)^d = 2 f'(\xi) = 2d\,\xi^{d-1} 
        \leq 2d(2n+1)^{d-1} \leq 2d\,3^{d-1}\,n^{d-1}.
    \end{equation*}
\end{proof}

\begin{lemma} \label{lem:drv_x_alpha_s} 
    For $q\in \bbN$, $s\in \bbC \setminus \{0\}$, $\balpha \in \bbN^d$
    and $0 \neq \bx = (x_1, \ldots, x_d) \in \bbR^d$, we have
    \begin{equation*}
        \frac{\partial^q}{\partial x_j^q} \left( \frac{\bx^{2\balpha}}{|\bx|^s} \right) 
        = \frac{1}{|\bx|^{q+s-2\onenorm{\balpha}}} P_{q, \balpha,j}\left( \frac{\bx}{|\bx|} \right), 
        \quad j=1, \ldots, d,
    \end{equation*}
    where $P_{q, \balpha,j}(\bx)$ is a polynomial with degree
    $\mathrm{deg} P_{q, \balpha,j} \leq q + 2\onenorm{\balpha}$ in
    variable $\bx$. Furthermore, the coefficients of $P_{q,
    \balpha,j}(\bx)$ depend polynomially on the parameter $s$.
\end{lemma}

\begin{proof}
    Since $|\bx|>0$, for complex $s$ we have
    $|\bx|^{-s}:=\exp(-s\log|\bx|)$, where $\log|\bx|$ is real, so no
    choice of complex logarithm branch is involved. Differentiation with
    respect to the real variable $x_j$ therefore gives 
    \begin{equation*}
        \frac{\partial}{\partial x_j}|\bx|^{-s}
        =-s x_j|\bx|^{-s-2},
    \end{equation*}
    which is the same formula as for real $s$.

    We complete the proof by induction on $q$. For $q=0$,
    \begin{equation*}
        \frac{\bx^{2\balpha}}{|\bx|^s} 
        = \frac{1}{|\bx|^{s-2\onenorm{\balpha}}} \left( \frac{\bx}{|\bx|} \right)^{2\balpha}, 
        \quad P_{0, \balpha,j} (\bx) = \bx^{2\balpha}.
    \end{equation*}
    For the inductive step, we assume the result holds for $q$ and
    differentiate again:
    \begin{align*}
        \frac{\partial^{q+1}}{\partial x_j^{q+1}} \left( \frac{\bx^{2\balpha}}{|\bx|^s} \right)
        &= \frac{\partial}{\partial x_j} \left( \frac{1}{|\bx|^{s + q - 2\onenorm{\balpha}}} P_{q, \balpha,j}\left( \frac{\bx}{|\bx|} \right) \right) \\
        &= \frac{-(s + q - 2\onenorm{\balpha})}{|\bx|^{s + q - 2\onenorm{\balpha} + 1}} 
            \, \frac{x_j}{|\bx|} \, P_{q, \balpha,j}\left( \frac{\bx}{|\bx|} \right) \\
        &\quad + \frac{1}{|\bx|^{s + q - 2\onenorm{\balpha}}} \left[
            P^{(j)}_{q, \balpha,j}\left(\frac{\bx}{|\bx|}\right) \left(\frac{1}{|\bx|} - \frac{x_j^2}{|\bx|^3} \right) 
            - \sum_{l\neq j} P^{(l)}_{q, \balpha,j}\left(\frac{\bx}{|\bx|}\right) \frac{x_jx_l}{|\bx|^3} 
        \right] \\
        &:= \frac{1}{|\bx|^{s + q + 1 - 2\onenorm{\balpha}}} P_{q+1, \balpha,j}\left( \frac{\bx}{|\bx|} \right),
    \end{align*}
    where $P^{(l)}_{q,\balpha,j}(\bx)=\frac{\partial}{\partial x_l}
    P_{q,\balpha,j}(\bx)$ for $l=1, \ldots,d$ are polynomials with $\deg
    P^{(l)}_{q,\balpha,j} \leq \deg P_{q,\balpha,j} - 1$, and
    \begin{equation*}
        P_{q+1, \balpha,j}(\bx) 
        := -(s + q - 2\onenorm{\balpha}) x_j P_{q, \balpha,j}\left(\bx\right) 
        + P^{(j)}_{q, \balpha,j}\left(\bx\right)\left(1 - x_j^2\right) 
        - \sum_{l\neq j} P^{(l)}_{q, \balpha,j}\left(\bx\right) x_j x_l
    \end{equation*}
    is a polynomial with degree at most $q + 2\onenorm{\balpha} + 1$ in
    variable $\bx = (x_1, \ldots, x_d)$.
\end{proof}

\begin{proof}[Proof of \cref{lem:error_order_rhs}]
    We begin by introducing a smooth cutoff function $\eta(\bx) \in
    C^{\infty}(\bbR^d)$ such that
    \begin{equation} \label{eq:cutoff_rhs}
        \eta(\bx) = \eta(|\bx|), \quad \eta(\bx) = \begin{cases}
            0, & |\bx| \leq \frac{1}{2}, \\
            1, & |\bx| \geq 1.
        \end{cases}
    \end{equation}
    Then we rewrite the difference of the integral and summation as:
    \begin{align*}
        & \quad \int_{\bbR^d} g(\bx h) \frac{\bx^{2\balpha}}{|\bx|^s} \,\rd \bx 
            - \sum_{\infnorm{\bi}=1}^{\infty} g(\bi h) \frac{\bi^{2\balpha}}{|\bi|^s} \\
        &= \int_{|\bx| \leq 1} g(\bx h) \frac{\bx^{2\balpha}}{|\bx|^s} (1 - \eta(\bx)) \,\rd \bx 
            + \left( \int_{\bbR^d} g(\bx h) \frac{\bx^{2\balpha}}{|\bx|^s} \eta(\bx) \,\rd \bx 
                - \sum_{\infnorm{\bi}=1}^{\infty} g(\bi h) \frac{\bi^{2\balpha}}{|\bi|^s} \eta(\bi) 
            \right) \\
        &:= {\rm I} + {\rm II}.
    \end{align*}

    \paragraph{Estimate of ${\rm I}$.}
    Note that $\partial^{\bbeta} g(0) = 0$ for $1 \leq \onenorm{\bbeta}
    \leq 2p+1$. Moreover, for multi-index $\onenorm{\bbeta}=2p+2$ and
    constant $\theta \in [0, 1]$, 
    \begin{equation*}
        \int_{|\bx| \leq 1} \Big| \frac{\partial^{\bbeta} g(\theta \bx h)}{\bbeta!} (\bx h)^{\bbeta} \frac{\bx^{2\balpha}}{|\bx|^s} (1 - \eta(\bx)) \Big| \,\rd \bx 
        \leq C h^{2p+2} \int_{|\bx|\leq 1} |\bx|^{2\onenorm{\balpha} + 2p + 2 - \rep(s)} \,\rd \bx
        \leq C h^{2p+2}, 
    \end{equation*}
    where $0 < \rep(s) < d + 2\onenorm{\balpha}$ guarantees the
    convergence of the integral. Thus, the Taylor expansion of $g(\bx
    h)$ around $\bx=0$ gives
    \begin{equation*}
        {\rm I} = \int_{|\bx| \leq 1} g(\bx h) \frac{\bx^{2\balpha}}{|\bx|^s} (1 - \eta(\bx)) \,\rd \bx 
        = g(0) \int_{|\bx|\leq 1} \frac{\bx^{2\balpha}}{|\bx|^s} (1 - \eta(\bx)) \,\rd \bx + O(h^{2p+2}).
    \end{equation*}   
    
    \paragraph{Estimate of ${\rm II}$.} 
    Define
    \begin{equation} \label{eq:def_fh_rhs}
        f_h(\bx) := g(\bx h) \frac{\bx^{2\balpha}}{|\bx|^s} \eta(\bx),
    \end{equation}
    then $f_h(0)=0$ and $f_h \in \ccinf(\bbR^d)$ for $h > 0$. Apply the
    Poisson summation formula in \cref{lem:poisson_summation} we obtain
    \begin{align*}
        {\rm II} 
        = \int_{\bbR^d} f_h(\bx) \,\rd \bx - \sum_{\infnorm{\bi}=1}^{\infty} f_h(\bi) 
        = \widehat{f_h}(0) - \sum_{\infnorm{\bi}=0}^{\infty} f_h(\bi)
        = \widehat{f_h}(0) - \sum_{\infnorm{\bk}=0}^{\infty} \widehat{f_h}(\bk)
        = - \sum_{\infnorm{\bk}=1}^{\infty} \widehat{f_h}(\bk).
    \end{align*}

    Choose integer 
    \begin{equation*}
        q = 2p+3+d+2\onenorm{\balpha}.
    \end{equation*}
    For each $\bk \in \bbZ^d \setminus \{0\}$, let index $j = j(\bk) \in 
    \text{argmax}_{\ell=1, \ldots, d} |k_{\ell}|$ and hence $|k_j| =
    \infnorm{\bk}$. Then we have the following Fourier estimate for
    $h>0$ sufficiently small, proved in \cref{lem:estimate_fourier_fh}
    below:
    \begin{equation*}
        \widehat{f_h}(\bk)
        = g(0)\frac{W(\bk)}{(2\pi k_j)^q} + R_h(\bk),
        \qquad
        |R_h(\bk)|\le C h^{2p+2} \frac{1}{|k_j|^q},
        \qquad
        |W(\bk)|\le C.
    \end{equation*}
    Furthermore, for $n\in \bbN_+$, let $S_n := \{\bk \in \bbZ^d :
    \infnorm{\bk} = n\}$, by \cref{lem:linfty_shell_count} the number of
    elements in $S_n$ satisfies $\# S_n \leq C_d n^{d-1}$ for some
    constant $C_d>0$ depending only on the dimension $d$. Thus, whenever
    $q > d$, we have
    \begin{equation*}
        \sum_{\infnorm{\bk}=1}^{\infty} \frac{1}{|k_j|^q} 
        = \sum_{n=1}^{\infty} \sum_{\bk \in S_n} \frac{1}{n^q} 
        \leq C_d \sum_{n=1}^{\infty} \frac{1}{n^{q-d+1}} < \infty.
    \end{equation*}
    Therefore, 
    \begin{equation*}
        {\rm II} 
        = - \sum_{\infnorm{\bk}=1}^{\infty} \widehat{f_h}(\bk) 
        = - g(0) \sum_{\infnorm{\bk}=1}^{\infty} \frac{W(\bk)}{(2\pi k_j)^q} + O(h^{2p+2}).
    \end{equation*}

    \paragraph{Conclusion.}
    Combining the estimates for ${\rm I}$ and ${\rm II}$ yields
    \begin{equation*}
        {\rm I} + {\rm II} = g(0)b_{\balpha}(s) + O(h^{2p+2}), \quad 
        b_{\balpha}(s)
        := \int_{|\bx|\le 1}\frac{\bx^{2\balpha}}{|\bx|^s}(1-\eta(\bx))\,\rd\bx
        - \sum_{\infnorm{\bk}=1}^{\infty} \frac{W(\bk)}{(2\pi k_j)^q}.
    \end{equation*}
\end{proof}

\begin{lemma} \label{lem:estimate_fourier_fh}
    Let $p\in\bbN$, $\balpha\in\bbN^d$ with $0\leq\onenorm{\balpha}\leq
    p$, and $s\in\bbC$ with $0<\rep(s)<d+2\onenorm{\balpha}$. Assume
    that $g\in\ccinf(\bbR^d)$ satisfies
    \begin{equation*}
        \partial^{\bbeta}g(0)=0,
        \qquad 1\leq\onenorm{\bbeta}\leq 2p+1.
    \end{equation*}
    Let $\eta$ be the cutoff function defined in \eqref{eq:cutoff_rhs},
    and let $q\in\bbN$ satisfy $q\geq 2p+3+d+2\onenorm{\balpha}$. For
    each $\bk\in\bbZ^d\setminus\{0\}$, choose any index $j=j(\bk)$ in
    $\text{argmax}_{\ell=1,\ldots,d}|k_\ell|$, so that
    $|k_j|=\infnorm{\bk}$, and define
    \begin{equation*}
        W(\bk) := (-\mi)^q \int_{\bbR^d} 
            \partial^q_j \left( \frac{\bx^{2\balpha}}{|\bx|^s} \eta(\bx) \right) 
            \re^{-2\pi \mi \, \bk \cdot \bx} \,\rd \bx.
    \end{equation*}
    Then there exists a constant $C>0$, independent of $h$ and $\bk$,
    such that $|W(\bk)|\leq C$ and, for all sufficiently small $h>0$,
    the function $f_h$ defined in \eqref{eq:def_fh_rhs} satisfies
    \begin{equation*}
        \widehat{f_h}(\bk)
        = g(0)\frac{W(\bk)}{(2\pi k_j)^q} + R_h(\bk),
        \qquad
        |R_h(\bk)|\le C h^{2p+2} \, \frac{1}{|k_j|^q}.
    \end{equation*}
\end{lemma}

\begin{proof}
    Since $f_h\in \ccinf(\bbR^d)$, after integration by parts $q$ times
    in the $x_j$ variable we have 
    \begin{equation*}
        \widehat{f_h}(\bk) 
        = \int_{\bbR^d} f_h(\bx)\, \re^{-2\pi \mi \, \bk \cdot \bx} \,\rd \bx 
        = \left( \frac{1}{2\pi \mi k_j} \right)^q 
        \int_{\bbR^d} \partial^q_j \left( g(\bx h) \frac{\bx^{2\balpha}}{|\bx|^s} \eta(\bx) \right) 
            \re^{-2\pi \mi \, \bk \cdot \bx} \,\rd \bx.
    \end{equation*}

    \paragraph{Bound of $W(\bk)$.}
    By definition of $W(\bk)$, note that $\eta(\bx)=0$ for $|\bx|\leq
    \frac{1}{2}$ and $\eta(\bx)=1$ for $|\bx|\geq 1$, we have
    \begin{equation*}
        |W(\bk)|
        \leq \int_{\bbR^d}\left|\partial_j^q\!\left(\frac{\bx^{2\balpha}}{|\bx|^s}\eta(\bx)\right)\right|\,\rd\bx 
        \leq \int_{\frac{1}{2} \leq |\bx| \leq 1} \left|\partial_j^q\!\left(\frac{\bx^{2\balpha}}{|\bx|^s}\eta(\bx)\right)\right|\,\rd\bx 
        + \int_{|\bx| \geq 1} \left|\partial_j^q\!\left(\frac{\bx^{2\balpha}}{|\bx|^s}\right)\right|\,\rd\bx.
    \end{equation*}
    On the annulus $\{\frac12\leq|\bx|\leq1\}$, the integrand is smooth
    and hence integrable. For $|\bx|\geq1$, \cref{lem:drv_x_alpha_s} and
    the boundedness of $P_{q,\balpha,j}$ on the unit sphere imply
    \begin{equation*}
        \left|\partial_j^q\!\left(\frac{\bx^{2\balpha}}{|\bx|^s}\right)\right|
        \leq C|\bx|^{-q-\rep(s)+2\onenorm{\balpha}},
        \qquad |\bx|\geq1.
    \end{equation*}
    Here $C$ is chosen independently of $j$, since $j\in\{1,\ldots,d\}$.
    Consequently,
    \begin{equation*}
        \int_{|\bx|\geq1}
        \left|\partial_j^q\!\left(\frac{\bx^{2\balpha}}{|\bx|^s}\right)\right|\,\rd\bx
        \leq C\int_1^\infty
        r^{-q-\rep(s)+2\onenorm{\balpha}+d-1}\,\rd r<\infty,
    \end{equation*}
    because $q+\rep(s)>d+2\onenorm{\balpha}$. Thus we have $|W(\bk)|\leq C$.

    \paragraph{Estimate of the remainder $R_h(\bk)$.}
    Define 
    \begin{equation*}
        \Delta_h(\bx):=\bigl(g(h\bx)-g(0)\bigr)\eta(\bx)\frac{\bx^{2\balpha}}{|\bx|^s}.
    \end{equation*}
    Then
    \begin{equation*}
        \widehat{f_h}(\bk) - g(0)\frac{W(\bk)}{(2\pi k_j)^q}
        = \left(\frac{1}{2\pi \mi k_j}\right)^q
        \int_{\bbR^d} \partial_j^q \Delta_h(\bx)\,\re^{-2\pi \mi \bk\cdot \bx}\,\rd\bx
        =: R_h(\bk),
    \end{equation*}
    and hence
    \begin{equation*}
        |R_h(\bk)|
        \le (2\pi)^{-q}|k_j|^{-q}\,\|\partial_j^q\Delta_h\|_{L^1(\bbR^d)}.
    \end{equation*}
    It remains to show $\|\partial_j^q\Delta_h\|_{L^1(\bbR^d)}\le
    Ch^{2p+2}$. Since $\eta(\bx)=0$ for $|\bx|\le 1/2$, we only
    integrate over $|\bx|\ge 1/2$. Furthermore, fix $R>0$ such that
    $\supp(g(\bx))\subset\{\bx:|\bx| < R\}$, so
    $\supp(g(h\bx))\subset\{\bx:|\bx| < R/h\}$. 

    By Leibniz' rule,
    \begin{equation*}
        \partial_j^q \Delta_h(\bx)
        = \sum_{m=0}^{q}\binom{q}{m}\,
        \partial_j^{m}\bigl(g(h\bx)-g(0)\bigr)\,
        \partial_j^{\,q-m}\!\left(\eta(\bx)\frac{\bx^{2\balpha}}{|\bx|^s}\right).
    \end{equation*}
    We prove $\|\partial_j^q\Delta_h\|_{L^1(\bbR^d)}\le Ch^{2p+2}$ in
    three steps. We first give bounds for $\partial_j^{m}(g(h\bx)-g(0))$
    and
    $\partial_j^{q-m}\!\left(\eta(\bx)\bx^{2\balpha}/|\bx|^s\right)$,
    then combine these bounds to obtain the $L^1$ estimate of
    $\partial_j^q\Delta_h$.    

    \paragraph{Step 1: bounds for $\partial_j^{m}(g(h\bx)-g(0))$.} 
    Let $m\in\{0,1,\ldots,q\}$. Since $\partial^{\bbeta}g(0)=0$ for
    $1\le \onenorm{\bbeta}\le 2p+1$, Taylor's theorem gives, for $m=0$,
    we have $|g(h\bx)-g(0)| \le C (h|\bx|)^{2p+2}$; for
    $m=1,\ldots,2p+1$, we have 
    \begin{equation*}
        \left|\partial_j^m\bigl(g(h\bx)-g(0)\bigr)\right|
        = h^m |(\partial_j^m g)(h\bx) - (\partial_j^m g)(0)|
        \le C h^m (h|\bx|)^{2p+2-m}
        = C h^{2p+2}|\bx|^{2p+2-m}.
    \end{equation*}
    Further, for $2p+2 \leq m \leq q$, the uniform boundedness of
    $\partial_j^m g$ and the chain rule give
    \begin{equation*}
        \left|\partial_j^m\bigl(g(h\bx)-g(0)\bigr)\right|
        =h^m|(\partial_j^m g)(h\bx)|\leq Ch^m.
    \end{equation*}
    Note that on the region $\{|\bx|\le R/h\}$, $h|\bx| \leq R$ implies
    \begin{equation*}
        h^m
        = h^{2p+2}(h|\bx|)^{m-(2p+2)}|\bx|^{2p+2-m}
        \le C_R h^{2p+2}|\bx|^{2p+2-m}.
    \end{equation*}
    Therefore, we conclude that for all $m=0,1,\ldots,q$,
    \begin{equation*}
        \left|\partial_j^{m}\bigl(g(h\bx)-g(0)\bigr)\right|
        \le C h^{2p+2}|\bx|^{2p+2-m}, 
        \qquad \frac{1}{2}\le |\bx|\le R/h.
    \end{equation*}

    \paragraph{Step 2: bounds for $\partial_j^{q-m}\!\left(\eta(\bx)\bx^{2\balpha}/|\bx|^s\right)$.}
    For $|\bx|\ge 1$, we have $\eta(\bx)=1$ and thus
    \cref{lem:drv_x_alpha_s}, together with the boundedness of
    $P_{q-m,\balpha,j}$ on the unit sphere, gives
    \begin{equation*}
        \left|\partial_j^{\,q-m}\!\left(\eta(\bx)\frac{\bx^{2\balpha}}{|\bx|^s}\right)\right|
        = \left|\partial_j^{\,q-m}\!\left(\frac{\bx^{2\balpha}}{|\bx|^s}\right)\right|
        \le C |\bx|^{2\onenorm{\balpha}-\rep(s)-(q-m)},\qquad |\bx|\ge 1,
    \end{equation*}
    where $C$ is chosen uniformly for $m\in\{0,\ldots,q\}$ and
    $j\in\{1,\ldots,d\}$. On the annulus $\{1/2\le |\bx|\le 1\}$, the
    function $\eta(\bx)\bx^{2\balpha}/|\bx|^s$ is smooth and the
    integral domain is bounded, hence
    \begin{equation*}
        \left|\partial_j^{\,q-m}\!\left(\eta(\bx)\frac{\bx^{2\balpha}}{|\bx|^s}\right)\right|
        \leq C \leq C |\bx|^{2\onenorm{\balpha}-\rep(s)-(q-m)}, \qquad \frac12\le |\bx|\le 1.
    \end{equation*}
    Therefore, we obtain the uniform bound for
    $\partial_j^{q-m}\!\left(\eta(\bx)\bx^{2\balpha}/|\bx|^s\right)$
    over $|\bx|\ge 1/2$.

    \paragraph{Step 3: $L^1$ estimate of $\partial_j^q\Delta_h$.}
    For $h > 0$ sufficiently small, we have $R/h > 1$.
    Using Steps 1, 2 and splitting $\bbR^d$ into  
    $\{\frac{1}{2} \leq |\bx|\leq R/h\}$ and $\{|\bx| \geq R/h\}$, we obtain
    \begin{align*}
        \left\| \partial^q_j \Delta_h \right\|_{L^1(\bbR^d)} 
        = \left( \int_{\frac{1}{2} \leq |\bx| \leq R/h} 
        + \int_{|\bx| \geq R/h} \right) \left| \partial^q_j \Delta_h(\bx) \right| \rd \bx 
        := T_1 + T_2.
    \end{align*}
    For the first part, we have
    \begin{align*}
        T_{1}
        &= \int_{\frac{1}{2} \leq |\bx| \leq R/h} \Big| \sum_{m=0}^q \binom{q}{m}\,
        \partial_j^{m}\bigl(g(h\bx)-g(0)\bigr)\,
        \partial_j^{\,q-m}\!\left(\eta(\bx)\frac{\bx^{2\balpha}}{|\bx|^s}\right) \Big| \, \rd \bx \\
        &\le C h^{2p+2}\sum_{m=0}^{q}\binom{q}{m}
        \int_{\frac{1}{2} \le |\bx|\le R/h}
        |\bx|^{2p+2-m}\,|\bx|^{2\onenorm{\balpha}-\rep(s)-(q-m)} \rd\bx \\
        &\le C h^{2p+2}
        \int_{\frac{1}{2}}^{R/h} r^{2p+2+2\onenorm{\balpha}-\rep(s)-q+d-1}\,\rd r
        \le C h^{2p+2},
    \end{align*}
    where the last inequality uses $q\ge 2p+3+d+2\onenorm{\balpha}$, so
    the exponent of $r$ is at most $-2-\rep(s)<-2$.

    For $T_2$, note that $g(h\bx)=0$ when $|\bx|\ge R/h$, hence
    $\Delta_h(\bx)=-g(0)\eta(\bx)\bx^{2\balpha}/|\bx|^s$ there, and
    since $\eta(\bx)=1$ for $|\bx|\ge 1$ and $R/h > 1$ for $h>0$
    sufficiently small,
    \begin{equation*}
        T_2
        = |g(0)| \int_{|\bx|\ge R/h}\left|\partial_j^q\left(\frac{\bx^{2\balpha}}{|\bx|^s}\right)\right|\rd\bx
        \le C \int_{R/h}^{\infty} r^{-q-\rep(s)+2\onenorm{\balpha}+d-1}\,\rd r
        \le C h^{2p+2}.
    \end{equation*}
    Therefore $\|\partial_j^q\Delta_h\|_{L^1(\bbR^d)}\le C h^{2p+2}$,
    which completes the proof.
\end{proof}

\section{Proofs of non-singularity of moment-matching matrix}
\label{app:nonsingularity_proof}

\begin{proof}[Proof of \cref{lem:non_singular_coeff}]
    Recall that
    \begin{equation*}
        A_{jk}=\sum_{\bi\in\calG_{\bbeta_k}}\bi^{2\balpha_j}.
    \end{equation*}
    Since the exponents are even, $\bi^{2\balpha_j}$ is constant on the 
    sign-orbit $\calG_{\bbeta_k}$, hence
    \begin{equation*}
        A_{jk}=c(\bbeta_k)\,\bbeta_k^{2\balpha_j},\qquad
        c(\bbeta):=\#\calG_{\bbeta}=2^{r(\bbeta)},\,\, 
        r(\bbeta):=\#\{\ell:\beta_\ell\neq 0\}.
    \end{equation*}
    Therefore $A=MD$, where $D=\mathrm{diag}(c(\bbeta_k))$ has strictly 
    positive diagonal entries and
    \begin{equation*}
        M_{jk}=\bbeta_k^{2\balpha_j}.
    \end{equation*}
    Since $D$ is invertible, it suffices to show that $M$ is nonsingular.

    For $m\in\bbN$, define the even 1D polynomials
    \begin{equation*}
        q_m(x) := \prod_{r=0}^{m-1} (x^2-r^2) = x^2(x^2-1^2)\cdots(x^2-(m-1)^2),
        \qquad q_0(x):=1.
    \end{equation*}
    and for $\bgamma\in\bbN^d$ define the tensor-product polynomials
    \begin{equation*}
        Q_{\bgamma}(\bx):=\prod_{\ell=1}^d q_{\gamma_\ell}(x_\ell)
        = \prod_{\ell=1}^d \prod_{r=0}^{\gamma_\ell-1} (x_\ell^2-r^2).
    \end{equation*}
    For $\bbeta\in\bbN^d$, we have $Q_{\bgamma}(\bbeta)=0$ whenever
    $\beta_\ell<\gamma_\ell$ for some $\ell$ (since then the factor with
    $r=\beta_\ell$ appears and vanishes). In particular,
    \begin{equation*}
        Q_{\bgamma}(\bbeta)=0 \quad \text{unless } \bgamma\le \bbeta,
    \end{equation*}
    where $\bgamma\le \bbeta$ means componentwise inequality.  
    Moreover, for $\onenorm{\bbeta} > 0$,
    \begin{equation*}
        Q_{\bbeta}(\bbeta)
        =\prod_{\ell=1}^d\prod_{r=0}^{\beta_\ell-1}(\beta_\ell^2-r^2)\neq 0,
        \quad \text{and} \quad 
        Q_{\bzero}(\bzero) = 1.
    \end{equation*}
    Order the index set $\{\bgamma\in\bbN^d:0\le \onenorm{\bgamma}\le p\}$ 
    by any linear extension of the partial order $\bgamma\le\bbeta$ 
    (componentwise). Then the evaluation matrix
    \begin{equation*}
        E_{\bgamma,\bbeta}:=Q_{\bgamma}(\bbeta)
    \end{equation*}
    is upper triangular with nonzero diagonal, hence $E$ is invertible.

    Now we express $Q_{\bgamma}(\bx)$ in the monomial basis.
    Since $q_m(x)$ is an even polynomial of degree $2m$ with leading 
    coefficient $1$, it admits an expansion
    \begin{equation*}
        q_m(x)=\sum_{r=0}^{m}l_{m,r}\,x^{2r},\qquad l_{m,m}=1.
    \end{equation*}
    Consequently, for every $\bgamma\in\bbN^d$ we can expand the tensor
    product as
    \begin{equation*}
        Q_{\bgamma}(\bx)
        =\prod_{\ell=1}^d\Bigl(\sum_{r=0}^{\gamma_\ell}l_{\gamma_\ell,r}\,x_\ell^{2r}\Bigr)
        =\sum_{\bdelta\le \bgamma} L_{\bgamma,\bdelta}\,\bx^{2\bdelta},
        \qquad
        L_{\bgamma,\bdelta}:=\prod_{\ell=1}^d l_{\gamma_\ell,\delta_\ell}.
    \end{equation*}
    Specifically, we have $L_{\bgamma,\bgamma}=1$. Furthermore, we
    extend this notation by setting $L_{\bgamma,\bdelta}=0$ whenever
    $\delta_\ell>\gamma_\ell$ for at least one component $\ell$. Under
    the same ordering for $\{\bgamma\in\bbN^d:0\le \onenorm{\bgamma}\le
    p\}$ as above, the matrix $L=[L_{\bgamma,\bdelta}]$ is lower
    triangular with ones on the diagonal, hence invertible.

    Evaluating the expansion at $\bx=\bbeta$ yields
    \begin{equation*}
        Q_{\bgamma}(\bbeta)=\sum_{\bdelta\le \bgamma} L_{\bgamma,\bdelta}\,\bbeta^{2\bdelta}.
    \end{equation*}
    Stacking these identities for all $\bgamma,\bbeta$ gives the matrix
    relation
    \begin{equation*}
        E = L M.
    \end{equation*}
    Since $E$ and $L$ are invertible, then $M$ is invertible. Therefore
    $A=MD$ is invertible as well.
\end{proof}

\section{Proofs for right-hand side of moment-matching system}

We first introduce a lemma to verify the holomorphicity of an
integral with respect to a complex parameter.

\begin{lemma}[Holomorphic parameter integral] \label{lem:holomorphic}
    Let $U\subset \bbR^d$ be measurable and let $\Omega\subset \bbC$ be
    open. Suppose $F:\Omega\times U\to \bbC$ is jointly measurable and
    satisfies:
    \begin{enumerate}
        \item[(1)] For each fixed $\bx\in U$, the map $z\mapsto F(z,\bx)$
        is holomorphic on $\Omega$.
        \item[(2)] For every compact set $K\subset \Omega$, there exists
        $M_K\in L^1(U)$ such that $|F(z,\bx)|\le M_K(\bx)$ for all
        $z\in K$ and $\bx\in U$.
    \end{enumerate}
    Define
    \begin{equation*}
        f(z) := \int_U F(z,\bx)\,\rd \bx, \qquad z\in \Omega.
    \end{equation*}
    Then $f$ is holomorphic on $\Omega$.
\end{lemma}

\begin{proof}
    By dominated convergence, $f$ is continuous on $\Omega$. Let $T$ be
    any triangle whose closure $\overline{T}$ is contained in $\Omega$.
    By assumption (2), there exists $M\in L^1(U)$ such that
    $|F(z,\bx)|\leq M(\bx)$ for $z\in\overline T$. Hence Fubini's
    theorem and Cauchy's theorem give
    \begin{equation*}
        \int_{\partial T}f(z)\,\rd z
        =\int_U\left(\int_{\partial T}F(z,\bx)\,\rd z\right)\rd\bx
        =0.
    \end{equation*}
    Thus $f$ is holomorphic by Morera's theorem.
\end{proof}

\begin{proof}[Proof of \cref{lem:holom_bs_D1}]
    Recall that in the proof of \cref{lem:error_order_rhs} we obtained
    \begin{align*}
        b_{\balpha}(s)
        &= \lim_{N\rightarrow \infty} 
            \left( \int_{\bbR^d} g(\bx \frac{a}{N}) \frac{\bx^{2\balpha}}{|\bx|^s} \rd \bx 
            - \sum_{\infnorm{\bi}=1}^{\infty} g(\bi \frac{a}{N}) \frac{\bi^{2\balpha}}{|\bi|^s} 
            \right) \\
        &= \int_{|\bx|\le 1}\frac{\bx^{2\balpha}}{|\bx|^s}(1-\eta(\bx))\,\rd\bx
        - \sum_{\infnorm{\bk}=1}^{\infty}\frac{W(\bk;s)}{(2\pi k_j)^q}
        =: b_{\balpha}^{(1)}(s) - b_{\balpha}^{(2)}(s),
        \qquad s\in D_1,
    \end{align*}
    where $j=j(\bk)\in \text{argmax}_{\ell=1, \ldots, d} |k_{\ell}|$ and
    $q = 2p+3+d+2\onenorm{\balpha}$ are chosen as in the proof of
    \cref{lem:error_order_rhs}. We show that both $b_{\balpha}^{(1)}$
    and $b_{\balpha}^{(2)}$ are holomorphic on 
    $D_1 = \left\{ s \in \bbC\, : \, 0 < \rep(s) < d+ 2\onenorm{\balpha} \right\}$.

    For the first term $b_{\balpha}^{(1)}(s)$, we split the integral as
    \begin{equation*}
        b_{\balpha}^{(1)}(s)
        = \int_{\frac12\le |\bx|\le 1}\frac{\bx^{2\balpha}}{|\bx|^s}(1-\eta(\bx))\,\rd\bx
        + \int_{|\bx|\le \frac12}\frac{\bx^{2\balpha}}{|\bx|^s}\,\rd\bx
        =: J_1(s)+J_2(s).
    \end{equation*}
    For $J_1$, the integrand is holomorphic in $s$ for each $\bx$ in the
    integration domain, and is uniformly bounded on compact subsets of
    $D_1$. Thus, by \cref{lem:holomorphic}, $J_1$ is holomorphic on $D_1$.
    For $J_2$, using polar coordinates $\bx=r\omega$ with $\omega\in\bbS^{d-1}$,
    \begin{align*}
        J_2(s)
        = \int_{\bbS^{d-1}}\omega^{2\balpha}\,\rd S(\omega)\int_0^{1/2} r^{2\onenorm{\balpha}-s+d-1}\,\rd r
        = \left(\int_{\bbS^{d-1}}\omega^{2\balpha}\,\rd S(\omega)\right)
        \frac{2^{\,s-d-2\onenorm{\balpha}}}{d+2\onenorm{\balpha}-s}.
    \end{align*}
    Since $D_1$ excludes $s=d+2\onenorm{\balpha}$, the right-hand side is
    holomorphic on $D_1$. Hence $b_{\balpha}^{(1)}$ is holomorphic on $D_1$.

    For the second term $b_{\balpha}^{(2)}$, fix an arbitrary compact
    set $K\subset D_1$ and set $\sigma_K:=\min_{s\in K}\rep(s)>0$. For
    each fixed $\bk\neq0$,
    \begin{align*}
        W(\bk;s)
        &= (-\mi)^q \int_{\bbR^d} \partial_j^q\!\left(\frac{\bx^{2\balpha}}{|\bx|^s}\eta(\bx)\right)
        \re^{-2\pi \mi\, \bk\cdot \bx}\,\rd\bx \\
        &=
        (-\mi)^q \int_{\frac12\le |\bx|\le 1}\partial_j^q\!\left(\frac{\bx^{2\balpha}}{|\bx|^s}\eta(\bx)\right)
        \re^{-2\pi \mi\,\bk\cdot \bx}\,\rd\bx
        + (-\mi)^q \int_{|\bx|\geq 1}\partial_j^q\!\left(\frac{\bx^{2\balpha}}{|\bx|^s}\right)
        \re^{-2\pi \mi\,\bk\cdot \bx}\,\rd\bx.
    \end{align*}
    The integrand is holomorphic in $s$ for each fixed $\bx$. On the
    annulus $\{1/2\le|\bx|\le1\}$, it is bounded uniformly for $s\in K$.
    On $\{|\bx|\ge1\}$, \cref{lem:drv_x_alpha_s} gives, uniformly for
    $s\in K$,
    \begin{equation*}
        \left|\partial_j^q\!\left(\frac{\bx^{2\balpha}}{|\bx|^s}\right)\right|
        \le C_K|\bx|^{-q-\rep(s)+2\onenorm{\balpha}}
        \le C_K|\bx|^{-q-\sigma_K+2\onenorm{\balpha}}.
    \end{equation*}
    This majorant is integrable because
    $q+\sigma_K>d+2\onenorm{\balpha}$. Since $K\subset D_1$ was
    arbitrary, \cref{lem:holomorphic} shows that $W(\bk;\cdot)$ is
    holomorphic on $D_1$. Moreover, the estimates for the integrands in
    $W(\bk;s)$ are uniform in $\bk$. Thus,
    \begin{equation*}
        \sup_{s\in K}\left|\frac{W(\bk;s)}{(2\pi k_j)^q}\right|
        \leq \frac{C_K}{|k_j|^q},
        \qquad \bk\in\bbZ^d\setminus\{0\},
    \end{equation*}
    where $C_K$ is a constant independent of $\bk$. Since
    $|k_j|=\infnorm{\bk}$, \cref{lem:linfty_shell_count} and $q>d$ imply
    \begin{equation*}
        \sum_{\infnorm{\bk}=1}^{\infty}|k_j|^{-q}
        \leq C_d\sum_{n=1}^{\infty}n^{d-1-q}<\infty.
    \end{equation*}
    The Weierstrass M-test therefore gives uniform convergence of the
    series defining $b_{\balpha}^{(2)}$ on every compact $K\subset D_1$.
    By the Weierstrass theorem, $b_{\balpha}^{(2)}$ is holomorphic on
    $D_1$.
\end{proof}

Next we give two lemmas analyzing the regular and singular parts of
$b_{\balpha}(s)$ respectively.

\begin{lemma} \label{lem:bs_regular_part}
    Let $p \in \bbN$. Assume that $g\in\ccinf(\bbR^d)$ satisfies
    $\supp(g)\subset[-a,a]^d$ and
    \begin{equation*}
        g(\bx) = g(|\bx|), \qquad \partial^{\bbeta} g(0) = 0, \quad
        \text{for } 1 \leq \onenorm{\bbeta} \leq 2p+1.
    \end{equation*}
    Given multi-index $\balpha \in \bbN^d$ such that $0 \leq
    \onenorm{\balpha} \leq p$, define for $N\in\bbN_+$
    \begin{equation*}
        I_N(s)
        := \int_{\frac{1}{2} \leq \infnorm{\bx} \leq N + \frac{1}{2}}
        g\!\left(\bx \frac{a}{N}\right)\frac{\bx^{2\balpha}}{|\bx|^s}\,\rd \bx
        - \sum_{\infnorm{\bi}=1}^{N} g\!\left(\bi \frac{a}{N}\right)\frac{\bi^{2\balpha}}{|\bi|^s}.
    \end{equation*}
    Then for every $s$ with $\rep(s) > d+2\onenorm{\balpha}-2$ the limit
    \begin{equation*}
        I(s) := \lim_{N\rightarrow\infty} I_N(s)
    \end{equation*}
    exists and defines a holomorphic function on $D :=
    \{s\in\bbC:\rep(s) > d+2\onenorm{\balpha}-2\}$.
\end{lemma}

\begin{proof}
    We first show $\{I_N\}$ is locally bounded on domain $D$ via a shell
    decomposition and uniform Taylor estimates. This local boundedness,
    together with pointwise convergence on $\rep(s) >
    d+2\onenorm{\balpha}$, implies the limit $I(s)$ exists and is
    holomorphic on $D$ by Vitali-Porter theorem \cite{schiff1993normal}.

    \paragraph{Shell decomposition and local Taylor expansion.}
    For each $N$, decompose the integral into shells and shift each
    shell to the cube $\{\infnorm{\bx}\le 1/2\}$:
    \begin{align*}
        \int_{\frac{1}{2} \leq \infnorm{\bx} \leq N + \frac{1}{2}}
        g\!\left(\bx \frac{a}{N}\right)\frac{\bx^{2\balpha}}{|\bx|^s}\,\rd \bx
        &=
        \sum_{n=1}^N \int_{n-\frac{1}{2} \leq \infnorm{\bx} \leq n + \frac{1}{2}}
        g\!\left(\bx \frac{a}{N}\right)\frac{\bx^{2\balpha}}{|\bx|^s}\,\rd \bx \\
        &=
        \sum_{n=1}^N \sum_{\infnorm{\bi}=n}
        \int_{\infnorm{\bx}\le \frac{1}{2}}
        g\!\left((\bx+\bi)\frac{a}{N}\right)\frac{(\bx+\bi)^{2\balpha}}{|\bx+\bi|^s}\,\rd \bx.
    \end{align*}
    Hence we can write
    \begin{equation*}
        I_N(s) = \sum_{n=1}^N \delta_{n,N}(s),
    \end{equation*}
    where
    \begin{equation*}
        \delta_{n,N}(s)
        := \sum_{\infnorm{\bi}=n}
        \int_{\infnorm{\bx}\le \frac{1}{2}}
        \left[
        g\!\left((\bx+\bi)\frac{a}{N}\right)\frac{(\bx+\bi)^{2\balpha}}{|\bx+\bi|^s}
        - g\!\left(\bi\frac{a}{N}\right)\frac{\bi^{2\balpha}}{|\bi|^s}
        \right]\rd \bx.
    \end{equation*}
    For any fixed $N$, define
    \begin{equation*}
        f_{N}(\bx;s)
        := g\!\left(\bx\frac{a}{N}\right) \frac{\bx^{2\balpha}}{|\bx|^s}. 
    \end{equation*}
    Taylor's theorem with integral remainder gives, for
    $\infnorm{\bx}\le 1/2$,
    \begin{equation} \label{eq:taylor_bs_regular}
        f_{N}(\bx + \bi;s) - f_{N}(\bi;s)
        = \sum_{j=1}^d \partial_{x_j} f_{N}(\bi;s)\,x_j
        + \sum_{j,k=1}^d x_jx_k\int_0^1(1-t)
        \partial_{x_jx_k}^2f_N(\bi+t\bx;s)\,\rd t.
    \end{equation}

    \paragraph{Local boundedness of $\{I_N\}$ on $D$.}
    Now we show that $\{I_N\}$ is uniformly bounded on any compact
    subset of $D$. Fix an arbitrary compact set $K\subset D$ and set
    \begin{equation*}
        \sigma_K:=\min_{s\in K}\rep(s)>d+2\onenorm{\balpha}-2.
    \end{equation*}
    According to \cref{lem:drv_x_alpha_s}, for $j, k=1, \ldots, d$, we
    have
    \begin{equation*}
        \frac{\partial f_N}{\partial x_j}(\bx;s)
        = (\partial_j g)\!\left(\bx \frac{a}{N}\right) \frac{a}{N} \cdot \frac{\bx^{2\balpha}}{|\bx|^s} 
            + g\!\left(\bx \frac{a}{N}\right) \cdot \frac{1}{|\bx|^{1+s-2\onenorm{\balpha}}} P_{1, \balpha,j}\!\left(\frac{\bx}{|\bx|} \right),
    \end{equation*}
    and
    \begin{align*}
        \frac{\partial^2 f_N}{\partial x_j\partial x_k}(\bx;s)
        &= (\partial_{jk}^2 g)\!\left(\bx \frac{a}{N}\right) (\frac{a}{N})^2 \cdot \frac{\bx^{2\balpha}}{|\bx|^s} \\
        &\quad + \frac{a}{N} \frac{1}{|\bx|^{1+s-2\onenorm{\balpha}}} 
            \left[ (\partial_j g)\!\left(\bx \frac{a}{N}\right) P_{1, \balpha,k}\!\left(\frac{\bx}{|\bx|}\right) 
                + (\partial_k g)\!\left(\bx \frac{a}{N}\right) P_{1, \balpha,j}\!\left(\frac{\bx}{|\bx|}\right) \right] \\
        &\quad - g\!\left(\bx \frac{a}{N}\right) \cdot \frac{1+s-2\onenorm{\balpha}}{|\bx|^{2+s-2\onenorm{\balpha}}} 
            \frac{x_k}{|\bx|} \cdot P_{1, \balpha,j}\!\left(\frac{\bx}{|\bx|}\right) \\
        &\quad + g\!\left(\bx \frac{a}{N}\right) \cdot \frac{1}{|\bx|^{2+s-2\onenorm{\balpha}}}
        \left[ (\partial_kP_{1, \balpha,j})\!\left(\frac{\bx}{|\bx|}\right) (1 - \frac{x_k^2}{|\bx|^2}) 
            - \sum_{\ell\neq k} (\partial_\ell P_{1, \balpha,j})\!\left(\frac{\bx}{|\bx|}\right) \frac{x_{\ell} x_k}{|\bx|^2} \right],
    \end{align*}
    where $P_{1, \balpha,j}(\bx)$ is a polynomial of $\bx$.
    
    Given $\infnorm{\bi} \leq N$, let $n=\infnorm{\bi}$, then for
    $\infnorm{\bx}\le 1/2$, using $n \leq N$ it follows that
    \begin{equation*}
        \max_{j,k}\sup_{\infnorm{\bx}\le 1/2}
        \left|\partial^2_{x_jx_k} f_N(\bx+\bi;s)\right|
        \le C_K\, n^{2\onenorm{\balpha}-\rep(s)-2}
        \le C_K\, n^{2\onenorm{\balpha}-\sigma_K-2}.
    \end{equation*}
    Since $x_j$ is odd and the domain is centrally symmetric, the first
    order term in \eqref{eq:taylor_bs_regular} integrates to $0$ over
    $\{\infnorm{\bx}\le 1/2\}$. Consequently, the Hessian estimate above
    gives
    \begin{equation*}
        \left|\int_{\infnorm{\bx}\le \frac12}
        \left[f_N(\bi+\bx;s)-f_N(\bi;s)\right]\,\rd\bx\right|
        \le C_K\, n^{2\onenorm{\balpha}-\sigma_K-2}.
    \end{equation*}
    Summing over the shell $\{\bi\in\bbZ^d:\infnorm{\bi}=n\}$ and using
    $\#\{\bi:\infnorm{\bi}=n\}\le C_d n^{d-1}$ by
    \cref{lem:linfty_shell_count} gives
    \begin{equation} \label{eq:delta_nN_bound}
        \sup_{s\in K}|\delta_{n,N}(s)|
        \leq C_K\, n^{d+2\onenorm{\balpha}-\sigma_K-3},
        \qquad 1\leq n\leq N.
    \end{equation}
    Combining $I_N(s)=\sum_{n=1}^N\delta_{n,N}(s)$ with
    \eqref{eq:delta_nN_bound} gives
    \begin{equation*}
        \sup_{N\in\bbN_+}\sup_{s\in K}|I_N(s)|
        \leq \sum_{n=1}^{\infty} \sup_{N\geq n}\sup_{s\in K}|\delta_{n,N}(s)| 
        \leq C_K\sum_{n=1}^{\infty} n^{d+2\onenorm{\balpha}-\sigma_K-3}<\infty,
    \end{equation*}
    where the last series converges since
    $\sigma_K>d+2\onenorm{\balpha}-2$.

    \paragraph{Holomorphicity of the limit.}
    For each fixed $N$, the map $s\mapsto I_N(s)$ is holomorphic on $D$
    by \cref{lem:holomorphic}. If $\rep(s)>d+2\onenorm{\balpha}$, then
    $|\bx|^{2\onenorm{\balpha}-\rep(s)}$ is integrable on
    $\{\infnorm{\bx}\geq1/2\}$ and summable over $\bbZ^d\setminus\{0\}$.
    Since $\supp(g)\subset[-a,a]^d$, we have $g(a\bi/N)=0$ whenever
    $\infnorm{\bi}>N$. Note that $g$ is bounded and $g(a\bx/N)\to g(0)$
    pointwise, dominated convergence with respect to Lebesgue and
    counting measure gives
    \begin{equation*}
        \lim_{N\to\infty}I_N(s)
        =g(0)\left[
        \int_{\infnorm{\bx}\geq1/2}\frac{\bx^{2\balpha}}{|\bx|^s}\,\rd\bx
        -\sum_{\infnorm{\bi}=1}^{\infty}
        \frac{\bi^{2\balpha}}{|\bi|^s}\right].
    \end{equation*}
    Thus $I_N$ converges pointwise on the nonempty open subset
    $\{\rep(s)>d+2\onenorm{\balpha}\}$ of $D$. Since $\{I_N\}$ is
    locally bounded on $D$, Vitali-Porter theorem implies that $I_N\to
    I$ uniformly on compact subsets of $D$, and the limit $I(s)$ is
    holomorphic on $D$.
\end{proof}

\begin{lemma} \label{lem:bs_singular_part}
    Let $p \in \bbN$. Assume that $g\in\ccinf(\bbR^d)$ satisfies
    $\supp(g)\subset[-a,a]^d$ and
    \begin{equation*}
        g(\bx) = g(|\bx|), \quad \partial^{\bbeta} g(0) = 0, \quad 
        \text{for } 1 \leq \onenorm{\bbeta} \leq 2p+1.
    \end{equation*}
    Given multi-index $\balpha \in \bbN^d$ such that $0 \leq
    \onenorm{\balpha} \leq p$, denote 
    \begin{align*}
        D_1 &= \left\{ s \in \bbC \, : \, 0 < \rep(s) < d+ 2\onenorm{\balpha} \right\}, \\
        D_2 &= \left\{ s \in \bbC \, : \, d + 2\onenorm{\balpha} \leq \rep(s) < d+2\onenorm{\balpha} + 1, s \neq d+2\onenorm{\balpha} \right\},
    \end{align*}
    whose union $\Omega_{\balpha}:=D_1\cup D_2$ is connected and open.
    Then for $h > 0$ and $t >0$, 
    \begin{multline} \label{eq:analytic_continuation_bs_singular}
        \aco{D_1}{D_2} \left[ \int_{\infnorm{\bx} \leq t} g(\bx h) \frac{\bx^{2\balpha}}{|\bx|^s} \,\rd \bx \right](s) 
        = \int_{|\bx| \geq t, \, \infnorm{\bx} \leq t} g(\bx h) \frac{\bx^{2\balpha}}{|\bx|^s} \,\rd \bx \\
        + \int_{\bbS^{d-1}} \omega^{2\balpha} \rd S(\omega) 
            \left[ \int_{0}^{t} \left(g(rh) - g(0)\right) r^{d+2\onenorm{\balpha}-1-s} \,\rd r 
                + g(0)\frac{t^{d+2\onenorm{\balpha} - s}}{d+2\onenorm{\balpha} - s} \right].
    \end{multline}
    Furthermore, the limit $h\to 0^+$ commutes with analytic
    continuation in the sense that
    \begin{equation*}
        \lim_{h\rightarrow 0^+} \aco{D_1}{D_2} 
        \left[ \int_{\infnorm{\bx} \leq t} g(\bx h) \frac{\bx^{2\balpha}}{|\bx|^s} \,\rd \bx \right](s) 
        = \aco{D_1}{D_2} \left[ \lim_{h\rightarrow 0^+} 
        \int_{\infnorm{\bx} \leq t} g(\bx h) \frac{\bx^{2\balpha}}{|\bx|^s} \,\rd \bx \right](s).
    \end{equation*}
\end{lemma}

\begin{proof}
    We begin by splitting the cube into the ball plus the remaining
    corner region:
    \begin{equation*}
        \int_{\infnorm{\bx} \leq t} g(\bx h) \frac{\bx^{2\balpha}}{|\bx|^s} \,\rd \bx 
        = \int_{|\bx| \leq t} g(\bx h) \frac{\bx^{2\balpha}}{|\bx|^s} \,\rd \bx 
        + \int_{|\bx| \geq t, \, \infnorm{\bx} \leq t} g(\bx h) \frac{\bx^{2\balpha}}{|\bx|^s} \,\rd \bx,
    \end{equation*}
    where the second term is an entire function in $s$ by
    \cref{lem:holomorphic}. For the first term, using polar coordinates
    $x=r\omega$ with $\omega \in \bbS^{d-1}$ and radiality
    $g(\bx)=g(|\bx|)$ we have
    \begin{align*}
        \int_{|\bx|\le t} g(\bx h)\,\frac{\bx^{2\balpha}}{|\bx|^s}\,\rd\bx
        = \int_{\bbS^{d-1}} \omega^{2\balpha}\,\rd S(\omega)
        \int_{0}^{t} g(rh)\,r^{d+2\onenorm{\balpha}-1-s}\,\rd r.
    \end{align*}
    Thus, the integral converges and is analytic for $\rep(s) <
    d+2\onenorm{\balpha}$. We then rewrite the radial integral as
    \begin{align} \label{eq:integral_gr}
        \int_{0}^{t} g(rh) r^{d+2\onenorm{\balpha}-1-s} \,\rd r 
        = \int_{0}^{t} \bigl(g(rh) - g(0)\bigr) r^{d+2\onenorm{\balpha}-1-s} \,\rd r 
        + g(0) \frac{t^{d+2\onenorm{\balpha} - s}}{d+2\onenorm{\balpha} - s}. 
    \end{align}
    The second term in \eqref{eq:integral_gr} is meromorphic in $s$ with
    a simple pole at $s=d+2\onenorm{\balpha}$, and hence is holomorphic
    on $\Omega_{\balpha}$. For the first term, Taylor's theorem gives
    \begin{equation*}
        | g(rh) - g(0) | \leq C (rh)^{2p+2}. 
    \end{equation*}
    If $K\subset\Omega_{\balpha}$ is compact and $\sigma_K:=\max_{s\in
    K}\rep(s)$, then the integrand in the first term of
    \eqref{eq:integral_gr} is bounded near $r=0$ by $C_K
    r^{d+2\onenorm{\balpha}+2p+1-\sigma_K}$, which is integrable because
    $\sigma_K<d+2\onenorm{\balpha}+1$. Thus \cref{lem:holomorphic} shows
    that the first term is holomorphic on $\Omega_{\balpha}$ as well.
    Consequently, the right-hand side of
    \eqref{eq:analytic_continuation_bs_singular} defines a holomorphic
    function on $\Omega_{\balpha}$ that agrees with the original
    integral on $D_1$. Its restriction to $D_2$ is therefore the stated
    analytic continuation.

    Next we verify the interchange of orders of limit with $h$ and
    analytic continuation operator. By applying Lebesgue's dominated
    convergence theorem to the right-hand side of
    \eqref{eq:analytic_continuation_bs_singular}, we have
    \begin{align*}
        & \quad \lim_{h\rightarrow 0^+} \aco{D_1}{D_2} \left[ \int_{\infnorm{\bx} \leq t} g(\bx h) \frac{\bx^{2\balpha}}{|\bx|^s} \,\rd \bx \right](s) \\
        &= \int_{|\bx| \geq t, \, \infnorm{\bx} \leq t} g(0) \frac{\bx^{2\balpha}}{|\bx|^s} \,\rd \bx 
        + \int_{\bbS^{d-1}} \omega^{2\balpha} \rd S(\omega) 
            \left[ g(0)\frac{t^{d+2\onenorm{\balpha} - s}}{d+2\onenorm{\balpha} - s} \right].
    \end{align*} 
    On the other hand, for fixed $s\in D_1$, the dominated convergence
    theorem gives
    \begin{equation} \label{eq:lim_box_integral}
        \lim_{h\rightarrow 0^+} \int_{\infnorm{\bx} \leq t} g(\bx h) \frac{\bx^{2\balpha}}{|\bx|^s} \,\rd \bx 
        = \int_{\infnorm{\bx} \leq t} g(0) \frac{\bx^{2\balpha}}{|\bx|^s} \,\rd \bx.
    \end{equation}
    Applying the same ball--corner decomposition, the analytic
    continuation of \eqref{eq:lim_box_integral} to $D_2$ is
    \begin{equation*}
        \int_{|\bx| \geq t, \, \infnorm{\bx} \leq t} g(0) \frac{\bx^{2\balpha}}{|\bx|^s} \,\rd \bx 
        + \int_{\bbS^{d-1}} \omega^{2\balpha} \rd S(\omega) \left[ g(0)\frac{t^{d+2\onenorm{\balpha} - s}}{d+2\onenorm{\balpha} - s} \right].
    \end{equation*}
    Therefore, by comparing these two expressions we complete the proof
    of the lemma.
\end{proof}

\begin{proof}[Proof of \cref{lem:exist_bs_D2}]
    Since $\supp(g)\subset[-a,a]^d$, we restrict the integral in
    $u_N^{D_1}(s)$ to $\{\infnorm{\bx} \leq N + \frac{1}{2}\}$ as
    \begin{equation*}
        u_N^{D_1}(s) 
        = \int_{\bbR^d} g(\bx \frac{a}{N}) \frac{\bx^{2\balpha}}{|\bx|^s} \,\rd \bx
        = \int_{\infnorm{\bx} \leq N + \frac{1}{2}} g(\bx \frac{a}{N}) \frac{\bx^{2\balpha}}{|\bx|^s} \,\rd \bx.
    \end{equation*}
    
    To perform analytic continuation from $D_1$ to $D_2$, we decompose
    the difference $\bigl[u_N^{D_1} - v_N^{D_1}\bigr](s)$ into regular
    and singular parts as,
    \begin{align*}
        u_N^{D_1} - v_N^{D_1} 
        &= \int_{\infnorm{\bx} \leq N + \frac{1}{2}} g(\bx \frac{a}{N}) \frac{\bx^{2\balpha}}{|\bx|^s} \,\rd \bx 
            - \sum_{\infnorm{\bi}=1}^{N} g(\bi \frac{a}{N}) \frac{\bi^{2\balpha}}{|\bi|^s} \\
        &= \left( \int_{\frac{1}{2} \leq \infnorm{\bx} \leq N + \frac{1}{2}} g(\bx \frac{a}{N}) \frac{\bx^{2\balpha}}{|\bx|^s} \,\rd \bx - \sum_{\infnorm{\bi}=1}^{N} g(\bi \frac{a}{N}) \frac{\bi^{2\balpha}}{|\bi|^s} \right) 
            + \int_{\infnorm{\bx} \leq \frac{1}{2}} g(\bx\frac{a}{N}) \frac{\bx^{2\balpha}}{|\bx|^s} \,\rd \bx \\
        &:= I_N(s) + J_N^{D_1}(s).
    \end{align*}
    For the singular part, $J_N^{D_1}(s)$ is initially defined for $s\in
    D_1$. After applying \cref{lem:bs_singular_part} with $h=a/N$ and
    $t=1/2$, $J_N^{D_1}(s)$ and $\lim_{N\to\infty} J_N^{D_1}(s)$ possess
    analytic continuation over $D_1 \cup D_2$ and we further have
    \begin{equation*}
        \lim_{N\rightarrow\infty} \aco{D_1}{D_2} \left[ J^{D_1}_N \right] (s) 
        = \aco{D_1}{D_2} \left[ \lim_{N\rightarrow\infty} J_N^{D_1} \right] (s).
    \end{equation*}
    For the regular part, $I_N(s)$ is analytic on $D_1 \cup D_2$, and by
    \cref{lem:bs_regular_part}, the limit
    $I(s):=\lim_{N\to\infty}I_N(s)$ is analytic on
    $D=\{s\in\bbC:\rep(s)>d+2\onenorm{\balpha}-2\}$.  
    Introduce the overlap strip between $D_1$ and $D$ as 
    \begin{equation*}
        D_0 := \left\{ s\in\bbC:\ d+2\onenorm{\balpha}-1<\rep(s)<d+2\onenorm{\balpha}\right\}
        \subset (D \cap D_1).
    \end{equation*}
    Then the analyticity of $I_N(s)$ and $I(s)$ over $D_0 \cup D_2$
    gives
    \begin{equation*}
        \lim_{N\rightarrow\infty} \aco{D_0}{D_2} \left[ I^{D_0}_N \right] (s) 
        = \lim_{N\rightarrow\infty} I^{D_2}_N(s) = I^{D_2}(s) = \aco{D_0}{D_2} \left[I^{D_0}\right] (s) 
        = \aco{D_0}{D_2} \left[ \lim_{N\rightarrow\infty}I_N^{D_0} \right] (s).
    \end{equation*}

    The relative positions of the four domains used in this argument are
    summarized in \cref{fig:domains_analytic_continuation}.

    \begin{figure}[htbp]
        \centering
        \begin{tikzpicture}[x=1.05cm,y=1.0cm,>=Stealth,font=\small]
            \tikzset{
                axis/.style={black!70,line width=0.55pt},
                boundary/.style={black!45,densely dashed,line width=0.7pt},
                domainarrow/.style={<->,line width=1.35pt},
                domainray/.style={<->,line width=1.35pt},
                domainlabel/.style={font=\small\bfseries,inner sep=1pt},
            }
            \def\xmin{-0.55}
            \def\xregular{1.05}
            \def\xoverlap{2.65}
            \def\xpole{4.25}
            \def\xright{6.55}
            \def\xmax{7.45}
            \def\ymin{-2.05}
            \def\ymax{2.05}

            \fill[blue!8] (0,0) rectangle (\xpole,\ymax);
            \fill[green!10] (\xpole,0) rectangle (\xright,\ymax);
            \fill[orange!10] (\xregular,\ymin) rectangle (\xmax,0);
            \fill[violet!16] (\xoverlap,\ymin) rectangle (\xpole,0);

            \draw[axis,->] (\xmin,0) -- (\xmax,0)
                node[below right] {$\rep(s)$};
            \draw[axis,->] (0,\ymin) -- (0,\ymax+0.25)
                node[above left] {$\operatorname{Im}(s)$};
            \draw[boundary] (\xregular,\ymin) -- (\xregular,\ymax);
            \draw[boundary] (\xoverlap,\ymin) -- (\xoverlap,\ymax);
            \draw[boundary] (\xpole,\ymin) -- (\xpole,\ymax);
            \draw[boundary] (\xright,\ymin) -- (\xright,\ymax);

            \draw[domainarrow,blue!65!black]
                (0,1.25) -- (\xpole,1.25)
                node[midway,above=3pt,domainlabel,text=blue!65!black]
                {$D_1$};
            \draw[domainarrow,green!45!black]
                (\xpole,1.25) -- (\xright,1.25)
                node[midway,above=3pt,domainlabel,text=green!45!black]
                {$D_2$};

            \draw[domainarrow,violet!70!black]
                (\xoverlap,-0.68) -- (\xpole,-0.68)
                node[midway,below=3pt,domainlabel,text=violet!70!black]
                {$D_0$};
            \draw[domainray,orange!80!black]
                (\xregular,-1.35) -- (\xmax,-1.35)
                node[midway,below=3pt,domainlabel,text=orange!80!black]
                {$D$};

            \node[blue!65!black,below left=2pt] at (0,0) {$0$};
            \node[orange!80!black,below=2pt] at (\xregular,0)
                {$s_{\balpha}-2$};
            \node[violet!70!black,below=2pt] at (\xoverlap,0)
                {$s_{\balpha}-1$};
            \node[below=2pt] at (\xpole,0) {$s_{\balpha}$};
            \node[green!45!black,above=2pt] at (\xright,0)
                {$s_{\balpha}+1$};

            \fill[white] (\xpole,0) circle (2.5pt);
            \draw[red!75!black,line width=0.9pt] (\xpole,0) circle (2.5pt);

        \end{tikzpicture}
        \caption{Schematic domains used in the proof of
        \cref{lem:exist_bs_D2}. Here $s_{\balpha}=d+2\onenorm{\balpha}$,
        $D_1=\{0<\rep(s)<s_{\balpha}\}$,
        $D_2=\{s_{\balpha}\leq\rep(s)<s_{\balpha}+1\}\setminus\{s_{\balpha}\}$,
        $D_0=\{s_{\balpha}-1<\rep(s)<s_{\balpha}\}$, and
        $D=\{\rep(s)>s_{\balpha}-2\}$. Horizontal lengths are schematic.
        To keep overlapping colors distinguishable, the shading is split
        across the real axis; each domain itself is a full vertical
        strip extending into both half-planes.}
        \label{fig:domains_analytic_continuation}
    \end{figure}
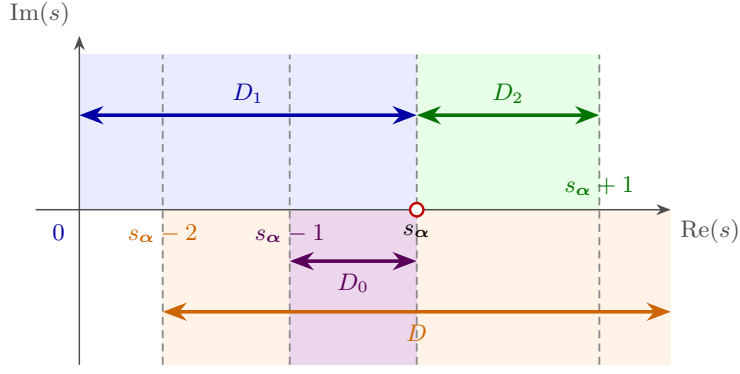

    Combine the two results, since $D_0 \subset D_1$,
    $I_N^{D_0}(s)+J_N^{D_0}(s)$ and $\lim_{N\to\infty}
    I_N^{D_0}(s)+J_N^{D_0}(s)$ possess analytic continuation over $D_0
    \cup D_2$. Finally, by the identity theorem of analytic function on
    connected open sets, we have
    \begin{align*}
        \aco{D_1}{D_2} \left[ \lim_{N\rightarrow \infty} \left( u_N^{D_1} - v_N^{D_1} \right) \right] (s) 
        &= \aco{D_0}{D_2} \left[ \lim_{N\rightarrow \infty} \left( I_N^{D_0} + J_N^{D_0} \right) \right] (s) 
        = \lim_{N\rightarrow\infty} \aco{D_0}{D_2} \left[ I^{D_0}_N + J^{D_0}_N \right] (s) \\
        &= \lim_{N\rightarrow \infty} \left( \aco{D_1}{D_2} \left[u_N^{D_1} - v_N^{D_1}\right](s) \right).
    \end{align*}
\end{proof}

\begin{proof}[Proof of \cref{lem:expression_bs_D3}]
    Since $g$ is compactly supported in $[-a,a]^d$, by changing
    variables $\by=(a/N)\bx$ we have
    \begin{equation*}
        u_N^{D_1}(s)
        = \int_{\infnorm{\bx}\le N} g\!\left(\bx\frac{a}{N}\right)\frac{\bx^{2\balpha}}{|\bx|^s}\,\rd\bx
        = \left(\frac{N}{a}\right)^{d+2\onenorm{\balpha}-s}
        \int_{\infnorm{\by}\le a} g(\by)\frac{\by^{2\balpha}}{|\by|^s}\,\rd\by.    
    \end{equation*}
    The prefactor is entire in $s$, and by \cref{lem:bs_singular_part}
    the analytic continuation of the $\by$-integral to $D_3$ exists and
    does not depend on $N$. Since $s\in D_3$ implies
    $\rep(s)>d+2\onenorm{\balpha}$, we obtain
    \begin{equation*}
        \lim_{N\to\infty}\aco{D_1}{D_3}[u_N^{D_1}](s)
        = \lim_{N\rightarrow\infty} \left(\frac{N}{a}\right)^{d+2\onenorm{\balpha}-s} \aco{D_1}{D_3}
        \left[\int_{\infnorm{\by} \leq a} g(\by) \frac{\by^{2\balpha}}{|\by|^s} \rd \by\right](s) 
        = 0.
    \end{equation*}
    For $v_N^{D_1}(s)=\sum_{\infnorm{\bi}=1}^{N} g(\bi
    a/N)\,\bi^{2\balpha}/|\bi|^s$, analytic continuation does not change
    it because it is a finite sum. Moreover, for $s\in D_3$ the series
    $\sum_{\infnorm{\bi}=1}^{\infty} |\bi|^{2\onenorm{\balpha}-\rep(s)}$
    converges, so dominated convergence yields
    \begin{equation*}
        \lim_{N\to\infty}\aco{D_1}{D_3}[v_N^{D_1}](s)
        = \lim_{N\rightarrow\infty} \sum_{\infnorm{\bi}=1}^{N} g(\bi \frac{a}{N}) \frac{\bi^{2\balpha}}{|\bi|^s}
        = \lim_{N\rightarrow\infty} \sum_{\infnorm{\bi}=1}^{\infty} g(\bi \frac{a}{N}) \frac{\bi^{2\balpha}}{|\bi|^s}
        = \sum_{\infnorm{\bi}=1}^{\infty}\frac{\bi^{2\balpha}}{|\bi|^s},
    \end{equation*}
    since $g(\bi a/N)\to g(0)=1$ for each fixed $\bi$. Combining the two
    limits gives
    \begin{equation*}
        b_{\balpha}^{D_3}(s)
        = \lim_{N\to\infty}\left(\aco{D_1}{D_3}[u_N^{D_1}](s)-\aco{D_1}{D_3}[v_N^{D_1}](s)\right)
        = -\sum_{\infnorm{\bi}=1}^{\infty}\frac{\bi^{2\balpha}}{|\bi|^s},
        \qquad s\in D_3,
    \end{equation*}
    as claimed.
\end{proof}

\section{Proofs for generalized zeta function}
\label{app:generalized_zeta_function}

We first introduce the multivariate Hermite polynomials. 

\begin{definition}[Multivariate Hermite polynomials] \label{def:hermite_poly}
    Let $A\in \bbR^{d\times d}$ be symmetric, the multivariate Hermite
    polynomial $H_{\balpha}(\bx; A)$ for multi-index $\balpha \in
    \bbN^d$ is defined as
    \begin{equation}
        H_{\balpha}(\bx; A)
        := (-1)^{\onenorm{\balpha}}\re^{\bx^\top A \bx}\,\partial_{\bx}^{\balpha}\!\left(\re^{-\bx^\top A \bx}\right),
        \qquad \bx\in\bbR^d.
    \end{equation}
    In particular, $H_{\bzero}(\bx; A) = 1$ for $\balpha = \bzero$;
    $H_{\be_j}(\bx; A) = 2 (A \bx)_j$ for the standard unit vector
    $\be_j$. When $A = I_d$, we write $H_{\balpha}(\bx) =
    H_{\balpha}(\bx; I_d)$.
\end{definition}

The following recursion provides an efficient way to compute
$H_{\balpha}(\bx;A)$.

\begin{lemma}
    Given multi-index $\balpha \in \bbN^d$ and $j \in \{1, \ldots, d\}$,
    the multivariate Hermite polynomials with symmetric matrix $A$
    defined in \cref{def:hermite_poly} satisfy the following recursion
    relation,
    \begin{equation} \label{eq:hermite_recursion}
        \begin{split}
            H_{\balpha + \be_j}(\bx; A) 
            &= 2 (A \bx)_j H_{\balpha}(\bx; A) - \partial_{x_j} H_{\balpha}(\bx; A) \\
            &= 2 (A \bx)_j H_{\balpha}(\bx; A) - \sum_{k=1}^d 2 A_{jk} \alpha_k H_{\balpha - \be_k}(\bx; A),            
        \end{split}
    \end{equation}
    where we adopt the convention $H_{\bgamma}(\bx;A)\equiv 0$ if any
    component of $\bgamma$ is negative.
\end{lemma}

\begin{proof}
    By Taylor's formula and \cref{def:hermite_poly}, the generating
    function is
    \begin{equation*}
        \sum_{\balpha\in\bbN^d} H_{\balpha}(\bx;A)
        \frac{\bt^{\balpha}}{\balpha!}
        = \re^{\bx^\top A\bx}\re^{-(\bx-\bt)^\top A(\bx-\bt)}
        = \re^{2\bt^\top A\bx-\bt^\top A\bt}.
    \end{equation*}
    Differentiating with respect to $t_j$ gives
    \begin{equation*}
        \sum_{\balpha\in\bbN^d} H_{\balpha+\be_j}(\bx;A)
        \frac{\bt^{\balpha}}{\balpha!}
        = \left(2(A\bx)_j-2(A\bt)_j\right)
        \sum_{\balpha\in\bbN^d} H_{\balpha}(\bx;A)
        \frac{\bt^{\balpha}}{\balpha!}.
    \end{equation*}
    Comparing the coefficient of $\bt^{\balpha}/\balpha!$ yields
    \begin{equation*}
        H_{\balpha+\be_j}(\bx;A)
        =2(A\bx)_jH_{\balpha}(\bx;A)
        -\sum_{k=1}^d2A_{jk}\alpha_kH_{\balpha-\be_k}(\bx;A).
    \end{equation*}
    On the other hand, differentiating the same generating function with
    respect to $x_j$ gives
    \begin{equation*}
        \sum_{\balpha\in\bbN^d} \partial_{x_j}H_{\balpha}(\bx;A)
        \frac{\bt^{\balpha}}{\balpha!}
        =2(A\bt)_j
        \sum_{\balpha\in\bbN^d} H_{\balpha}(\bx;A)
        \frac{\bt^{\balpha}}{\balpha!}.
    \end{equation*}
    Comparing coefficients again gives
    $\partial_{x_j}H_{\balpha}(\bx;A)
    =\sum_{k=1}^d2A_{jk}\alpha_kH_{\balpha-\be_k}(\bx;A)$, which is
    exactly \eqref{eq:hermite_recursion}.
\end{proof}

\begin{corollary} \label{cor:high_deri_gaussian}
    Let $A\in\bbR^{d\times d}$ be symmetric, $c>0$, and
    $\balpha\in\bbN^d$. Then
    \begin{equation*}
        \partial_{\bx}^{\balpha}\,\re^{-c\,\bx^\top A\bx}
        = (-1)^{\onenorm{\balpha}} c^{\onenorm{\balpha}/2}\,
            H_{\balpha}(\sqrt{c}\,\bx;A)\,\re^{-c\,\bx^\top A\bx}.
    \end{equation*}
\end{corollary}

Now we are ready to prove \cref{thm:zeta_function}.

\begin{proof}[Proof of \cref{thm:zeta_function}]
    Firstly, note that for $\lambda > 0$ and $\rep(s) > 0$ one has
    \begin{equation*}
        \frac{1}{\lambda^s} = \frac{1}{\Gamma(s)} \int_{0}^{\infty} \re^{-\lambda x} x^{s-1} \,\rd x.
    \end{equation*}
    Then for $\rep(s) > d+2\onenorm{\balpha}$, we have
    \begin{equation*}
        Z_{2\balpha}(s) 
        = \sum_{\bn \in \bbZ^{d} \setminus \{\bzero\}} \frac{\bn^{2\balpha}}{|\bn|^s} 
        = \sum_{\infnorm{\bn}=1}^{\infty} \bn^{2\balpha} \frac{\pi^{s/2}}{\Gamma(s/2)} 
            \int_{0}^{\infty} \re^{-|\bn|^2\pi t} \, t^{\frac{s}{2} -1} \,\rd t.
    \end{equation*}
    Then we split the integral at $t=1$ and change variables $t\mapsto
    1/t$ on $(0,1)$ as 
    \begin{align*}
        \frac{\Gamma(s/2)}{\pi^{s/2}} Z_{2\balpha}(s) 
        &= \sum_{\infnorm{\bn}=1}^{\infty} \bn^{2\balpha} 
        \left( \int_{1}^{\infty} \re^{-|\bn|^2\pi t} \, t^{\frac{s}{2} -1} \,\rd t 
            + \int_{0}^{1} \re^{-|\bn|^2\pi t} \, t^{\frac{s}{2} -1} \,\rd t \right) \\
        &= \sum_{\infnorm{\bn}=1}^{\infty} \bn^{2\balpha} \int_{1}^{\infty} \re^{-|\bn|^2\pi t} \, t^{\frac{s}{2} -1} \,\rd t 
        + \int_{1}^{\infty} \sum_{\infnorm{\bn}=1}^{\infty} 
            \bn^{2\balpha} \re^{-\frac{\pi}{t} |\bn|^2} \left(\frac{1}{t}\right)^{\frac{s}{2}+1} \,\rd t
    \end{align*}
    with the interchange of summation and integral justified by the
    absolute convergence when $\rep(s) > d+2\onenorm{\balpha}$. Further,
    for the second term, by introducing the Kronecker $\delta$ notation
    we have
    \begin{align*}
        \int_{1}^{\infty} \sum_{\infnorm{\bn}=1}^{\infty} \bn^{2\balpha} \re^{-\frac{\pi}{t} |\bn|^2} \left(\frac{1}{t}\right)^{\frac{s}{2}+1} \,\rd t
        &= \int_{1}^{\infty} \sum_{\bn \in \bbZ^d} \bn^{2\balpha} \re^{-\frac{\pi}{t} |\bn|^2} \left(\frac{1}{t}\right)^{\frac{s}{2}+1} \,\rd t 
        - \delta_{0, \balpha} \int_{1}^{\infty} \left(\frac{1}{t}\right)^{\frac{s}{2}+1} \,\rd t \\
        &= \int_{1}^{\infty} \sum_{\bn \in \bbZ^d} \bn^{2\balpha} \re^{-\frac{\pi}{t} |\bn|^2} \left(\frac{1}{t}\right)^{\frac{s}{2}+1} \,\rd t 
        - \delta_{0, \balpha} \frac{2}{s} .
    \end{align*}
    Define
    \begin{equation*}
        f(\bx) := \bx^{2\balpha} \re^{-\frac{\pi}{t}|\bx|^2}, \qquad \bx \in \bbR^d.
    \end{equation*}
    Then we compute the Fourier transform of $f$ as 
    \begin{align*}
        \fto{f}(\by) 
        &= \fto{\bx^{2\balpha} \re^{-\frac{\pi}{t}|\bx|^2}} (\by) 
        = \left(\frac{\mi}{2\pi}\right)^{2\onenorm{\balpha}} \left( \partial^{2\balpha}_{\by} \fto{\re^{-\frac{\pi}{t}|\bx|^2}}(\by) \right) \\
        &= \left(\frac{\mi}{2\pi}\right)^{2\onenorm{\balpha}} \partial^{2\balpha}_{\by} \left( t^{d/2} \re^{-\pi t|\by|^2} \right) \\
        &= \left(\frac{\mi}{2\pi}\right)^{2\onenorm{\balpha}} t^{d/2} \cdot (-1)^{2\onenorm{\balpha}} (\pi t)^{\onenorm{\balpha}} H_{2\balpha}(\sqrt{\pi t} \, \by) \, \re^{-\pi t |\by|^2} \\
        &= \left(\frac{-1}{4\pi}\right)^{\onenorm{\balpha}} t^{\frac{d}{2}+\onenorm{\balpha}} \re^{-\pi t |\by|^2} H_{2\balpha}(\sqrt{\pi t} \, \by),
    \end{align*}
    where we have used the multivariate Hermite polynomials for the
    derivatives according to \cref{cor:high_deri_gaussian}. Therefore,
    apply the Poisson summation formula introduced in
    \cref{lem:poisson_summation} to $f(\bx)$, we have 
    \begin{equation*}
        \sum_{\bn \in \bbZ^d} \bn^{2\balpha} \re^{-\frac{\pi}{t} |\bn|^2} 
        = \sum_{\bk \in \bbZ^d} \left(\frac{-1}{4\pi}\right)^{\onenorm{\balpha}} 
            t^{\frac{d}{2}+\onenorm{\balpha}} \re^{-\pi t |\bk|^2} H_{2\balpha}(\sqrt{\pi t} \, \bk).
    \end{equation*}
    Thus, for $\rep(s) > d+2\onenorm{\balpha}$, 
    \begin{align*}
        & \quad \int_{1}^{\infty} \sum_{\bn \in \bbZ^d} 
            \bn^{2\balpha} \re^{-\frac{\pi}{t} |\bn|^2} \left(\frac{1}{t}\right)^{\frac{s}{2}+1} \rd t 
        = \int_{1}^{\infty} \sum_{\bk \in \bbZ^d} 
            \left(\frac{-1}{4\pi}\right)^{\onenorm{\balpha}} \re^{-\pi t |\bk|^2} H_{2\balpha}(\sqrt{\pi t} \, \bk) \, t^{\frac{d}{2}+\onenorm{\balpha}-\frac{s}{2} - 1} \rd t \\
        &= \left(\frac{-1}{4\pi}\right)^{\onenorm{\balpha}} 
            \left( \sum_{\infnorm{\bk}=1}^{\infty} \int_{1}^{\infty} \re^{-\pi t |\bk|^2} H_{2\balpha}(\sqrt{\pi t} \, \bk) \, t^{\frac{d}{2}+\onenorm{\balpha}-\frac{s}{2} - 1} \rd t 
            + H_{2\balpha}(0) \int_{1}^{\infty} t^{\frac{d}{2}+\onenorm{\balpha}-\frac{s}{2} - 1} \rd t \right) \\
        &= \left(\frac{-1}{4\pi}\right)^{\onenorm{\balpha}} 
            \left( \sum_{\infnorm{\bk}=1}^{\infty} \int_{1}^{\infty} \re^{-\pi t |\bk|^2} H_{2\balpha}(\sqrt{\pi t} \, \bk) \, t^{\frac{d}{2}+\onenorm{\balpha}-\frac{s}{2} - 1} \rd t 
            - \frac{2H_{2\balpha}(0)}{d+2\onenorm{\balpha}-s} \right).
    \end{align*}
    Hence, we obtain the functional equation as
    \begin{equation*}
        \begin{split}
            \frac{\Gamma(s/2)}{\pi^{s/2}} Z_{2\balpha}(s) 
            &= \left(-\frac{1}{4\pi}\right)^{\onenorm{\balpha}} \frac{2H_{2\balpha}(0)}{s - d - 2\onenorm{\balpha}} 
            - \frac{2\delta_{0, \balpha}}{s} 
            + \sum_{\bn \in \bbZ^{d} \setminus \{\bzero\}} \bn^{2\balpha} \int_{1}^{\infty} \re^{-|\bn|^2\pi t} \, t^{\frac{s}{2} -1} \rd t \\
            &\quad + \left(-\frac{1}{4\pi}\right)^{\onenorm{\balpha}} \sum_{\bk \in \bbZ^{d} \setminus \{\bzero\}} \int_{1}^{\infty} \re^{-|\bk|^2 \pi t} H_{2\balpha}(\sqrt{\pi t} \, \bk) \, t^{\frac{d-s}{2}+\onenorm{\balpha} - 1} \rd t.        
        \end{split}
    \end{equation*}

    We now verify that the functional equation yields the meromorphic
    continuation of $Z_{2\balpha}(s)$. The two series appearing there are
    entire. Take the first series as an example, rewrite it with the
    upper incomplete gamma functions as
    \begin{equation*}
        \sum_{\infnorm{\bn}=1}^{\infty} \bn^{2\balpha} \int_{1}^{\infty} \re^{-|\bn|^2\pi t} \, t^{\frac{s}{2} -1} \,\rd t 
        = \frac{1}{\pi^{s/2}} \sum_{\infnorm{\bn}=1}^{\infty} \left(\frac{\bn}{|\bn|}\right)^{2\balpha} 
        \frac{1}{|\bn|^{s-2\onenorm{\balpha}}} \Gamma(\frac{s}{2}, |\bn|^2 \pi). 
    \end{equation*}
    According to the tail estimate in \cref{lem:tail_estimate}, let $N$
    be the smallest integer such that $\pi N^2 \geq \max \{\rep(s)/2, (d
    + 2\onenorm{\balpha})/2 \}$, then we have
    \begin{equation*}
        \sum_{\infnorm{\bn}=N+1}^{\infty} \left(\frac{\bn}{|\bn|}\right)^{2\balpha} \frac{1}{|\bn|^{s-2\onenorm{\balpha}}} \Gamma(\frac{s}{2}, |\bn|^2 \pi) 
        \leq C N^{d+2\onenorm{\balpha}} \re^{-\pi N^2}.
    \end{equation*}
    Hence, the series is uniformly convergent in any given compact
    subset of the complex plane. By the Weierstrass theorem, this series
    is holomorphic with respect to $s$ over the complex plane. The
    second series is treated similarly: expanding the Hermite polynomial
    $H_{2\balpha}$ reduces it to a finite linear combination of series
    of the same type, each of which is entire by the same argument.

    Therefore the right-hand side of the functional equation is meromorphic,
    with possible poles only at $s=0$ and $s=d+2\onenorm{\balpha}$.
    Since the two series in the functional equation are entire, for $s$ near $0$ we have
    \begin{equation*}
        \frac{\Gamma(s/2)}{\pi^{s/2}}Z_{2\balpha}(s)
        = -\frac{2\delta_{0,\balpha}}{s}+O(1).
    \end{equation*}
    Furthermore, since $\lim_{s\to 0} |\Gamma(s)| = +\infty$, we have
    \begin{equation*}
        \lim_{s\to 0} Z_{2\balpha}(s)
        = \lim_{s\to 0} -\frac{\pi^{s/2}}{\Gamma(s/2)} \frac{\delta_{0,\balpha}}{s/2}
        = \lim_{s\to 0} -\delta_{0,\balpha}\frac{\pi^{s/2}}{\Gamma(s/2+1)}
        = -\delta_{0, \balpha}.
    \end{equation*}
    Thus $s=0$ is not a pole of $Z_{2\balpha}(s)$.
    Finally, we conclude that $Z_{2\balpha}(s)$ has a meromorphic
    continuation to the whole complex plane $\bbC$ with a simple pole at
    $s = d + 2\onenorm{\balpha}$.
\end{proof}

\begin{proof}[Proof of \cref{thm:scaled_zeta_function}]
    The proof is similar to that of \cref{thm:zeta_function} and we omit
    most details for brevity. The key step is to compute a different
    Fourier transform as
    \begin{align*}
        \fto{\bx^{\balpha} \re^{-\frac{\pi}{t} \bx^\top B \bx}} (\by) 
        = \left(\frac{-1}{4\pi}\right)^{\frac{\onenorm{\balpha}}{2}} t^{\frac{d + \onenorm{\balpha}}{2}} |B|^{-1/2} \re^{-\pi t \by^\top B^{-1} \by} H_{\balpha}(\sqrt{\pi t} \,\by; B^{-1}),
    \end{align*}
    where $H_{\balpha}(\cdot; B^{-1})$ is the multivariate Hermite
    polynomial associated with matrix $B^{-1}$. Note that here we have
    used the assumption that $\onenorm{\balpha}$ is even to ensure the
    Fourier-transform prefactor is real-valued. 
\end{proof}

Next we give the tail estimate of series in the generalized zeta
function. We highlight that these estimates are sufficient for our
purpose and they could be improved to be sharper.

\begin{lemma} \label{lem:incomp_gamma_estimate}
    Let $a\in\bbR$ and $x>0$. The upper incomplete Gamma function
    \begin{equation*}
        \Gamma(a,x):=\int_x^\infty t^{a-1}\re^{-t}\,\rd t
    \end{equation*}
    satisfies
    \begin{equation*}
        0 < \Gamma(a,x) \leq x^a\re^{-x}, 
        \quad \text{for } x\ge \max\{a,1\}.
    \end{equation*}
\end{lemma}

\begin{proof}
    Since the integrand is positive for $t\ge x>0$, we have
    $\Gamma(a,x)>0$. Make the change of variables $t=x(1+u)$ to obtain
    \begin{equation*}
        \Gamma(a,x)
        = \int_x^\infty t^{a-1}\re^{-t}\,\rd t
        = x^a\re^{-x}\int_0^\infty (1+u)^{a-1}\re^{-ux}\,\rd u.
    \end{equation*}
    When $a \leq 1$, $|(1+u)^{a-1}| \leq 1$ for $u \geq 0$, thus we have
    \begin{equation*}
        \Gamma(a, x) \leq x^a \re^{-x} \int_0^{\infty} \re^{-ux} \,\rd u 
        = x^a \re^{-x} \cdot \frac{1}{x} \leq x^a \re^{-x}, 
        \qquad x \geq 1.
    \end{equation*}
    When $a > 1$, by $1 + u \leq e^u$ we have
    \begin{equation*}
        \Gamma(a, x) \leq x^a \re^{-x} \int_0^{\infty} \re^{-u(x-a+1)} \rd u 
        = x^a \re^{-x} \cdot \frac{1}{x - a + 1} \leq x^a \re^{-x}, 
        \qquad x \geq a.
    \end{equation*}
    Combining the discussions above we obtain the desired estimate.
\end{proof}

\begin{lemma} \label{lem:tail_estimate}
    Let $a, b \in \bbR$ and let $B\in\bbR^{d\times d}$ be symmetric
    positive definite matrix with minimum and maximum eigenvalues
    denoted by $\lmin$ and $\lmax$ respectively. Assume $N \in \bbN$
    satisfies 
    \begin{equation*}
        \pi \lmin N^2 \geq \max\left\{a, 1, \frac{2a+b+d}{2}\right\}.   
    \end{equation*}
    Then
    \begin{equation*}
        \sum_{\infnorm{\bk}=N+1}^{\infty} (\bk^\top B \bk)^{b/2}\, \Gamma(a,\pi\,\bk^\top B \bk)
        \le C_d\, C_{B,b}\,(\pi \lmin)^a\,N^{2a+b+d}\,\re^{-\pi \lmin N^2},
    \end{equation*}
    where
    \begin{equation*}
        C_d := d\,3^{d-1},
        \qquad
        C_{B,b}:=
        \begin{cases}
            \lmin^{b/2}, & b\le 0,\\
            (d\lmax)^{b/2}, & b>0.
        \end{cases}
    \end{equation*}
\end{lemma}

\begin{proof}
    For $n\in\bbN_+$, let $S_n:=\{\bk\in\bbZ^d:\infnorm{\bk}=n\}$. If
    $\bk\in S_n$, then $n^2\le |\bk|^2\le dn^2$, hence by the Rayleigh
    quotient bounds
    \begin{equation*}
        \lmin n^2 \leq \lmin |\bk|^2 \leq \bk^\top B \bk \leq \lmax |\bk|^2 \leq d \lmax n^2.
    \end{equation*}
    Since $\Gamma(a, x)$ is decreased with $x$ increasing when $x> 0$,
    we have 
    \begin{equation*}
        \Gamma(a, \pi \bk^\top B \bk) \leq \Gamma(a, \pi \lmin n^2).        
    \end{equation*}
    Moreover, $x\mapsto x^{b/2}$ is decreasing for $b\le 0$ and
    increasing for $b>0$, so for $\bk\in S_n$,
    \begin{equation*}
        (\bk^\top B\bk)^{b/2}\le C_{B,b}\,n^{b}.
    \end{equation*}
    Note that \cref{lem:linfty_shell_count} gives $\#S_n \leq
    2d\,3^{d-1}n^{d-1}$, so
    \begin{align*}
        \sum_{\infnorm{\bk}=N+1}^{\infty} (\bk^\top B \bk)^{b/2}\, \Gamma(a,\pi\,\bk^\top B \bk)
        &= \sum_{n=N+1}^{\infty}\ \sum_{\bk\in S_n} (\bk^\top B \bk)^{b/2}\, \Gamma(a,\pi\,\bk^\top B \bk) \\
        &\le 2d\,3^{d-1} C_{B,b}\sum_{n=N+1}^{\infty} n^{b+d-1}\,\Gamma(a,\pi \lmin n^2).
    \end{align*}
    By \cref{lem:incomp_gamma_estimate} and the assumption $\pi \lmin
    n^2\ge \max\{a,1\}$ for all $n\ge N$, we have
    \begin{equation*}
        \Gamma(a,\pi \lmin n^2)\le (\pi \lmin n^2)^a\,\re^{-\pi \lmin n^2}.
    \end{equation*}
    Hence, with $c:=2a+b+d-1$,
    \begin{equation*}
        \sum_{n=N+1}^{\infty} n^{b+d-1}\,\Gamma(a,\pi \lmin n^2)
        \le (\pi \lmin)^a\sum_{n=N+1}^{\infty} \re^{-\pi \lmin n^2}\,n^{c}.
    \end{equation*}

    Let $f(x):=\re^{-\pi \lmin x^2}x^c$. Then $f'(x)=\re^{-\pi \lmin
    x^2}x^{c-1}(c-2\pi \lmin x^2)$. Thus, we have $f'(x)\leq 0$ for $x >
    \sqrt{\frac{|c|}{2\pi \lmin}}$. The condition $\pi \lmin N^2\ge
    (c+1)/2$ implies $f$ is decreasing on $[N,\infty)$, so
    \begin{equation*}
        \sum_{n=N+1}^{\infty} \re^{-\pi \lmin n^2} n^{c} = \sum_{n=N+1}^{\infty} f(n) 
        \leq \sum_{n=N+1} \int_{n-1}^{n} f(x) \rd x = \int_N^{\infty} \re^{-\pi \lmin x^2}x^c \,\rd x.
    \end{equation*}
    Further, by changing variables $t=\pi \lmin x^2$,
    \begin{align*}
        \int_N^{\infty} \re^{-\pi \lmin x^2}x^c \,\rd x 
        &= \frac{1}{2} \left(\frac{1}{\pi \lmin}\right)^{\frac{c+1}{2}} \Gamma\!\left(\frac{c+1}{2}, \pi \lmin N^2\right) \\
        & \leq \frac{1}{2} \left(\frac{1}{\pi \lmin}\right)^{\frac{c+1}{2}} (\pi \lmin N^2)^{\frac{c+1}{2}} \re^{-\pi \lmin N^2} 
        = \frac{1}{2} N^{c+1} \re^{-\pi \lmin N^2},
    \end{align*}
    where we have used \cref{lem:incomp_gamma_estimate} with the
    assumption $\pi \lmin N^2 \geq \max\{\frac{c+1}{2}, 1\}$. Therefore,
    we conclude that
    \begin{equation*}
        \sum_{\infnorm{\bk}=N+1}^{\infty} (\bk^\top B \bk)^{b/2} \, \Gamma(a, \pi \, \bk^\top B \bk) 
        \leq d\, 3^{d-1} \, C_{B, b} \cdot (\pi \lmin)^a \cdot N^{2a+b+d} \re^{-\pi \lmin N^2},
    \end{equation*}
    as claimed.
\end{proof}

Finally, we present the practical evaluation formula for the generalized
zeta function associated with symmetric positive definite matrix $B$ and
its tail estimate. Write the Hermite expansion
$H_{\balpha}(\bx;B^{-1})=\sum_{0\le \onenorm{\bbeta}\le \onenorm{\balpha}}c_{\bbeta}\bx^{\bbeta}$. 
Then we have
\begin{multline*}
    \frac{\Gamma(s/2)}{\pi^{s/2}} Z_{\balpha}^B(s) 
    = |B|^{-\frac{1}{2}} \left(-\frac{1}{4\pi}\right)^{\frac{\onenorm{\balpha}}{2}} \frac{2H_{\balpha}(0; B^{-1})}{s - d - \onenorm{\balpha}} 
    - \frac{2\delta_{0, \balpha}}{s}
    + \frac{1}{\pi^{s/2}} \sum_{\infnorm{\bn}=1}^{\infty} \frac{\bn^{\balpha}}{(\bn^\top B \bn)^{s/2}} \Gamma\!\left(\frac{s}{2}, \pi \, \bn^\top B \bn\right) \\
    + |B|^{-1/2} \left(-\frac{1}{4\pi}\right)^{\frac{\onenorm{\balpha}}{2}} \pi^{\frac{s-d-\onenorm{\balpha}}{2}} \sum_{0 \leq \onenorm{\bbeta} \leq \onenorm{\balpha}} c_{\bbeta} \\
    \sum_{\infnorm{\bk}=1}^{\infty} \frac{\bk^{\bbeta}}{(\bk^\top B^{-1} \bk)^{\frac{\onenorm{\bbeta} + d + \onenorm{\balpha} - s}{2}}} \Gamma\!\left(\frac{\onenorm{\bbeta} + d + \onenorm{\balpha} - s}{2}, \pi \, \bk^\top B^{-1} \bk\right).
\end{multline*}

Denote the minimum and maximum eigenvalues of $B$ as $\lmin$ and
$\lmax$, respectively. Note that
\begin{equation*}
    \lmin \bn^T \bn \leq \bn^\top B \bn, \quad \frac{1}{\lmax} \bk^\top \bk \leq \bk^\top B^{-1} \bk.
\end{equation*}
Then for $\bn \neq \bzero$ we have
\begin{align*}
    \left| \frac{\bn^{\balpha}}{(\bn^\top B \bn)^{s/2}} \Gamma\!\left(\frac{s}{2}, \pi \, \bn^\top B \bn\right) \right| 
    &\leq \frac{|\bn|^{\onenorm{\balpha}}}{(\bn^\top B \bn)^{\frac{\onenorm{\balpha}}{2}}} 
        \frac{1}{(\bn^\top B \bn)^{\frac{\rep(s)-\onenorm{\balpha}}{2}}} 
        \Gamma\!\left(\frac{\rep(s)}{2}, \pi \, \bn^\top B \bn\right) \\
    &\leq \left(\frac{1}{\lmin}\right)^{\frac{\onenorm{\balpha}}{2}} 
        (\bn^\top B \bn)^{\frac{\onenorm{\balpha} - \rep(s)}{2}} 
        \, \Gamma\!\left(\frac{\rep(s)}{2}, \pi \, \bn^\top B \bn\right).
\end{align*}
Similarly, for the second series, let $s_1 :=
\onenorm{\beta}+d+\onenorm{\balpha}-s$, then we have
\begin{align*}
    \left|
    \frac{\bk^{\bbeta}}{(\bk^\top B^{-1} \bk)^{s_1 / 2}}
    \Gamma\!\left(\frac{s_1}{2}, \pi \, \bk^\top B^{-1} \bk\right)
    \right|
    \leq (\lmax)^{\frac{\onenorm{\bbeta}}{2}}
    (\bk^\top B^{-1} \bk)^{\frac{\onenorm{\bbeta}-\rep(s_1)}{2}}
    \,\Gamma\!\left(\frac{\rep(s_1)}{2}, \pi \, \bk^\top B^{-1} \bk\right).
\end{align*}
Both estimates are now in the form required by \cref{lem:tail_estimate},
so we can estimate the truncation error for evaluating
$Z_{\balpha}^B(s)$ accordingly.

\bibliography{SinCoTrap.bib}

\end{document}